\documentclass[article,reqno,11pt]{amsart}

\usepackage{amssymb,amsfonts,amsmath,amsthm,color}
\usepackage[all,arc]{xy}
\usepackage{enumerate}
\usepackage{mathrsfs}
\usepackage[toc,page]{appendix}
\usepackage[left=3cm, right=3cm, bottom=3cm]{geometry}
\usepackage{tabularx}
\usepackage{caption}
\newtheorem{thm}{Theorem}[section]
\newtheorem*{thm*}{Theorem}
\newtheorem{cor}[thm]{Corollary}
\newtheorem{prop}[thm]{Proposition}
\newtheorem*{prop*}{Proposition}
\newtheorem{lem}[thm]{Lemma}

\newtheorem*{ntn*}{Notation}

\theoremstyle{definition}
\newtheorem{defn}[thm]{Definition}

\newtheorem{exmp}[thm]{Example}

\theoremstyle{remark}
\newtheorem{rem}[thm]{Remark}

\newcommand{\pardeg}{\operatorname{pardeg}}

\makeatletter
\let\c@equation\c@thm
\makeatother
\numberwithin{thm}{section}
\numberwithin{equation}{section}

\title{\textsc{Topological invariants of parabolic $G$-Higgs bundles}}

\author{Georgios Kydonakis, Hao Sun and Lutian Zhao}

\date{13 March 2020
}

\begin{document}
\maketitle
\begin{abstract}
For a semisimple real Lie group $G$, we study topological properties of moduli spaces of polystable parabolic $G$-Higgs bundles over a Riemann surface with a divisor of finitely many distinct points. For a split real form of a complex simple Lie group, we compute the dimension of apparent parabolic Teichm{\"u}ller components. In the case of isometry groups of classical Hermitian symmetric spaces of tube type, we provide new topological invariants for maximal parabolic $G$-Higgs bundles arising from a correspondence to orbifold Higgs bundles. Using orbifold cohomology we count the least number of connected components of moduli spaces of such objects. We further exhibit an alternative explanation of fundamental results on counting components in the absence of a parabolic structure.
\end{abstract}

\section{Introduction}

Parabolic vector bundles over Riemann surfaces with marked points were introduced by C. Seshadri in \cite{Sesh} and similar to the Narasimhan-Seshadri correspondence, there is an analogous correspondence between stable parabolic bundles and the unitary representations of the fundamental group of the punctured surface with fixed holonomy class around each puncture \cite{Meht}. Later on, C. Simpson in \cite{Simp} provided a non-abelian Hodge correspondence in the non-compact case. The analysis on the non-compact algebraic curve has to presume the appropriate growth of the harmonic metric at the punctures, a notion called by C. Simpson tameness. In a particular case, parabolic Higgs bundles are in bijection with meromorphic flat connections, whose holonomy around each puncture defines a conjugacy class of an element in the unitary group described by the weights in the parabolic structure of the bundle. These connections correspond to representations of the fundamental group of the punctured surface in the general linear group, which send a small loop around each parabolic point to an element conjugate to a unitary element. Chern classes for parabolic bundles were constructed by I. Biswas in \cite{Biswas-Chern}; one can also define Chern characters of parabolic bundles in the rational Chow groups (see \cite{Iyer-Simpson}).

In this article, we study connected components of moduli spaces of polystable parabolic $G$-Higgs bundles for a semisimple real Lie group $G$. These objects were explicitly defined in \cite{BiGaRi}, where a Hitchin-Kobayashi correspondence was also established. Moreover, P. Boalch in \cite{Boalch} provided a local Riemann-Hibert correspondence for logarithmic connections on $G$-bundles and parabolic $G$-bundles on a curve. In the case when $G$ is a split real form of a complex simple Lie group, there exists a topologically trivial connected component in the moduli space, extending N. Hitchin's classical result from the non-parabolic case \cite{Hit92}. To be more precise, we show:

\begin{thm*}{\textnormal{\textbf{\ref{401}}}}
Let $X$ be a compact Riemann surface of genus $g$ and let $D=\left\{ {{x}_{1}},\ldots {{x}_{s}} \right\}$ a divisor of $s$-many distinct points on $X$, such that $2g-2+s>0$, that is, the surface $X$ can be equipped with a metric of constant negative curvature (-4). Let $G$ be the adjoint group of the split real form of a complex simple Lie group with Cartan decomposition in the Lie algebra $\mathfrak{g}=\mathfrak{h}\oplus \mathfrak{m}$. The space of homomorphisms from the fundamental group of $X$ into $G$ with fixed conjugacy class of monodromy around the points in $D$ has a component of real dimension $2\left( g-1 \right){{\dim}_{\mathbb{R}}}G+2s\cdot \mathrm{rk}E\left( {{\mathfrak{m}}^{\mathbb{C}}} \right)$.
\end{thm*}

An important tool towards studying the topology of moduli spaces of parabolic Higgs bundles over a Riemann surface $X$ with a divisor $D$ is provided by the correspondence of these objects to \textit{orbifold Higgs bundles} over a finite Galois covering $Y$ of $X$, ramified along $D$. I. Biswas in \cite{Biswas2} provided a correspondence between a parabolic vector bundle and an orbifold vector bundle, that is, a vector bundle on a variety equipped with a finite group action together with a lift of the action of the group to the bundle. In his work, I. Biswas explicitly constructs a class of parabolic bundles using the ``Covering Lemma'' of Y. Kawamata \cite{Kawa}; the correspondence depends on the choice of the parabolic weights, whereas the Galois covering $Y$ is constructed to have the same dimension as $X$. A similar correspondence without such restrictions was provided by I. Mundet i Riera in \cite{Rier}. In the Higgs bundle case, for a finite group $\Gamma$ acting as a group of automorphisms on a smooth projective variety $Y$, such that the quotient $X = Y/\Gamma$ is also a smooth variety, I. Biswas, S. Majumder and M. Wong in \cite{BMW} provide a bijective correspondence between parabolic Higgs bundles on $X$ and $\Gamma$-Higgs bundles on $Y$.

When the parabolic weights are rational, an equivalence between parabolic bundles and holomorphic bundles over $V$-surfaces (that means 2-dimensional orbifolds) provides an effective method to study the moduli problem, developing a Yang-Mills-Higgs theory on Riemann $V$-surfaces and calculating the cohomology of the gauge group of a $V$-bundle. The correspondence between $V$-bundles and parabolic bundles was first studied by H. Boden in \cite{Boden} and by M. Furuta and B. Steer in \cite{FuSt} using similar methods. Subsequently, B. Nasatyr and B. Steer in \cite{NaSt} introduced Higgs $V$-bundles as a straightforward extension of the original approach of N. Hitchin to orbifold Riemann surfaces studying solutions of the $\mathrm{U}(2)$ - Yang-Mills-Higgs equations on orbifold Riemann surfaces and their reinterpretation as $\mathrm{SL}(2,\mathbb{C})$-representations of the orbifold fundamental group. Moreover, the Teichm\"uller space in the orbifold case was studied in \cite{NaSt}, as well as the topology of the moduli space of Higgs bundles in the orbifold situation for rank 2 bundles.

Under the correspondence between parabolic Higgs bundles and Higgs $V$-bundles, we map a parabolic Higgs bundle over a Riemann surface $X$ with trivial filtration over each puncture ${{p}_{k}}$, for $1\le k\le s$, and weight either 0 or $\frac{1}{2}$ to a Higgs $V$-bundle over a $V$-manifold $M$ with $s$ many marked points around which the isotropy group is ${{\mathbb{Z}}_{2}}$, whereas $X$ is the underlying surface of $M$. $V$-cohomology with coefficients in ${{\mathbb{Z}}_{2}}$ is now used to describe new topological invariants and thus compute the least number of connected components of moduli of maximal parabolic $G$-Higgs bundles for semisimple Lie groups $G$, when the homogeneous space ${G}/{H}$ is a Hermitian symmetric space of tube type, where $H\subset G$ is a maximal compact subgroup. Note here that maximality is provided by a general Milnor-Wood type inequality established in \cite{BiGaRi}. Calculations in orbifold cohomology provide the rank of the cohomology groups where the topological invariants of the corresponding Higgs $V$-bundles live as Stiefel-Whitney classes; we deduce the following theorems by counting all possible numbers of these invariants for parabolic $\mathrm{Sp}(2n,\mathbb{R})$-Higgs bundles with maximal degree.

\begin{thm*}{\textbf{\ref{709}}}, {\textbf{\ref{710}}}.
Let $X$ be a smooth Riemann surface of genus $g$ and let $D$ be a reduced effective divisor of $s$ many points on $X$, such that $2g-2+s>0$.
\begin{enumerate}
\item The moduli space ${{\mathsf{\mathcal{M}}}_{par}^{max}}( \mathrm{Sp}\left( 2,\mathbb{R} \right))$ of maximal polystable parabolic $\mathrm{Sp}\left( 2,\mathbb{R} \right)$-Higgs bundles over the pair $(X, D)$ has at least $2^{2g+s-1}$ connected components. 
\item The moduli space $\mathcal{M}_{par}^{max}(\mathrm{Sp}(4,\mathbb{R}))$ of maximal polystable parabolic $\mathrm{Sp}\left( 4,\mathbb{R} \right)$-Higgs bundles over the pair $(X, D)$ has at least $(2^s+1)2^{2g+s-1}+2^s(2g-3+s)$ connected components.
\item The moduli space  ${{\mathsf{\mathcal{M}}}_{par}^{max}}\left( \mathrm{Sp}\left( 2n,\mathbb{R} \right) \right)$ of maximal polystable parabolic $\mathrm{Sp}\left( 2n,\mathbb{R} \right)$-Higgs bundles over the pair $(X, D)$, for $n \ge 3$, has at least $(2^s+1)2^{2g+s-1}$ connected components.
\end{enumerate}
\end{thm*}

We subsequently study the topological invariants for the moduli space when the parabolic structure $\alpha$ is fixed, but not necessarily involving all weights equal to $\frac{1}{2}$. We introduce the notation $\mathcal{M}_{par}^{max,\alpha}(\mathrm{Sp}(2n,\mathbb{R}))$ to mean polystable parabolic $\mathrm{Sp}(2n,\mathbb{R})$-Higgs bundles, where $\alpha$ is a given parabolic structure, which is fixed and it is the same for all parabolic $\mathrm{Sp}(2n,\mathbb{R})$-Higgs bundles in $\mathcal{M}_{par}^{max,\alpha}(\mathrm{Sp}(2n,\mathbb{R}))$, this means, they have the same filtration over each $x \in D$ with the same weight $\alpha(x)$, for every $x \in D$. Note that the moduli space $\mathcal{M}_{par}^{max,\alpha}(\mathrm{Sp}(2n,\mathbb{R}))$ is a subspace of $\mathcal{M}_{par}^{max}(\mathrm{Sp}(2n,\mathbb{R}))$ considered earlier.

In this case, the monodromy around the points in the divisor needs special attention. To be more precise, the $V$-fundamental group is described by
\begin{align*}
\pi_V^1(M)=\{a_1,b_1,...,a_g,b_g,\sigma_1,...,\sigma_s \quad | \quad \sigma_1...\sigma_s[a_1,b_1]...[a_g,b_g]=1,\sigma_i^2=1, \mathrm{ for } 1 \leq i \leq s\},
\end{align*}
where $\sigma_i$ describe the monodromy around the point $x_i$. By the correspondence between line $V$-bundles and parabolic line bundles, the monodromy around $x_i$ corresponds to the weight of the corresponding parabolic line bundle over the point $x_i$. Thus fixing a parabolic structure $\alpha$ is equivalent to fixing the monodromy around $x_i$, for $1 \leq i \leq s$. However, not every parabolic structure corresponds to a well-defined element in $\mathrm{Hom}(\pi_V^1(M),\mathbb{Z}_2)$. Indeed, the relation $\sigma_1...\sigma_s[a_1,b_1]...[a_g,b_g]=1$ implies that the number of nontrivial $\sigma_i$ is even. Equivalently, if the cardinality of the set $\{x \in D \mathrm{ } | \mathrm{ } \alpha(x)=\frac{1}{2}\}$ is even, then the parabolic structure corresponds to an element in $\mathrm{Hom}(\pi_V^1(M),\mathbb{Z}_2)$, and such a parabolic structure could be a choice for the square root of $K(D)^2$. Thus we say that the parabolic structure $\alpha$ is \emph{even} (resp. \emph{odd}) if the cardinality of the set $\{x \in D \mathrm{ } | \mathrm{ } \alpha(x)=\frac{1}{2}\}$ is even (resp. odd). Recall that the divisor $D$ contains an integer number of $s$-many points. Our result which includes this extra analysis is the following. 
\begin{prop*}{\textnormal{\textbf{\ref{711}}}}
Let $X$ be a smooth Riemann surface of genus $g$ and let $D$ be a reduced effective divisor of $s$ many points on $X$, such that $2g-2+s>0$. Consider the moduli space $\mathcal{M}_{par}^{max,\alpha}(\mathrm{Sp}(2n,\mathbb{R}))$ of maximal polystable parabolic $\mathrm{Sp}(2n,\mathbb{R})$-Higgs bundles, where $\alpha$ is a given parabolic structure, which is fixed for all Higgs bundles in the moduli space, this means, the parabolic Higgs bundles have the same filtration over each $x \in D$ with the same weight $\alpha(x)$, for every $x \in D$. Then,
\begin{enumerate}
\item[$\mathrm{i}.$] If $\alpha$ is even, the moduli space $\mathcal{M}_{par}^{max,\alpha}(\mathrm{Sp}(4,\mathbb{R}))$ has at least $2^{2g+s-1}+(2g-3+s)+2^{2g}$ connected components.
\item[$\mathrm{ii}$] If $\alpha$ is odd, the moduli space $\mathcal{M}_{par}^{max,\alpha}(\mathrm{Sp}(4,\mathbb{R}))$ has at least $2^{2g+s-1}+(2g-3+s)$ connected components.
\item[$\mathrm{iii}.$] If $\alpha$ is even, the moduli space $\mathcal{M}_{par}^{max,\alpha}(\mathrm{Sp}(2,\mathbb{R}))$ has at least $2^{2g}$ connected components, and the moduli space $\mathcal{M}_{par}^{max,\alpha}(\mathrm{Sp}(2n,\mathbb{R}))$ has at least $2^{2g+s-1}+2^{2g}$ connected components.
\item[$\mathrm{iv}.$] If $\alpha$ is odd, there are no maximal polystable parabolic $\mathrm{Sp}(2,\mathbb{R})$-Higgs bundles with fixed parabolic structure $\alpha$, and the moduli space $\mathcal{M}_{par}^{max,\alpha}(\mathrm{Sp}(2n,\mathbb{R}))$ has at least $2^{2g+s-1}$ many connected components.
\end{enumerate}
\end{prop*}

Using the component count method above for the Lie group $G = \mathrm{Sp}\left( 4,\mathbb{R} \right)$ via the correspondence to orbifold Higgs bundles, we provide a minimum component count for moduli spaces of maximal polystable parabolic $G$-Higgs bundles analogously to the non-parabolic case \cite{BrGPGHermitian}. Our results are summarized in Tables 1, 2, 3 appearing at the end of the main body of this article.

The classical tool in order to provide an exact count of the number of connected components of the moduli spaces considered in Table 1 involves the analysis of a particular moment map on the moduli space, which is also a Morse-Bott function. For parabolic $G$-Higgs bundles for a real Lie group $G$, this was pioneered in the dissertation of M. Logares \cite{Loga} in the case $G = \mathrm{U} (p, q)$ following analogous methods from non-parabolic cases. This problem for groups other than $\mathrm{U} (p, q)$ is addressed in our subsequent article \cite{KSZ2}, where we develop the relevant Morse theoretic machinery to show that for $\mathrm{Sp}(2n,\mathbb{R})$ the above numbers of components are in fact the exact ones (and the same follows for the rest of the groups appearing in Table 1). Therefore, we deduce that the topological invariants introduced in this article are fine enough to distinguish the connected components of polystable parabolic $G$-Higgs moduli spaces in the cases above.

Yet, an alternative method for counting components of moduli spaces of (non-parabolic) $G$-Higgs bundles, especially for a split real form $G$, is by studying orbits of the monodromy group on the first $\mathrm{mod}2$ cohomology of the fibers of the Hitchin fibration; this was first described by L. Schaposnik in \cite{Schap} for the case $\mathrm{SL}(2,\mathbb{R})$, and subsequently the method was applied also for non split real cases as well (see \cite{BarSch1}, \cite{BarSch2}, \cite{HitSch}). These techniques in the parabolic $G$-Higgs setting are also developed in an ensuing article \cite{KSZ3}.

It is interesting at this point to compare the results of Table 1 with the analogous results in the non-parabolic case from \cite{Stru}, \cite{BrGPGHermitian}, \cite{GaGMsymplectic}, \cite{Gothen} and \cite{Hit92}, thus providing further applications of our study of the topological invariants for parabolic $G$-Higgs bundles, as well as the methods developed for finding those. In \cite{Stru}, T. Strubel using Fenchel-Nielsen coordinates showed that the moduli space ${{\mathsf{\mathcal{R}}}^{\max }}\left( {{\Sigma }_{g,m}},\mathrm{Sp}\left( 2n,\mathbb{R} \right) \right)$ of maximal representations of the fundamental group of a topological surface ${{\Sigma }_{g,m}}$ of genus $g$ and $m\ge 1$ boundary components into $\mathrm{Sp}\left( 2n,\mathbb{R} \right)$, has exactly ${{2}^{2g+m-1}}$ connected components for every $n\ge 1$. We explain how one can use the method involving the $V$-manifold correspondence to obtain an alternative description of T. Strubel's result. Furthermore, we exhibit maximal non-parabolic $G$-Higgs bundles as $V$-bundles equipped with a trivial action, an interpretation which leads to an explanation of the component counts established by S. Bradlow, O. Garc\'{i}a-Prada, P. Gothen and I. Mundet i Riera, and are summarized in \cite{BrGPGHermitian}, as special cases of our parabolic case component count, when there is only one puncture considered.

This article involves the study of topological properties of moduli of polystable $G$-Higgs bundles equipped with a parabolic structure, building on the general parabolic $G$-Higgs definitions from \cite{BiGaRi}. Section 2 contains the basic definitions for parabolic $\mathrm{GL}\left(n,\mathbb{C} \right)$-Higgs bundles, while the more general definitions for any semisimple Lie group $G$ are put in an Appendix. In this Appendix we also work out particular examples to demonstrate the relation between the general stability condition for any group $G$ and the classical one for $\mathrm{GL}\left(n,\mathbb{C}\right)$. In Section 3 we adapt the deformation theory from the non-parabolic case to provide a calculation of the expected dimension of the moduli space of polystable parabolic $G$-Higgs bundles. Section 4 contains the construction of parabolic Teichm\"{u}ller components in full generality for any split real form $G$ of a complex simple Lie group; a proof of the parabolic Higgs analog of N. Hitchin's main theorem in \cite{Hit92} is presented here and is a generalization of the argument in \cite{Biswas3}. From Section 5 on, we are dealing with the cases when the group $G$ is Hermitian symmetric. After introducing the parabolic Toledo invariant for parabolic $\mathrm{Sp}\left( 2n,\mathbb{R} \right)$-Higgs bundles, we prove a Milnor-Wood type inequality, which allows one to introduce the notion of maximality. Section 6 is again preparatory and contains no new results; we review here the correspondence to orbifold Higgs bundles in a way that should fit into our needs. In Section 7 we introduce the topological invariants induced by this correspondence and calculate the total number of invariants in the case $G = \mathrm{Sp}\left( 2n,\mathbb{R} \right)$. In Section 8 we continue applying this component count method for the rest Hermitian symmetric Lie groups of tube type. Lastly, Section 9 contains a realization of classical results from the non-parabolic case formulated in terms of our $V$-bundle correspondence method.

\begin{ntn*} Throughout the article, we will be making use of the following notation for the corresponding moduli spaces we shall be considering:
\begin{itemize}
\item $\mathcal{M}_{par}^\alpha(G)$: moduli space of polystable parabolic $G$-Higgs bundles with fixed parabolic structure $\alpha$; all parabolic bundles have the same filtration over each $x \in D$ with the same weight $\alpha(x)$, for every $x \in D$.
\item $\mathcal{M}_{par}^{\mathbf{n}}(G)=\cup_{\alpha\in\frac{1}{n}\mathfrak{t}}\mathcal{M}_{par}^{\alpha}(G)$: the points in $\mathcal{M}_{par}^\alpha(G)$ for $\alpha\in\frac{1}{n}\mathfrak{t}$, where $\mathfrak{t}$ is the Lie algebra of a maximal torus of $H$ which is a maximal compact subgroup of $G$.
\item $\mathcal{M}_{par}(G)$: the moduli space $\mathcal{M}_{par}^{\mathbf{2}}(G)$.
\item $\mathcal{M}_{par}^{max, \alpha}(G)$: the points in $\mathcal{M}_{par}^\alpha(G)$ with maximal parabolic Toledo invariant.
\item $\mathcal{M}_{par}^{max}(G)$: the points in $\mathcal{M}_{par}^{\mathbf{2}}(G)$ with maximal parabolic Toledo invariant.
\end{itemize}
\end{ntn*}

\section{Definitions}

In this preliminary section, we  review the basic definitions for parabolic $\mathrm{GL}\left( n,\mathbb{C} \right)$-Higgs bundles; further details may be found in \cite{Biswas1}, \cite{BoYo}, or \cite{GaGoMu}. Parabolic $G$-Higgs bundles for a non-compact real reductive group $G$ were introduced in \cite{BiGaRi}, where a Hitchin-Kobayashi correspondence was also established. We have included these more general definitions in an Appendix at the end of this article, where we exhibit the cases $G = \mathrm{GL}(n,\mathbb{C})$ and $G = \mathrm{Sp}(2n,\mathbb{R})$ as well as the stability condition in detail.
\begin{defn}\label{201}
Let $X$ be a closed, connected, smooth Riemann surface of genus $g\ge 2$ and $D=\left\{ {{x}_{1}},\ldots ,{{x}_{s}} \right\}$ a divisor of $s$-many distinct points on $X$; denote this pair by $\left( X,D \right)$. A \emph{parabolic vector bundle} $E$ over $\left( X,D \right)$ is a holomorphic vector bundle $E\to X$ with \emph{parabolic structure} at each $x\in D$ (\emph{weighted flag} on each fiber ${{E}_{x}}$): \[\begin{matrix}
   {{E}_{x}}={{E}_{x,1}}\supset {{E}_{x,2}}\supset \ldots \supset {{E}_{x,r\left( x \right)+1}}=\left\{ 0 \right\}  \\
   0\le {{\alpha }_{1}}\left( x \right)<\ldots <{{\alpha }_{r\left( x \right)}}\left( x \right)<1.  \\
\end{matrix}\]
\end{defn}

We usually write $\left( E, \alpha \right)$ to denote a vector bundle equipped with a parabolic structure determined by a system of weights $\alpha \left( x \right)=\left( {{\alpha }_{1}}\left( x \right),\ldots ,{{\alpha }_{n}}\left( x \right) \right)$ at each $x\in D$. Moreover, set ${{k}_{i}}\left( x \right)=\dim\left( {{{E}_{x,i}}}/{{{E}_{x,i+1}}}\; \right)$ be the \emph{multiplicity} of the weight ${{\alpha }_{i}}\left( x \right)$. We can also write the weights repeated according to their multiplicity as
 \[0\le {{\tilde{\alpha }}_{1}}\left( x \right)\le \ldots \le {{\tilde{\alpha }}_{n}}\left( x \right)<1\]
where now $n=\mathrm{rk}E$. A weighted flag shall be called \emph{full}, if ${{k}_{i}}\left( x \right)=1$ for every $i$ and $x\in D$.

Given a pair of parabolic vector bundles the basic constructions for a parabolic subbundle, direct sum, dual and tensor product have been described in \cite{Biswas1} and \cite{GaGoMu}; we will be making frequent use of these constructions.

\begin{defn}\label{202}
A holomorphic map $f:E\to {E}'$ of parabolic vector bundles $\left( E, \alpha \right),\left( {E}', \alpha^{\prime } \right)$ is called \emph{parabolic} if ${{\alpha }_{i}}\left( x \right)>{{{\alpha }'}_{j}}\left( x \right)$ implies $f\left( {{E}_{x,i}} \right)\subset {{{E}'}_{x,j+1}}$, for every $x\in D$.\\
Furthermore, we call such map \emph{strongly parabolic} if ${{\alpha }_{i}}\left( x \right)\ge {{{\alpha }'}_{j}}\left( x \right)$ implies $f\left( {{E}_{x,i}} \right)\subset {{{E}'}_{x,j+1}}$ for every $x\in D$.\\
\end{defn}

\begin{defn}\label{203}
A notion of \emph{parabolic degree} and \emph{parabolic slope} of a vector bundle equipped with a parabolic structure can be defined as follows \[\pardeg \left( E \right)=\deg E+\sum\limits_{x\in D}{\sum\limits_{i=1}^{r\left( x \right)}{{{k}_{i}}\left( x \right){{\alpha }_{i}}\left( x \right)}}\]
\[\mathrm{par}\mu \left( \mathrm{E} \right)=\frac{\mathrm{pardeg}\left( E \right)}{\mathrm{rk}\left( E \right)}\]
\end{defn}

\begin{defn}\label{204}
A parabolic vector bundle will be called \emph{stable} (resp. \emph{semistable}), if for every non-trivial proper parabolic subbundle $F\le E$, it is $\mathrm{par}\mu \left( F \right)<\mathrm{par}\mu \left( E \right)$, (resp. $\le $).
\end{defn}

\begin{defn}\label{205}
Let $K=\Omega^{1}_{X} $ be the canonical bundle over $X$ and $E$ a parabolic vector bundle. Let $D=\{x_1,\ldots,x_s\}$. The bundle morphism $\Phi :E\to E\otimes K\left( D \right)$ will be called a \emph{parabolic Higgs field}, if it preserves the parabolic structure at each point $x\in D$:
	\[\Phi \left| _{x}\left( {{E}_{x,i}} \right) \right.\subset {{E}_{x,i}}\otimes K\left( D \right)\left| _{x} \right.\]
In particular, we call the Higgs field $\Phi$ \emph{strongly parabolic}, if
\[\Phi \left| _{x}\left( {{E}_{x,i}} \right) \right.\subset {{E}_{x,i+1}}\otimes K\left( D \right)\left| _{x} \right.,\]
in other words, $\Phi$ is a meromorphic endomorphism valued 1-form with at most simple poles along the divisor $D$, whose residue at $x\in D$ is nilpotent with respect to the filtration. Note that the divisor $D$ is always considered to be an effective divisor, and since $K\left( D \right)$ does not carry a weighted filtration, the parabolic structure on $E\otimes K\left( D \right)$ is induced only by the one on $E$.
\end{defn}
After these considerations we define parabolic Higgs bundles as follows.

\begin{defn}\label{206}
Let $K$ be the canonical bundle over $X$ and $E$ be a parabolic vector bundle over $X$. A \emph{parabolic Higgs bundle} over $\left( X,D \right)$ is given by a pair $\left( E,\Phi \right)$, where $\Phi :E\to E\otimes K\left( D \right)$ is a parabolic Higgs field.
\end{defn}

Analogously to the non-parabolic case, we may define stability as follows.

\begin{defn}\label{207}
A parabolic Higgs bundle will be called \emph{stable} (resp. \emph{semistable}), if for every $\Phi $-invariant parabolic subbundle $F\le E$ it is $\mathrm{par}\mu \left( F \right)<\mathrm{par}\mu \left( E \right)$ (resp. $\le $).
Furthermore, it will be called \emph{polystable}, if it is the direct sum of stable parabolic Higgs bundles of the same parabolic slope.
\end{defn}

In \cite{Yoko1} and \cite{Yoko2} K. Yokogawa has constructed the \emph{moduli space of semistable $K\left( D \right)$-pairs} ${{\mathsf{\mathcal{P}}}_{\alpha }}$, that is, pairs $\left( E,\Phi  \right)$ with $\Phi$ parabolic, using geometric invariant theory and has shown that it is a normal, smooth at the stable points, quasi-projective variety of dimension
	\begin{align}\dim{{\mathsf{\mathcal{P}}}_{\alpha }}=\left( 2g-2+s \right){{n}^{2}}+1,\label{paradim}\end{align}
for fixed $n=\mathrm{rk}E$, $d=\mathrm{pardeg} (E)$ and weight type $\alpha $. Moreover, in \cite{Konn} H. Konno constructed the \emph{moduli space of stable parabolic Higgs bundles} ${{\mathsf{\mathcal{N}}}_{\alpha }}$ as a hyperk{\"a}hler quotient. It is contained in ${{\mathsf{\mathcal{P}}}_{\alpha }}$ as a closed subvariety of dimension
\begin{align}\dim{{\mathsf{\mathcal{N}}}_{\alpha }}=2\left( g-1 \right){{n}^{2}}+2+2\sum\limits_{x\in D}{{{f}_{x}}},\label{sparadim}\end{align} where ${{f}_{x}}=\frac{1}{2}\left( {{n}^{2}}-\sum\limits_{i=1}^{r\left( x \right)}{{{\left( {{k}_{i}}\left( x \right) \right)}^{2}}} \right)$ is the dimension of the associated flag variety.

Parabolic $G$-Higgs bundles for a real reductive Lie group $G$ were first introduced in full generality in \cite{BiGaRi}; in the Appendix we include a brief review of these general definitions along with a detailed description of some examples. In particular, we check that the general definition of a parabolic $G$-Higgs bundle along with the (poly)stability condition in the case $G = \mathrm{GL}\left(n,\mathbb{C} \right)$ coincides with the definition included in this preliminary section. For a given parabolic structure, which we shall still denote by $\alpha$, we define $\mathcal{M}_{par}^\alpha(G)$ to be the \emph{moduli space of polystable parabolic $G$-Higgs bundles with fixed parabolic structure $\alpha$}. In particular, when $\alpha\in\frac{1}{n}\mathfrak{t}$ with $\mathfrak{t}$ the Lie algebra of a maximal torus of $H$ which is a maximal compact subgroup of $G$, we  denote $\mathcal{M}_{par}^{\mathbf{n}}(G)=\cup_{\alpha\in\frac{1}{n}\mathfrak{t}}\mathcal{M}_{par}^{\alpha}(G)$.  On the other hand, the standard notation $\mathcal{M}_{par}(G)$ will always mean in this article the moduli space in the special case when $n=2$, that is, polystable parabolic $G$-Higgs bundles with fixed parabolic structure $\alpha\in\frac{1}{2}\mathfrak{t}$.

\begin{rem}
The moduli spaces ${{\mathsf{\mathcal{P}}}_{\alpha }}$ and ${{\mathsf{\mathcal{N}}}_{\alpha }}$ considered above may be viewed as particular subspaces of $\mathcal{M}_{par}^\alpha(\mathrm{GL}\left(n,\mathbb{C} \right))$.
\end{rem}

\section{Deformation theory}

The deformation theory for parabolic $K(D)$-pairs was studied by K. Yokogawa in \cite{Yoko2}. We now adapt results from that article to the case of parabolic $G$-Higgs bundles for $G$ semisimple, analogously to the non-parabolic case studied in \S 3.3 of \cite{GaGoRi}. For a semisimple Lie group $G$, let $H\subset G$ be a maximal compact subgroup and let $\mathfrak{g}=\mathfrak{h}\oplus \mathfrak{m}$ be a Cartan decomposition  so that the Lie algebra structure of $\mathfrak{g}$ satisfies
	\[\left[ \mathfrak{h},\mathfrak{h} \right]\subset \mathfrak{h},\quad \left[ \mathfrak{h},\mathfrak{m} \right]\subset \mathfrak{m},\quad \left[ \mathfrak{m},\mathfrak{m} \right]\subset \mathfrak{h}.\]
Let ${{\mathfrak{g}}^{\mathbb{C}}}={{\mathfrak{h}}^{\mathbb{C}}}\oplus {{\mathfrak{m}}^{\mathbb{C}}}$ be the complexification of the Cartan decomposition and consider the sheaves $PE\left( {{\mathfrak{m}}^{\mathbb{C}}} \right)$ of \emph{parabolic sections of $E\left( {{\mathfrak{m}}^{\mathbb{C}}} \right)$} and $NE\left( {{\mathfrak{m}}^{\mathbb{C}}} \right)$ of \emph{strongly parabolic sections of $E\left( {{\mathfrak{m}}^{\mathbb{C}}} \right)$}; for the full definition, see the Appendix at the end of this article .

\begin{defn}\label{301}
Let $\left( E,\varphi  \right)$ be a parabolic $G$-Higgs bundle over $\left( X,D \right)$. The \emph{deformation complex} of the \emph{parabolic $G$-Higgs bundle} $\left( E,\varphi  \right)$ is the following complex of sheaves
	\[{{C}_P^{\bullet }}\left(G, E,\varphi  \right):PE\left( {{\mathfrak{h}}^{\mathbb{C}}} \right)\xrightarrow{[-,\varphi]} PE\left( {{\mathfrak{m}}^{\mathbb{C}}} \right)\otimes K\left( D \right),\]
where $[-,\varphi] :{{\mathfrak{h}}^{\mathbb{C}}}\to \mathrm{End}\left( {{\mathfrak{m}}^{\mathbb{C}}} \right)$ is the commutator $\psi \mapsto \psi\varphi-\varphi\psi$.\\
Assuming now that $\left( E,\varphi  \right)$ is strongly parabolic, the \emph{deformation complex} of the \emph{strongly parabolic $G$-Higgs bundle} $\left( E,\varphi  \right)$ is the complex of sheaves
	\[{{C}_N^{\bullet }}\left(G, E,\varphi  \right):PE\left( {{\mathfrak{h}}^{\mathbb{C}}} \right)\xrightarrow{[-,\varphi]} NE\left( {{\mathfrak{m}}^{\mathbb{C}}} \right)\otimes K\left( D \right).\]
\end{defn}
The above definition makes sense, since for instance in the parabolic case $\varphi $ is a meromorphic section of $PE\left( {{\mathfrak{m}}^{\mathbb{C}}} \right)\otimes K\left( D \right)$ and $\left[ {\mathfrak{m}_\alpha^{\mathbb{C}}},{\mathfrak{h}_\alpha^{\mathbb{C}}} \right]\subseteq {\mathfrak{m}_\alpha^{\mathbb{C}}}$ for any $\alpha\in \sqrt{-1}\bar{\mathcal{A}}$; we refer to the definitions in the Appendix for details. An analogous statement is true also in the strongly parabolic Higgs bundle case. Whenever there would be no ambiguity, we shall use the notation ${{C}_i^{\bullet }}\left( E,\varphi  \right)$ for the deformation complex, where $i=P,N$.

\begin{prop}\label{302}
The space of infinitesimal deformations of a parabolic $G$-Higgs bundle $\left( E,\varphi  \right)$ is naturally isomorphic to the hypercohomology group ${{\mathbb{H}}^{1}}\left( {{C}_P^{\bullet }}\left( E,\varphi  \right) \right)$. Analogously, the space of infinitesimal deformations of a strongly parabolic $G$-Higgs bundle $\left( E,\varphi  \right)$ is naturally isomorphic to the hypercohomology group ${{\mathbb{H}}^{1}}\left( {{C}_N^{\bullet }}\left( E,\varphi  \right) \right)$.
\end{prop}

\begin{proof} The proof follows exactly the same arguments as the proof of statement (3.2) from \cite{Thad} of M. Thaddeus, who described the infinitesimal deformations of parabolic Higgs bundles.
\end{proof}
For any parabolic $G$-Higgs bundle $\left( E,\varphi  \right)$ there is a natural long exact sequence
	\[0\to {{\mathbb{H}}^{0}}\left( {{C}_P^{\bullet }}\left( E,\varphi  \right) \right)\to {{H}^{0}}\left( PE\left( {{\mathfrak{h}}^{\mathbb{C}}} \right) \right)\xrightarrow{[-,\varphi]}{{H}^{0}}\left( PE\left( {{\mathfrak{m}}^{\mathbb{C}}} \right)\otimes K\left( D \right) \right)\]
	\[\to {{\mathbb{H}}^{1}}\left( {{C}_P^{\bullet }}\left( E,\varphi  \right) \right)\to {{H}^{1}}\left( PE\left( {{\mathfrak{h}}^{\mathbb{C}}} \right) \right)\xrightarrow{[-,\varphi]}{{H}^{1}}\left( PE\left( {{\mathfrak{m}}^{\mathbb{C}}} \right)\otimes K\left( D \right) \right)\to {{\mathbb{H}}^{2}}\left( {{C}_P^{\bullet }}\left( E,\varphi  \right) \right)\to 0.\]

The Serre duality theorem for parabolic sheaves (Proposition 3.7 in \cite{Yoko2}) provides that there are natural isomorphisms for $i=P,N$:
	\[{{\mathbb{H}}^{i}}\left( {{C}_i^{\bullet }}\left( E,\varphi  \right) \right)\cong {{\mathbb{H}}^{2-i}}{{\left( {{C}_i^{\bullet }}{{\left( E,\varphi  \right)}^{*}}\otimes K\left( D \right) \right)}^{*}},\]
where the dual of the deformation complex ${{C}_i^{\bullet }}\left( E,\varphi  \right)$ is defined as
	\[{{C}_P^{\bullet }}{{\left( E,\varphi  \right)}^{*}}:NE\left( {{\mathfrak{m}}^{\mathbb{C}}} \right)\otimes {{\left( K\left( D \right) \right)}^{-1}}\xrightarrow{[-,\varphi]}NE\left( {{\mathfrak{h}}^{\mathbb{C}}} \right),\]
	while the dual of the deformation complex for strongly parabolic pairs is respectively
	\[{{C}_N^{\bullet }}{{\left( E,\varphi  \right)}^{*}}:PE\left( {{\mathfrak{m}}^{\mathbb{C}}} \right)\otimes {{\left( K\left( D \right) \right)}^{-1}}\xrightarrow{[-,\varphi]}NE\left( {{\mathfrak{h}}^{\mathbb{C}}} \right).\]
	
An important special case is when $G$ is a complex group:
\begin{prop}\label{303}
Assume that $G$ is a complex semisimple Lie group. Then there is a natural isomorphism
	\[{{\mathbb{H}}^{2}}\left( {{C}_N^{\bullet }}\left( E,\varphi  \right) \right)\cong {{\mathbb{H}}^{0}}{{\left( {{C}_N^{\bullet }}\left( E,\varphi  \right) \right)}^{*}}.\]
\end{prop}

\begin{proof}When $G$ is complex, $\mathrm{ad:}\,\mathfrak{g}\to \mathfrak{g}$ and the Cartan decomposition of $\mathfrak{g}$ is $\mathfrak{g}=\mathfrak{u}+i\mathfrak{u}$, where $\mathfrak{u}=\mathrm{Lie}\left( U \right)$ for $U\subset G$ a maximal compact subgroup. Thus, in this case $\varphi \in NE\left( \mathfrak{g} \right)\otimes K\left( D \right)$. Moreover, for a complex group $G$ the deformation complex is dual to itself, except for a sign in the map, which does not affect cohomology:
	\[{{C}^{\bullet }_N} {{\left( E,\varphi  \right)}^{*}}\otimes K\left( D \right):PE\left( \mathfrak{g} \right)\xrightarrow{-\mathrm{ad}\left( \varphi  \right)}NE\left( \mathfrak{g} \right)\otimes K\left( D \right).\]
The result now follows from Serre duality.
\end{proof}

The proof of the next proposition is immediate, since $NE\left( {{\mathfrak{h}}^{\mathbb{C}}} \right)\oplus NE\left( {{\mathfrak{m}}^{\mathbb{C}}} \right)=NE\left( {{\mathfrak{g}}^{\mathbb{C}}} \right)$, given the Cartan decomposition ${{\mathfrak{g}}^{\mathbb{C}}}={{\mathfrak{h}}^{\mathbb{C}}}\oplus {{\mathfrak{m}}^{\mathbb{C}}}$. The Corollary that follows is also immediate from Serre duality:

\begin{prop}\label{304}
Let $G$ be a real semisimple group and let ${{G}^{\mathbb{C}}}$ be its complexification. Let $\left( E,\varphi  \right)$ be a strongly parabolic $G$-Higgs bundle. Then there is an isomorphism of complexes
	\[C_{N}^{\bullet }({G}^{\mathbb{C}}, E,\varphi )\cong C_{N}^{\bullet }\left(G, E,\varphi  \right)\oplus C_{N}^{\bullet }{{\left(G, E,\varphi  \right)}^{*}}\otimes K\left( D \right),\]
where $C_{N}^{\bullet }\left({G}^{\mathbb{C}}, E,\varphi  \right)$ denotes the deformation complex of $\left( E,\varphi  \right)$ viewed as a strongly parabolic ${{G}^{\mathbb{C}}}$-Higgs bundle, while $C_{N}^{\bullet }\left(G, E,\varphi  \right)$ denotes the deformation complex of $\left( E,\varphi  \right)$ viewed as a strongly parabolic $G$-Higgs bundle.
\end{prop}

\begin{cor}\label{305}
With the same hypotheses as in the previous Proposition, there is an isomorphism
\[{{\mathbb{H}}^{0}}\left( C_{N}^{\bullet }({G}^{\mathbb{C}}, E,\varphi ) \right)\cong {{\mathbb{H}}^{0}}\left( C_{N}^{\bullet }\left( G,E,\varphi  \right) \right)\oplus {{\mathbb{H}}^{2}}{{\left( C_{N}^{\bullet }\left(G, E,\varphi  \right) \right)}^{*}}.\]
\end{cor}

Consider now for a semisimple Lie group $G$, a stable and simple parabolic $G$-Higgs bundle $\left( E,\varphi  \right)$. As in the non-parabolic case \cite{GaGoRi}, if a (local) universal family exists then the dimension of the component of the moduli space containing the pair $\left( E,\varphi  \right)$ is equal to the dimension of the infinitesimal deformation space ${{\mathbb{H}}^{1}}\left( {{C}^{\bullet }}\left( E,\varphi  \right) \right)$; this dimension is referred to as the \emph{expected dimension} of the moduli space. In this situation, ${{\mathbb{H}}^{0}}\left( C_{i}^{\bullet }\left({G}^{\mathbb{C}}, E,\varphi  \right) \right)=0$ and so
\[{{\mathbb{H}}^{0}}\left( C^{\bullet }_i \left(G, E,\varphi  \right) \right)=0={{\mathbb{H}}^{2}}\left( C^{\bullet }_i\left(G, E,\varphi  \right) \right),\]
for $G$ semisimple and $i=P,N$. The long exact sequence then provides that
\[\dim{{\mathbb{H}}^{1}}\left( C^{\bullet }_i \left( E,\varphi  \right) \right)=-\chi \left( C^{\bullet }_i \left( E,\varphi  \right) \right),\]
where for simplicity we are keeping the same notation $\left( C^{\bullet }_i \left( E,\varphi  \right) \right)$ for the complex of sheaves for the group $G$ and $i=P,N$. The expected dimension of the moduli space ${{\mathsf{\mathcal{M}}}_{spar}}\left( G \right)$ of polystable parabolic $G$-Higgs bundles can be calculated using the Hirzebruch-Riemann-Roch formula and is independent of the choice of $\left( E,\varphi  \right)$:

\begin{prop}\label{306}
For a semisimple Lie group $G$, the moduli space ${{\mathsf{\mathcal{M}}}^{\alpha}_{par}}\left( G\right)$ of polystable parabolic $G$-Higgs bundles with parabolic structure $\alpha$ and $\mathbb{H}^2(C_P^\bullet(E,\varphi))=0$ for generic $(E,\varphi)$ has expected dimension
\begin{align}    -\chi(C_P^\bullet(E,\varphi)) &= (g-1)\dim G^{\mathbb{C}}+ s\cdot \mathrm{rk}\left(E\left(\mathfrak{m}^{\mathbb{C}}\right)\right)\nonumber\\ &+ \sum_{i} \left\{\dim\left(E\left(\mathfrak{h}^{\mathbb{C}}\right)_{x_i}/\mathfrak{h}^{-}_{\alpha_i}\right) - \dim\left(E\left(\mathfrak{m}^{\mathbb{C}}\right)_{x_i}/\mathfrak{m}^{-}_{\alpha_i}\right)  \right\},\end{align}
where $g$ is the genus of the Riemann surface $X$ and $s$ is the number of points in $D$. Here $\mathfrak{h}^{-}_{\alpha_i}$, $\mathfrak{m}^{-}_{\alpha_i}$ means the part of $\mathfrak{h}$, $\mathfrak{m}$ that is bounded by the $P_{\alpha_i}$-action as in Definition \ref{bs-}. Moreover, the moduli space of polystable  strongly parabolic $G$-Higgs bundles with parabolic structure $\alpha$ ${{\mathsf{\mathcal{M}}}^{\alpha}_{spar}}\left( G\right)$ has expected dimension
\begin{align}    -\chi(C_N^\bullet(E,\varphi))& = (g-1)\dim G^{\mathbb{C}}+ s\cdot \mathrm{rk}\left(E\left(\mathfrak{m}^{\mathbb{C}}\right)\right)\nonumber \\ &+ \sum_{i} \left\{\dim\left(E\left(\mathfrak{h}^{\mathbb{C}}\right)_{x_i}/\mathfrak{h}^{-}_{\alpha_i}\right) - \dim\left(E\left(\mathfrak{m}^{\mathbb{C}}\right)_{x_i}/\mathfrak{m}^{<0}_{\alpha_i}\right)  \right\},\end{align}
where $\mathfrak{m}^{<0}_{\alpha_i}$ is defined as in Definition \ref{bsls}.
\end{prop}
\noindent Note that the calculation for the expected dimension is based on a Riemann-Roch theorem argument, and thus we always mean here the complex dimension.
\begin{proof}
Let $\left( E,\varphi  \right)$ be any stable parabolic $G$-Higgs bundle. By definition we have
\[\chi(C_P^\bullet(E,\varphi))=\chi\left(PE\left(\mathfrak{h}^{\mathbb{C}}\right)\right)-\chi\left(PE\left(\mathfrak{m}^{\mathbb{C}}\right)\otimes K(D)\right).\]
The short exact sequence \ref{PEses} provides that
\[\chi\left(PE\left(\mathfrak{h}^{\mathbb{C}}\right)\right)=\chi\left(E\left(\mathfrak{h}^{\mathbb{C}}\right)\right)-\sum_{i} \chi\left(E\left(\mathfrak{h}^{\mathbb{C}}\right)_{x_i}/\mathfrak{h}^{-}_{\alpha_i}\right).\]
On the other hand, it is 
\[\chi\left(PE\left(\mathfrak{m}^{\mathbb{C}}\right)\otimes K(D)\right)=\chi\left(E\left(\mathfrak{m}^{\mathbb{C}}\right)\otimes K(D)\right)-\sum_{i} \chi\left(E\left(\mathfrak{h}^{\mathbb{C}}\right)_{x_i}/\mathfrak{m}^-_{\alpha_i}\right).\]
The Killing form now induces the isomorphisms $E\left(\mathfrak{m}^{\mathbb{C}}\right)\cong E\left(\mathfrak{m}^{\mathbb{C}}\right)^*$ and $E\left(\mathfrak{h}^{\mathbb{C}}\right)\cong E\left(\mathfrak{h}^{\mathbb{C}}\right)^*$, and hence 
$\deg E\left(\mathfrak{m}^{\mathbb{C}}\right)=\deg E\left(\mathfrak{h}^{\mathbb{C}}\right)=0$. Also, since $E\left(\mathfrak{m}^{\mathbb{C}}\right)$ and $E\left(\mathfrak{h}^{\mathbb{C}}\right)$ are vector bundles, one has by Riemann-Roch that
\begin{align*}
    \chi\left(PE\left(\mathfrak{h}^{\mathbb{C}}\right)\right)&=\deg\left(E\left(\mathfrak{h}^{\mathbb{C}}\right)\right)+\mathrm{rk}\left(E\left(\mathfrak{h}^{\mathbb{C}}\right)\right)(1-g)-\sum_{i} \chi\left(E\left(\mathfrak{h}^{\mathbb{C}}\right)_{x_i}/\mathfrak{h}^-_{\alpha_i}\right) \\
    &=\mathrm{rk}\left(E\left(\mathfrak{h}^{\mathbb{C}}\right)\right)(1-g)-\sum_{i} \dim\left(E\left(\mathfrak{h}^{\mathbb{C}}\right)_{x_i}/\mathfrak{h}^-_{\alpha_i}\right), \end{align*}
    as well as
    \begin{align*}
    &\chi\left(PE\left(\mathfrak{m}^{\mathbb{C}}\right)\otimes K(D)\right)\\&= \deg\left(E\left(\mathfrak{m}^{\mathbb{C}}\right)\otimes K(D)\right)+ \mathrm{rk}\left(E\left(\mathfrak{m}^{\mathbb{C}}\right)\right)(1-g) -\sum_{i} \chi\left(E\left(\mathfrak{m}^{\mathbb{C}}\right)_{x_i}/\mathfrak{m}^-_{\alpha_i}\right) 
    \\&= \mathrm{rk}\left(E\left(\mathfrak{m}^{\mathbb{C}}\right)\right)(2g-2+s)+ \mathrm{rk}\left(E\left(\mathfrak{m}^{\mathbb{C}}\right)\right)(1-g) -\sum_{i} \dim\left(E\left(\mathfrak{m}^{\mathbb{C}}\right)_{x_i}/\mathfrak{m}^-_{\alpha_i}\right) 
    \\&= \mathrm{rk}\left(E\left(\mathfrak{m}^{\mathbb{C}}\right)\right)(g-1) +  \mathrm{rk}\left(E\left(\mathfrak{m}^{\mathbb{C}}\right)\right)\cdot s - \sum_{i} \dim\left(E\left(\mathfrak{m}^{\mathbb{C}}\right)_{x_i}/\mathfrak{m}^-_{\alpha_i}\right).
\end{align*}
In conclusion we get 
\begin{align}
    -\chi(C_P^\bullet(E,\varphi)) &= (g-1)\dim G^{\mathbb{C}}+ s\cdot \mathrm{rk}\left(E\left(\mathfrak{m}^{\mathbb{C}}\right)\right) \nonumber \\ 
    & + \sum_{i} \left\{\dim\left(E\left(\mathfrak{h}^{\mathbb{C}}\right)_{x_i}/\mathfrak{h}^-_{\alpha_i}\right) - \dim\left(E\left(\mathfrak{m}^{\mathbb{C}}\right)_{x_i}/\mathfrak{m}^-_{\alpha_i}\right) \right\}.
\end{align}
    Similarly, for strongly parabolic Higgs bundles we only need to change from $PE\left(\mathfrak{m}^\mathbb{C}\right)$ to $NE\left(\mathfrak{m}^\mathbb{C}\right)$. Thus, we get
\begin{align}
    -\chi(C_N^\bullet(E,\varphi)) =& (g-1)\dim G^{\mathbb{C}}+ s\cdot \mathrm{rk}\left(E\left(\mathfrak{m}^{\mathbb{C}}\right)\right) \nonumber \\+ & \sum_{i} \left\{\dim\left(E\left(\mathfrak{h}^{\mathbb{C}}\right)_{x_i}/\mathfrak{h}^{-}_{\alpha_i}\right) - \dim\left(E\left(\mathfrak{m}^{\mathbb{C}}\right)_{x_i}/\mathfrak{m}^{<0}_{\alpha_i}\right)  \right\}.
\end{align}
The proof of the proposition is complete.
\end{proof}
In the case when $G$ is a complex Lie group, $\mathfrak{m}\cong \mathfrak{h}$, as so the last term in the summation for the parabolic deformation complex cancels out. We thus obtain the following:
\begin{cor}\label{311}
For a complex semisimple Lie group $G$, the moduli space of parabolic $G$-Higgs bundles is a smooth complex variety with expected dimension
\[2(g-1)\dim_{\mathbb{C}} G + s\cdot \mathrm{rk}\left(E\left(\mathfrak{m}^{\mathbb{C}}\right)\right).\]
\end{cor}
\noindent Notice that this calculation will give us real dimension
\[4(g-1)\dim_{\mathbb{C}} G + 2s\cdot \mathrm{rk}\left(E\left(\mathfrak{m}^{\mathbb{C}}\right)\right)=2(g-1)\dim_{\mathbb{\mathbb{R}}} G + 2s\cdot \mathrm{rk}\left(E\left(\mathfrak{m}^{\mathbb{C}}\right)\right).\]

The case for strongly parabolic will depend on the choice of parabolic structure. We can deduce part of the classical result for $G=\mathrm{GL}(n,\mathbb{C})$ in the following example:
\begin{exmp}
Let $G=\mathrm{GL}(n,\mathbb{C})$. As in previous calculation, we compute the Euler characteristic
\[\chi(C_P^\bullet(E,\varphi)) =-n^2(2g-2+s).\]
This does not provide directly the dimension of the moduli space, since when $G=\mathrm{GL}(n,\mathbb{C})$, we will have nonzero $\mathbb{H}^0(C_P^\bullet(E,\varphi))$, as $G$ is not semisimple. Indeed, $\dim \mathbb{H}^0(C_P^\bullet(E,\varphi))=1$, because there is an automorphism given by the identity. We would expect $\mathbb{H}^2(C_P^\bullet(E,\varphi))=0$ for a smooth point in the moduli space, then the expected dimension is
\[\dim \mathcal{M}_{par}^\alpha(\mathrm{GL}(n,\mathbb{C})) = n^2(2g-2+s)+1.\]
This calculation coincides with formula \ref{paradim}, which is also the result of Theorem 5.2 in \cite{Yoko2}.\\
In the case of strongly parabolic Higgs bundles, for a flag with grading such that $k_j(x_i)=\dim E_{x_i,j}-\dim E_{x_i,j+1}$ as in Example \ref{Gln}, we get
\[\dim\left(E\left(\mathfrak{h}^{\mathbb{C}}\right)_{x_i}/\mathfrak{h}^{-}_{\alpha_i}\right) - \dim\left(E\left(\mathfrak{m}^{\mathbb{C}}\right)_{x_i}/\mathfrak{m}^{<0}_{\alpha_i}\right) =-\sum_{j}k_j(x_i)^2.\]
Therefore, the Euler characteristic for the strongly parabolic Higgs bundle deformation complex will be
\[-\chi(C_N^\bullet(E,\varphi)) =n^2(2g-2)+\sum_{i} \left(n^2-\sum_{j} k_{j}(x_i)^2 \right).\]
In this situation, we have the Serre duality as in Corollary \ref{303}, which provides that
\[\dim \mathbb{H}^0(C_N^\bullet(E,\varphi))=\dim \mathbb{H}^2(C_N^\bullet(E,\varphi))=1.\]
Thus, one has
\[\dim {\mathsf{\mathcal{M}}}^{\alpha}_{spar}(\mathrm{GL}(n,\mathbb{C})) =n^2(2g-2)+\sum_{i} \left(n^2-\sum_{j} k_{j}(x_i)^2 \right) + 2.\]
In the case of a generic flag, which means that $k_{j}(x_i)=1$ for all possible $i,j$, the expected dimension is
\[\dim {\mathsf{\mathcal{M}}}^{\alpha}_{spar}(\mathrm{GL}(n,\mathbb{C})) =n^2(2g-2)+s \cdot n(n-1)+2.\]
This is in accordance with formula \ref{sparadim}, elaborated also in Proposition 2.4 in \cite{GaGoMu}. In a more general situation where we do not assume semisimplicity, we will be able to compute the dimension of $\mathbb{H}^0$ from the dimension of the center of $H^{\mathbb{C}}$ and derive a formula similar to Theorem II of \cite{Bhosle}, but we will not be discussing this here.

\end{exmp}

\begin{rem}\label{307}
Notice that when the number of punctures $s$ is zero, this dimension count coincides with the dimension count in Proposition 3.19 of \cite{GaGoRi} in the non-parabolic case as there is no contribution from parabolic points.
\end{rem}

\section{Parabolic Teichm\"{u}ller components}

In his seminal article \cite{Hit92}, N. Hitchin demonstrated the existence of topologically trivial connected components, which he then called \emph{Teichm\"{u}ller components}, in the moduli space $\mathrm{Ho}{{\mathrm{m}}^{+}}\left( {{\pi }_{1}}\left( \Sigma  \right),G \right)$ of reductive fundamental group representations into the adjoint group $G$ of the split real form of a complex simple Lie group ${{G}^{\mathbb{C}}}$, for a compact oriented surface $\Sigma $ of genus $g\ge 2$. Recall that the split real forms of the classical groups are the groups $\mathrm{SL}\left( n,\mathbb{R} \right)$, $\mathrm{SO}\left( n+1,n \right)$, $\mathrm{Sp}\left( 2n,\mathbb{R} \right)$ and $\mathrm{SO}\left( n,n \right)$. These components, in reality Euclidean spaces of dimension $2\left( g-1 \right){{\dim}_{\mathbb{R}}}G$, from the point of view of stable $G$-Higgs bundles are parameterized by fixed square roots of the canonical line bundle over the Riemann surface, for a choice of complex structure on $\Sigma $.

Later on, in \cite{Biswas3} the authors have extended N. Hitchin's results for a Riemann surface with $s$-many punctures and the group $G=\mathrm{SL}\left( k,\mathbb{R} \right)$. In particular, for a compact Riemann surface $X$ of genus $g$ and a divisor of $s$-many distinct points $D=\left\{ {{x}_{1}},\ldots ,{{x}_{s}} \right\}$ such that $2g-2+s>0$, they showed that Fuchsian representations of ${{\pi }_{1}}\left( X\backslash D \right)$ into $\mathrm{PSL}\left( 2,\mathbb{R} \right)$ are in one-to-one correspondence with parabolic $\mathrm{SL}\left( 2,\mathbb{R} \right)$-Higgs bundles of the form $\left( E,\theta  \right)$, satisfying the following:
\begin{enumerate}
  \item $E:={{\left( L\otimes \xi  \right)}^{*}}\oplus L$, \\
  where $L$ is a line bundle with ${{L}^{2}}={{K}_{X}}$ and $\xi ={{\mathsf{\mathcal{O}}}_{X}}\left( D \right)$ is the line bundle over the divisor $D$; the bundle $E$ is equipped with a parabolic structure given by a trivial flag ${{E}_{{{x}_{i}}}}\supset \left\{ 0 \right\}$ and weight $\frac{1}{2}$ for every $1\le i\le s$.
  \item $\theta :=\left( \begin{matrix}
   0 & 1  \\
   a & 0  \\
\end{matrix} \right)\in {{H}^{0}}\left( X,\mathrm{End}\left( E \right)\otimes K\left( D \right) \right)$,\\
for a meromorphic quadratic differential $a\in {{H}^{0}}\left( X,{{K}^{2}}\left( D \right) \right)$.
\end{enumerate}
Considering the $\left( k-1 \right)$-symmetric product of the parabolic vector bundle $E$, an extension of this result was provided also in \cite{Biswas3} for representations into $\mathrm{PSL}\left( k,\mathbb{R} \right)$, for $k>2$. Fuchsian representations of ${{\pi }_{1}}\left( X\backslash D \right)$ into $\mathrm{PSL}\left( k,\mathbb{R} \right)$ correspond to parabolic $\mathrm{SL}\left( k,\mathbb{R} \right)$-Higgs bundles $\left( {{W}_{k}},\theta \left( {{a}_{2}},\ldots ,{{a}_{k}} \right) \right)$, satisfying the following:
\begin{enumerate}
  \item ${{W}_{k}}={{S}^{k-1}}\left( E \right)\otimes {{\xi }^{m\left( k \right)}}$, where $m\left( k \right)=\left\{ \begin{matrix}
   \frac{k}{2}-1,\,\,\,k:\mathrm{even}  \\
   \frac{k-1}{2},\,\,\,k:\mathrm{odd}  \\
\end{matrix} \right.$, equipped with the trivial flag ${{\left( {{W}_{k}} \right)}_{{{x}_{i}}}}\supset \left\{ 0 \right\}$ with weight $\beta =\left\{ \begin{matrix}
   \frac{1}{2},\mathrm{ for }k\mathrm{ even}  \\
   0,\mathrm{ for }k\mathrm{ odd }  \\
\end{matrix} \right.$, for every $1\le i\le s$.
  \item $\theta \left( {{a}_{2}},\ldots {{a}_{k}} \right)=\left( \begin{matrix}
   0 & 1 & {} & 0  \\
   0 & 0 & \ddots  & 0  \\
   \vdots  & {} & \ddots  & 1  \\
   {{a}_{k}} & \cdots  & {{a}_{2}} & 0  \\
\end{matrix} \right)$, for merom. differentials ${{a}_{j}}\in {{H}^{0}}\left( X ,{{K}^{j}}\otimes {{\xi }^{j-1}} \right)$.
\end{enumerate}

Lastly, it was shown in \cite{Biswas3} that there exists a connected component of real dimension  $2\left( g-1 \right)\left( {{k}^{2}}-1 \right)+s\left( {{k}^{2}}-k \right)$ in the moduli space of representations of ${{\pi }_{1}}\left( X\backslash D \right)$ into $\mathrm{SL}\left( k,\mathbb{R} \right)$ with fixed conjugacy class of monodromy around the punctures. In the sequel, we extend these results for general split real $G$.

Using an irreducible representation $\phi :\mathrm{SL(2}\mathrm{,}\mathbb{R}\mathrm{)}\to G$ for a split real group $G$, which sends copies of a maximal compact subgroup of  $\mathrm{SL(2}\mathrm{,}\mathbb{R}\mathrm{)}$ into copies of a maximal compact subgroup of $G$, one can provide the existence of a parabolic Teichm\"{u}ller component similarly to the classical method by N. Hitchin. This was discussed in \S8 of \cite{BiGaRi} . In particular, the representation $\phi$ considered induces a decomposition ${{\mathfrak{g}}^{\mathbb{C}}}=\underset{i=1}{\overset{l}{\mathop{\oplus }}}\,{{V}_{i}}$ into irreducible pieces. For a standard $\mathfrak{sl}_{2}$ basis $\left( H,X,Y \right)$, with $H=\sqrt{-1}{{\mathfrak{u}}_{1}}$, take ${{e}_{1}}=\phi \left( X \right)$ and ${{e}_{-1}}=\phi \left( Y \right)$. There exists a basis ${{p}_{1}},\ldots ,{{p}_{l}}$ of invariant polynomials on ${{\mathfrak{g}}^{\mathbb{C}}}$ of degrees ${{m}_{i}}+1$, where $2{{m}_{i}}+1$ is the dimension of ${{V}_{i}}$, or equivalently ${{m}_{i}}$ is the eigenvalue of $\mathrm{ad}(H)$ on a highest weight vector ${{e}_{i}}\in {{V}_{i}}$, for $1\le i\le l$, with the property that for elements of the form
	\[f={{e}_{-1}}+{{f}_{1}}{{e}_{1}}+\ldots +{{f}_{l}}{{e}_{l}}\]
it is ${{p}_{i}}\left( f \right)={{f}_{i}}$. Analogously to the non-parabolic case of N. Hitchin \cite{Hit92}, one obtains a section $\psi$ of the map
	\[p:{{\mathsf{\mathcal{M}}}_{par}}\left( G \right)\to \underset{i=1}{\overset{l}{\mathop{\oplus }}}\,{{H}^{0}}\left( X,{{K}^{{{m}_{i}}+1}}\otimes {{\xi }^{{{m}_{i}}}} \right)\]
consisted of a family of parabolic $G$-Higgs bundles $\left( \phi \left( E \right),\varphi  \right)$, where:
\begin{enumerate}
  \item $E={{\left( L\otimes \xi  \right)}^{*}}\oplus L$, for ${{L}^{2}}={{K}_{X}}$ and $\xi ={{\mathsf{\mathcal{O}}}_{X}}\left( D \right)$, and $\phi \left( E \right)$ is equipped with the trivial flag ${{\left( \phi \left( E \right) \right)}_{{{x}_{i}}}}\supset \left\{ 0 \right\}$ with weight $\frac{1}{2}$, for every $1\le i\le s$, and
  \item The Higgs field is considered to be	given by\[\varphi ={{e}_{-1}}+{{a}_{1}}{{e}_{1}}+\ldots {{a}_{l}}{{e}_{l}}\]
with ${{a}_{j}}\in {{H}^{0}}\left( X,{{K}^{{{m}_{i}}+1}}\otimes {{\xi }^{{{m}_{i}}}} \right)$, for $1\le j\le l$.
\end{enumerate}
The Higgs field $\varphi $ is meromorphic with simple poles at the points ${{x}_{i}}$ in the divisor $D$ and the residue $\mathrm{Re}{{\mathrm{s}}_{{{x}_{i}}}}\varphi ={{e}_{-1}}$ is nilpotent with $l$-dimensional centralizer, where $l=\mathrm{rk}{{\mathfrak{g}}^{\mathbb{C}}}$. The section $\psi$ thus provides the existence of a \textit{parabolic Teichm\"{u}ller component}.

In fact, these components are parameterized by parabolic square roots of the line bundle $K\left( D \right)$ of degree $2g-2+s$, that is, line bundles ${{L}_{0}} \to X$ with $\deg {{L}_{0}}=g-1$ equipped with a trivial flag ${{\left( {{L}_{0}} \right)}_{{{x}_{i}}}}\supset \left\{ 0 \right\}$ and parabolic weight $\frac{1}{2}$ on each fiber ${{\left( {{L}_{0}} \right)}_{{{x}_{i}}}}$, for $1\le i\le s$. Notice that, for such objects, $\pardeg {{L}_{0}}=g-1+\frac{s}{2}$. In \S 7.2 later on, we show that there are ${{2}^{2g+s-1}}$ many non-isomorphic such parabolic square roots of $K\left( D \right)$, thus there exist ${{2}^{2g+s-1}}$ parabolic Teichm\"{u}ller components of ${{\mathsf{\mathcal{M}}}_{par}}\left( G \right)$ for a split real group $G$. We finally apply the Riemann-Roch formula to compute the dimension of these components:

\begin{thm}\label{401}
Let $X$ be a compact Riemann surface of genus $g$ and let $D=\left\{ {{x}_{1}},\ldots {{x}_{s}} \right\}$ a divisor of $s$-many distinct points on $X$, such that $2g-2+s>0$, that is, the surface $X$ can be equipped with a metric of constant negative curvature (-4). Let $G$ be the adjoint group of the split real form of a complex simple Lie group with Cartan decomposition in the Lie algebra $\mathfrak{g}=\mathfrak{h}\oplus \mathfrak{m}$. The space of homomorphisms from the fundamental group of $X$ into $G$, with fixed conjugacy class of monodromy around the points in $D$, has a component of real dimension $2\left( g-1 \right){{\dim}_{\mathbb{R}}}G+2s\cdot \mathrm{rk}E\left( {{\mathfrak{m}}^{\mathbb{C}}} \right)$.
\end{thm}

\begin{proof}
In \cite{Boalch}, P. Boalch established a correspondence between parabolic connections and $G$-filtered fundamental group representations, thus extending the non-abelian Hodge correspondence over non-compact curves of C. Simpson \cite{Simp} for $\mathrm{GL}(n,\mathbb{C})$. As in N. Hitchin's classical approach, this correspondence identifies the subfamily defined by the parabolic Hitchin section $\psi$ with the moduli space of completely reducible flat $G$-connections on $X\backslash D$, meromorphic at ${{x}_{i}}\in D$ and whose holonomy is $G$-conjugated to an element $U\in H$, where $H\subset G$ is a maximal compact subgroup of $G$.

\noindent From the Riemann-Roch formula, one obtains that the real dimension of the vector space $\underset{i=1}{\overset{l}{\mathop{\oplus }}}\,{{H}^{0}}\left( X,{{K}^{{{m}_{i}}+1}}\otimes {{\xi }^{{{m}_{i}}}} \right)$ is equal to
	\[2\left[ \sum\limits_{i=1}^{l}{\left( 2{{m}_{i}}+1 \right)\left( g-1 \right)+{{m}_{i}}s} \right]=2\left( g-1 \right){{\dim}_{\mathbb{R}}}G+2s\sum\limits_{i=1}^{l}{{{m}_{i}}}.\]
For the family of parabolic Higgs bundles we have considered, the residue of the Higgs field $\mathrm{Re}{{\mathrm{s}}_{{{x}_{i}}}}\varphi ={{e}_{-1}}$ is regular, nilpotent.

\noindent On the other hand, ${{\pi }_{1}}\left( X \right)=\left\langle {{c}_{1}},{{d}_{1}},\ldots ,{{c}_{g}},{{d}_{g}},{{e}_{1}},\ldots ,{{e}_{s}}\left| \prod\limits_{j=1}^{g}{\left[ {{c}_{j}},{{d}_{j}} \right]\prod\limits_{j=1}^{s}{{{e}_{j}}=id}} \right. \right\rangle $, where ${{c}_{j}},{{d}_{j}}$ are simple loops around the handles of $X$ and ${{e}_{j}}$ are simple loops around the points ${{x}_{i}}$ in the divisor. The image of the elements ${{c}_{j}},{{d}_{j}}$ via a representation $\rho :{{\pi }_{1}}\left( X \right)\to G$ depends on $\dim G$ different parameters. Moreover, for each loop ${{e}_{j}}$, the relations for weights and monodromies in Table 1 of \cite{BiGaRi} provide that the image of ${{e}_{j}}$ via a representation $\rho :{{\pi }_{1}}\left( X\backslash D \right)\to G$ is a regular unipotent element $E_j$. Let $U_j$ be the set of all conjugacy classes of $E_j$. To calculate the number of parameters for the image of $e_j$ is equivalent to calculating the number of parameters for the set $U_j$. Clearly, any element in $U_j$ can be written as $A E_j A^{-1}$, where $A\in {G}/{I}\;$, and where $I$ is the centralizer of the unipotent and regular element ${{E}_{j}}$. This means that $\dim I=l$, where $l=\mathrm{rk}{{\mathfrak{g}}^{\mathbb{C}}}$, thus the total number of parameters for the image $\rho \left( \left[ {{e}_{j}} \right] \right)$ is ${{\dim}_{\mathbb{C}}}{{\mathfrak{g}}^{\mathbb{C}}}-l=\sum\limits_{i=1}^{l}{2{{m}_{i}}}$. We deduce that the real dimension of the space of fundamental group representations into $G$ with the monodromy around the points in $D$ lying in the conjugacy class of an element in $H$, is equal to $2{{\dim}_{\mathbb{R}}}G\left( g-1 \right)+2s\sum\limits_{i=1}^{l}{{{m}_{i}}}$. This coincides with the dimension count for the vector space $\underset{i=1}{\overset{l}{\mathop{\oplus }}}\,{{H}^{0}}\left( X,{{K}^{{{m}_{i}}+1}}\otimes {{\xi }^{{{m}_{i}}}} \right)$, which can be written as
	\[2{{\dim}_{\mathbb{R}}}G\left( g-1 \right)+2s\cdot \mathrm{rk}E\left( {{\mathfrak{m}}^{\mathbb{C}}} \right)\]
since for the weights in the family of Higgs bundles in the Hitchin section, it holds in particular that ${{\left( E\left( {{\mathfrak{m}}^{\mathbb{C}}} \right) \right)}_{{{x}_{i}}}}\simeq {{\mathfrak{m}}^{\mathbb{C}}}$.
\end{proof}
\begin{rem}Note that the calculation here coincides with the calculation of expected real dimension of the moduli space from Corollary \ref{311}. Moreover, in the absence of punctures, the dimension of a Teichm\"{u}ller component coincides with the one from \cite{Hit92}.
\end{rem}

\section{Maximal Parabolic components}

Distinguished components of the moduli space $\mathcal{M}_{par}^\alpha(G)$ also exist when the homogeneous space  ${G}/{H}\;$ is a Hermitian symmetric space of noncompact type, where $H\subset G$ is a maximal compact subgroup. For the classical groups, this means considering the Lie groups  $\mathrm{SU}(p,q)$, $\mathrm{Sp}(2n,\mathbb{R})$, $\mathrm{S}{{\mathrm{O}}^{*}}(2n)$ and $\mathrm{S}{{\mathrm{O}}_{0}}(2,n)$. In this case,  $\mathfrak{h}=\mathrm{Lie}\left( H \right)$ has a 1-dimensional center and there is a decomposition of ${{\mathfrak{m}}^{\mathbb{C}}}$ into its $\pm i$-eigenspaces ${{\mathfrak{m}}^{\mathbb{C}}}={{\mathfrak{m}}^{+}}\oplus {{\mathfrak{m}}^{-}}$. For a parabolic $G$-Higgs bundle $\left( E,\varphi  \right)$ with the Higgs field $\varphi$ decomposing accordingly as $\varphi ={{\varphi }^{+}}+{{\varphi }^{-}}$, the authors in \cite{BiGaRi} define a Toledo invariant $\tau \left( E \right)$ analogously to the non-parabolic case and provide a general inequality of Milnor-Wood type.

\begin{prop}[O. Biquard, O. Garc\'{i}a-Prada, I. Mundet i Riera \cite{BiGaRi}]\label{501}
For a semistable parabolic $G$-Higgs bundle $\left( E,\varphi  \right)$ on a Riemann surface with a divisor $\left( X,D \right)$, it holds that
	\[-\mathrm{rk}\left( {{\varphi }^{+}} \right)\left( 2g-2+s \right)\le \tau \left( E \right)\le \mathrm{rk}\left( {{\varphi }^{-}} \right)\left( 2g-2+s \right).\]
\end{prop}

In the sequel of this section, we study explicitly the case when $G=\mathrm{Sp}(2n,\mathbb{R})$, while some further details are moved to the Appendix. Then, in \S 6 and \S 7 we describe topological invariants for maximal parabolic $G=\mathrm{Sp}(2n,\mathbb{R})$-Higgs bundles. The analysis for the case $G=\mathrm{Sp}(2n,\mathbb{R})$ can be then readily adapted for the study of maximal parabolic $G$-Higgs bundles also for the other Hermitian symmetric spaces ${G}/{H}$.

\subsection{Maximal Parabolic $\mathrm{Sp}\left( 2n,\mathbb{R}\right )$-Higgs bundles.} A maximal compact subgroup of $G=\mathrm{Sp}\left( 2n,\mathbb{R}\right )$ is $H=\mathrm{U}\left( n \right)$ and ${{H}^{\mathbb{C}}}=\mathrm{GL}\left( n,\mathbb{C} \right)$, thus the parabolic structure on a $\mathrm{GL}\left( n,\mathbb{C} \right)$-principal bundle is in this case defined by a weighted filtration. We will first fix some notation before giving the precise definitions.

Let $X$ be a compact Riemann surface of genus $g$ and let the divisor $D:=\left\{ {{x}_{1}},\ldots ,{{x}_{s}} \right\}$ of $s$-many distinct points on $X$, assuming that $2g-2+s>0$. Let $K$ denote, as usual, the canonical line bundle over $X$ of degree $2g-2$, and $\xi :={{\mathsf{\mathcal{O}}}_{{X}}}\left( D \right)$ the line bundle on $X$ given by the divisor $D$. The degree of the line bundle $ K\otimes \xi$ is $2g-2+s$, where $s$ is the number of points in the divisor considered.

Let $V$ be a rank $n$ holomorphic bundle over $X$. Equip this with a parabolic structure given by a weighted flag
on each fiber ${{V}_{x_{i}}}$:
\begin{equation}\tag{2}
\begin{matrix}
   {{V}_{x_{i}}}\supset {{V}_{x_{i},2}}\supset \ldots \supset {{V}_{x_{i},n+1}}=\left\{ 0 \right\}  \\
   0\le {{\alpha }_{1}}\left( x_{i} \right)\le \ldots \le  {{\alpha }_{n}}\left( {{x}_{i}} \right)<1  \\
\end{matrix}
\end{equation}
for each ${{x}_{i}}\in D$. The parabolic degree of the parabolic bundle $\left( V,\alpha  \right)$ is given by the rational number
	\[\pardeg V=\deg V+\sum\limits_{{{x}_{i}}\in D}{\sum\limits_{j=1}^{n}{{{\alpha }_{j}}\left( {{x}_{i}} \right)}}\]

For a parabolic principal ${{H}^{\mathbb{C}}}=\mathrm{GL}\left(n,\mathbb{C}\right)$-bundle $E$, let $E\left( {{\mathfrak{m}}^{\mathbb{C}}} \right)$ denote the (parabolic) bundle associated to $E$ via the isotropy representation and, as a bundle,
	\[E\left( {{\mathfrak{m}}^{\mathbb{C}}} \right)=\mathrm{Sy}{{\mathrm{m}}^{n}}\left( V \right)\oplus \mathrm{Sy}{{\mathrm{m}}^{n}}\left( {{V}^{*}} \right)\]
for $V$ the rank $n$ bundle associated by the standard representation. The definition of a parabolic $\mathrm{Sp}\left( 2n,\mathbb{R} \right)$-Higgs bundle according to the authors in \cite{BiGaRi} specializes to the following:

\begin{defn}\label{502}
Let $X$ be a compact Riemann surface of genus $g$ and let the divisor $D:=\left\{ {{x}_{1}},\ldots ,{{x}_{s}} \right\}$ of $s$-many distinct points on $X$, assuming that $2g-2+s>0$. A \emph{parabolic $\mathrm{Sp}\left( 2n,\mathbb{R} \right)$-Higgs bundle} is defined as a triple $\left( V,\beta ,\gamma  \right)$, where
\begin{itemize}
\item $V$ is a rank $n$ bundle on $X$, equipped with a parabolic structure given by a weighted flag as in (2), and
\item The maps  $\beta :{{V}^{\vee }}\to V\otimes K\otimes \xi$ and $\gamma :V\to {{V}^{\vee }}\otimes K\otimes \xi$ are parabolic symmetric morphisms.
\end{itemize}
\end{defn}

The parabolic structures on $V$ and ${{V}^{\vee }}$ now induce a parabolic structure on the parabolic sum $E=V\oplus {{V}^{\vee }}$, for which $\pardeg E=0$. We define alternatively a parabolic $\mathrm{Sp}\left( 2n,\mathbb{R} \right)$-Higgs bundle on $\left( X,D \right)$ as a parabolic Higgs bundle $\left( E,\Phi  \right)$, where $E=V\oplus {{V}^{\vee }}$ and $\Phi =\left( \begin{matrix}
   0 & \beta   \\
   \gamma  & 0  \\
\end{matrix} \right):E\to E\otimes K\left( D \right)$; the stability condition for such pairs $\left( E,\Phi  \right)$ will be the one considered in Definition \ref{204}.

\begin{defn}\label{503}
The \emph{parabolic Toledo invariant} of a parabolic $\mathrm{Sp}\left( 2n,\mathbb{R} \right)$-Higgs bundle is defined as the rational number
\[\tau =\pardeg \left( V \right).\]
\end{defn}
Moreover, we may obtain a \emph{Milnor-Wood type inequality} for this topological invariant:
\begin{prop}\label{504}
Let $\left( E,\Phi  \right)$ be a semistable parabolic $\mathrm{Sp}\left( 2n,\mathbb{R} \right)$-Higgs bundle. Then \[\left| \tau  \right|\le n\left( g-1+\frac{s}{2} \right),\]
where $s$ is the number of points in the divisor $D$.
\end{prop}

\begin{proof}
Consider parabolic bundles $N=\ker \left( \gamma  \right)$ and $I=\operatorname{Im}\left( \gamma  \right)\otimes \left(K\otimes \xi\right)^{-1}\le {{V}^{\vee }}$. \\
We thus get an exact sequence of parabolic bundles \[0\to N\to V\to I\otimes K\otimes \xi\to 0\] and so
\begin{align*}
   \mathrm{par}\deg \left( V \right) & = \mathrm{par}\deg \left( N \right)+\mathrm{par}\deg \left( I\otimes K\otimes \xi \right) \\
   & = \mathrm{par}\deg \left( N \right)+\mathrm{par}\deg \left( I \right)+\mathrm{rk}\left( I \right)\left( 2g-2+s \right)
   \tag{3}\end{align*}
using the formula that gives the parabolic degree for the tensor product and the fact that $\mathrm{par}\deg \left(K\otimes \xi\right)=2g-2+s$.\\
$I$ is a subsheaf of ${{V}^{\vee }}$ and $I\hookrightarrow {{V}^{\vee }}$ is a parabolic map. Let $\tilde{I}\subset {{V}^{\vee }}$ be its saturation, which is a subbundle of ${{V}^{\vee }}$ and endow it with the induced parabolic structure. So $N,V\oplus \tilde{I}\subset E$ are $\Phi $-invariant parabolic subbundles of $E$. The semistability of $\left( E,\Phi  \right)$ now implies $\mathrm{par}\mu \left( N \right)\le \mathrm{par}\mu \left( E \right)$ and
$\mathrm{par}\mu \left( V\oplus I \right)\le \mathrm{par}\mu \left( V\oplus \tilde{I} \right)\le \mathrm{par}\mu \left( E \right)$. However,
\[\mathrm{par}\mu \left( E \right)=\frac{\pardeg \left(E \right)}{\mathrm{rk}\left( E \right)}=0,\]
thus we have
\[\pardeg \left( N \right)\le 0\]
and
\[ \pardeg \left( V \right) +\pardeg \left( I \right)\le 0.\]
Equation (3) provides that
\[ \pardeg \left( V \right)\le -\pardeg \left( V \right)+rk\left( I \right)\left( 2g-2+s \right),\]
thus
\[ \tau \le n\left( g-1+\frac{s}{2} \right).\]
The map $\left( V,\beta ,\gamma  \right)\mapsto \left( {{V}^{\vee }}\otimes \xi ,\gamma ,\beta  \right)$ defines an isomorphism ${{\mathsf{\mathcal{M}}}_{-\tau }}\cong {{\mathsf{\mathcal{M}}}_{\tau }}$ providing the minimal bound $-\tau \le n\left( g-1+\frac{s}{2} \right)$, where ${{\mathsf{\mathcal{M}}}_{\tau }}$ denotes the subspace consisted of the pairs with fixed parabolic Toledo invariant $\tau$.
\end{proof}

\begin{defn}\label{505}
The polystable parabolic $\mathrm{Sp}\left( 2n,\mathbb{R} \right)$-Higgs bundles with fixed parabolic structure $\alpha$ and parabolic Toledo invariant $\tau =n\left( g-1+\frac{s}{2} \right)$ will be called \emph{maximal} and we shall denote the subspace of the moduli space containing those by $\mathsf{\mathcal{M}}_{par}^{max,\alpha}(\mathrm{Sp}(2n,\mathbb{R}))$. We define $\mathcal{M}_{par}^{max}(\mathrm{Sp}(2n,\mathbb{R}))$ to be the union of $\mathsf{\mathcal{M}}_{par}^{max,\alpha}(\mathrm{Sp}(2n,\mathbb{R}))$ with $\alpha \in \frac{1}{2} \mathfrak{t}$, i.e. $\mathcal{M}_{par}^{max}(\mathrm{Sp}(2n,\mathbb{R}))=\cup_{\alpha\in\frac{1}{2}\mathfrak{t}}\mathcal{M}_{par}^{max,\alpha}(\mathrm{Sp}(2n,\mathbb{R}))$.
\end{defn}

It can be shown that this subspace is non-empty (see \cite{Kydon}).

\section{The correspondence to orbifold Higgs bundles}

The topology of parabolic semistable $G$-Higgs bundle moduli spaces has been studied so far in the case when $G=\mathrm{GL(}n\mathrm{,}\mathbb{C}\mathrm{)}$ in \cite{GaGoMu}, $G=\mathrm{U}\left( 2,1 \right)$ in \cite{Loga2} and $G=\mathrm{U}\left( p,q \right)$ in \cite{GaLoMu}, \cite{Loga}, which pioneered the study of irreducible components of parabolic Higgs bundles for a real Lie group. In the case $G=\mathrm{Sp}\left( 2n,\mathbb{R} \right)$, we are fixing the parabolic degree $d=\mathrm{pardeg}(V)$ of the bundle and the weight type $\alpha$, where we assign weight equal to either $0$ or $\frac{1}{2}$ for the trivial flag on each fiber $V_{x_{i}}$, $x_{i}\in{D}$; let ${\mathsf{\mathcal{M}}}_{par}^d \left( \mathrm{Sp}\left( 2n,\mathbb{R} \right)  \right)$ be the moduli space of polystable parabolic $\mathrm{Sp}\left( 2n,\mathbb{R} \right)$-Higgs bundles of degree $d$ and the weight type described above. We note that the parabolic structures we consider in the rest of the paper lie in $\frac{1}{2}\mathfrak{t}$. In other words, any parabolic weight can be written as a fraction with denominator $2$.

The reason why we are fixing these particular parabolic structures for parabolic Higgs bundles in the moduli space $\mathsf{\mathcal{M}}_{par}^d \left( \mathrm{Sp}\left( 2n,\mathbb{R} \right) \right)$ will become clear in what follows. In the case when the weights in the parabolic structure of the bundle are rational numbers, we use a correspondence between parabolic Higgs bundles and orbifold Higgs bundles in order to define appropriate topological invariants and count connected components. Under the assumption that the choices of the weights are either $0$ or $\frac{1}{2}$, any weight can be written as a fraction with denominator $2$. Thus, we may construct a special $V$-manifold from the data $\left( X,D,2 \right)$, where $m=2$ describes the cyclic group action around the points in the divisor $D$ and it precisely  corresponds to the denominator $``2"$. The topological invariants of the corresponding Higgs $V$-bundles we are interested in will be conceived as characteristic classes in $V$-cohomology groups with ${{\mathbb{Z}}_{2}}$-coefficients. In this section, we review -for the most part- the correspondences from \cite{Biswas2}, \cite{FuSt} and \cite{NaSt}. In particular, we construct an orbifold Higgs field for a bundle of any rank following closely the constructions in the aforementioned references. 

\subsection{Orbifold Higgs bundles}
Let $Y$ be a closed, connected, smooth Riemann surface and let $\mathrm{Aut}\left( Y \right)$ be the group of algebraic automorphisms of $Y$. Assume that the finite group $\Gamma$ acts faithfully on $Y$, in other words, there is an injective homomorphism $h :\Gamma \to \mathrm{Aut}\left( Y \right)$. Denote by $[Y / \Gamma]$ the orbifold and let $E$ be a vector bundle over $Y$. We say that $E$ is \emph{$\Gamma$-equivariant}, if there is a group action on $E$, $\rho: \Gamma \times E \rightarrow E$, such that $\phi \circ \rho(\gamma,z)=h(\gamma)(\phi(z))$, where $\phi : E \rightarrow Y$ is the projection.

\begin{defn}\label{601}
An \emph{orbifold sheaf} on $[Y / \Gamma]$ is a torsion free coherent sheaf $E$ on $Y$ together with a lift of the action of $\Gamma$ to $E$, such that the automorphism of the space of stalks for the action of any $g\in \Gamma $ is a coherent sheaf isomorphism between $E$ and $\rho {{\left( {{g}^{-1}} \right)}^{*}}E$; when $E$ is locally free, it is called an \emph{orbifold bundle}.
\end{defn}

\begin{defn}
The \textit{degree} of an orbifold bundle $E$ on $[Y/\Gamma]$ is defined to be
\[\deg_{orb}(E)=\frac{1}{|\Gamma|}\int_Y c_1(E),\]
where $c_1(E)$ is the first Chern class of $E$ as a holomorphic bundle on $Y$.

\noindent The \textit{orbifold slope} will be given by the fraction
\[\mu_{orb}(E)=\frac{\deg(E)}{\mathrm{rk}(E)}.\]
\end{defn}

Recall that a \emph{Higgs field} $\Phi$ of a holomorphic bundle $E$ over $Y$ is a holomorphic section of $\mathrm{End}(E)\otimes K$, where $K$ is the canonical line bundle over $Y$. We define next the orbifold Higgs field:

\begin{defn}\label{602}
An \emph{orbifold Higgs field} $\Phi$ over the orbifold bundle $E$ is a Higgs field such that it is equivariant with respect to the action of $\Gamma$, i.e., $\rho(g^{-1})^* \Phi = \Phi$.
\end{defn}

\begin{defn}\label{603}
A \emph{Higgs bundle over the orbifold} $[Y / \Gamma]$ is a pair $(E, \Phi)$, where $E$ is an orbifold bundle and $\Phi$ is an orbifold Higgs field.
\end{defn}

An orbifold bundle $E$ is called \emph{orbifold stable (resp. semistable)}, if for any $\Gamma$-invariant stable (resp. semistable) subbundle $F$ of $E$ with $0 < \mathrm{rank } F <\mathrm{rank } E$, the inequality $\mu_{orb}(F) < \mu_{orb}(E)$ (resp. $\mu_{orb}(F) \le \mu_{orb}(E)$) holds. An orbifold Higgs bundle $(E, \Phi)$ will be called \emph{orbifold stable (resp. semistable)}, if for any orbifold Higgs field $\Phi$ and $\Gamma$-invariant subbundle $F$, the above inequality holds; further details can be found in \cite{FuSt}, \cite{NaSt} and \cite{Biswas2}.

In this article, we are interested in the global quotient $[Y / \Gamma]$ where the underlying space $X$ is a compact Riemann surface.

\subsection{Local Picture of an Orbifold Higgs Bundle}
Let $\widetilde{M}$ be a $k$-dimensional manifold with $s$-many marked points $x_1,...,x_s$. For each marked point, there is a linear representation $\sigma_i: \Gamma_i \rightarrow \mathrm{Aut}(\mathbb{R}^k)$ of a cyclic group ${{\Gamma }_{i}}=\left\langle {{\sigma }_{i}} \right\rangle $, $1 \leq i \leq s$, where $\Gamma_i$ acts freely on $\mathbb{R}^k \backslash \{0\}$ together with an atlas of coordinate charts
\begin{align*}
& \phi_i: U_i \rightarrow D^k / \sigma_i, & 1 \leq i \leq s;\\
& \phi_p: U_p \rightarrow D^k, & p \in \widetilde{M} \backslash \{x_1,...,x_s\}.
\end{align*}
We get the orbifold $M$ by gluing all local coordinate charts above, while $\widetilde{M}$ is the underlying manifold of $M$. The example we are interested in is $M=[Y / \Gamma]$, where $Y$ is a closed, connected, smooth Riemann surface and $\Gamma$ is a finite group acting effectively on $Y$. In this case, the underlying space $\widetilde{M}$ is exactly the underlying space $X$ of $[Y/\Gamma]$. In \cite{FuSt}, M. Furuta and B. Steer consider this construction to define a \emph{$V$-manifold}; we review some properties of $V$-manifolds in \S 7 later on.

\begin{defn}\label{604}
A \emph{holomorphic orbifold bundle} $E$ of rank $n$ over $M$ is defined locally on the charts as above with a collection of isotropy representations $\tau_i: \Gamma_i \rightarrow \mathrm{Aut}(C^n)$ and local trivializations $\theta_i : E|_{U_i} \rightarrow D^k \times C^n / \sigma_i \times \tau_i$, for $1 \leq i \leq s$.
\end{defn}
Forgetting the group action, we get a well defined holomorphic vector bundle $\widetilde{E}$ over the underlying space $\widetilde{M}$. We say that a local trivialization $\Theta_i : \widetilde{E}|_{D^k} \rightarrow D^k \times C^n$ is \emph{compatible} with the orbifold structure (with respect to $E$), if $\Theta_i$ is $\Gamma_i$-equivariant, where the $\Gamma_i$ action comes from the local trivialization $\theta_i$. Notice that Definition \ref{604} is the local description of Definition \ref{601}.

We now give an example of the local chart of a rank $n$ holomorphic orbifold bundle $E$ over $M=[U / \mathbb{Z}_m]$, where $m \geq 2$. A local trivialization $\Theta : \widetilde{E} \rightarrow U \times \mathbb{C}^n$ is $\mathbb{Z}_m$-equivariant with respect to the following action
\begin{align*}
t(z;z_1,z_2,...,z_n)=(tz;t^{k_1}z_1,t^{k_2}z_2,...,t^{k_n}z_n),
\end{align*}
where $k_1,...,k_n$ are integers such that $k_1 \leq k_2 \leq ... \leq k_n \leq m$. We can take local holomorphic sections $f_1,...,f_n$ of $\widetilde{E}$ such that $\{f_1(x),...,f_n(x)\}$ is a basis of $(\widetilde{E})_x$ consisting of eigenvectors. Then, we can set
\begin{align*}
\Theta=(t^{-k_1}(t \cdot f_1),...,t^{-k_n}(t \cdot f_n)),
\end{align*}
where $t \cdot f_i(x)=t^{k_i}f_i(x)$.

Let now $\Phi$ be a Higgs field over $E$. In our example, $\Phi$ can be written with respect to the local chart $[U / \mathbb{Z}_m]$ as follows:
\begin{align*}
\Phi=(\phi_{ij})_{1 \leq i,j \leq n},
\end{align*}
where
\begin{align*}\tag{4}
\phi_{ij}=
\begin{cases}
z^{k_{i}-k_{j}} \hat{\phi}_{ij}(z^{m})\frac{dz}{z} & \mathrm{ if } k_i \geq k_j\\
0 & \mathrm{ if } k_i < k_j,
\end{cases}
\end{align*}
and $\hat{\phi}_{ij}$ are holomorphic functions on $\widetilde{E}$. We explain why $\Phi$ can in fact be written this way in the following two remarks.

\begin{rem}\label{605}
In general, $\Phi \in H^0(\mathrm{End}_0 (E)\otimes K)$ is $\mathbb{Z}_m$-equivariant, where $\mathrm{End}_0(E)$ denotes the traceless homomorphisms of $E$ and the action of $\mathbb{Z}_m$ on $\mathrm{End}_0 (E)\otimes K$ is conjugation. Under the conjugation action, we have
\begin{align*}
\phi_{ij}=
\begin{cases}
z^{k_{i}-k_{j}} \hat{\phi}_{ij}(z^{m})\frac{dz}{z} & \mathrm{ if } k_i \geq k_j\\
z^{k_{i}-k_{j}} \hat{\phi}_{ij}(z^{m})\frac{dz}{z} & \mathrm{ if } k_i < k_j.
\end{cases}
\end{align*}
If $k_i \leq k_j$, then $z^{k_{i}-k_{j}}$ is a negative number, which means possibly a meromorphic section, not holomorphic. Hence, we may define $\phi_{ij}$ as in (4).
\end{rem}

\begin{rem}\label{606}
In the next subsection, we construct the correspondence between an orbifold Higgs bundle and a parabolic Higgs bundle. Under this correspondence, $\Phi=(\phi_{ij})$ is a Higgs field, and the fact that $\Phi$ is a ``lower triangular matrix'' means that $\Phi$  preserves the filtration (cf. Definition 2.5), hence $\Phi$ is a well-defined parabolic Higgs field.
\end{rem}

\subsection{Orbifold Higgs bundle vs. Parabolic Higgs bundle}

We construct a parabolic Higgs bundle over an underlying surface from a given orbifold Higgs bundle and show that this construction is precisely a one-to-one correspondence. This provides that given any parabolic Higgs bundle over the underlying surface, we can recover the orbifold Higgs bundle. We discuss the local construction in detail for both the holomorphic bundle and the Higgs field; this local construction can be glued naturally. The construction follows very closely \cite{FuSt}, \cite{NaSt}, and we are adding one more condition on the Higgs field which should be a lower triangular matrix; see also p. 624 in \cite{NaSt}.

In this section, all parabolic structures are assumed to have rational weights.

\subsubsection{Holomorphic bundle.}
We briefly review the construction by M. Furuta and B. Steer for the holomorphic bundle (cf. \cite{FuSt}). Since we work on the local chart, let $E$ be a rank $n$ holomorphic orbifold bundle over the orbifold surface $M=[U / \mathbb{Z}_m]$, where $m \geq 2$, with local trivialization $\Theta : \widetilde{E} \rightarrow U \times C^n$. The local trivialization $\Theta$ is $\mathbb{Z}_m$-equivariant with respect to the action
\begin{align*}
t(z;z_1,z_2,...,z_n)=(tz;t^{k_1}z_1,t^{k_2}z_2,...,t^{k_n}z_n).
\end{align*}

\noindent Now we consider a bundle map $f\left( {{k}_{1}},\ldots ,{{k}_{n}} \right):U\backslash \left\{ x \right\}\times {{\mathbb{C}}^{n}}\to U\backslash \left\{ x \right\}\times {{\mathbb{C}}^{n}}$ defined by
\[f=\left( \begin{matrix}
   {{z}^{{{k}_{1}}}} & {} & {}  \\
   {} & \ddots  & {}  \\
   {} & {} & {{z}^{{{k}_{n}}}}  \\
\end{matrix} \right).\]
Let $\tilde{\Theta }=f{{\left( {{k}_{1}},\ldots ,{{k}_{n}} \right)}^{-1}}\Theta $. It is not hard to check that
\begin{equation}\tag{5}
  \tilde{\Theta }\left( t\cdot z \right)=\left( \begin{matrix}
   {{z}^{{{k}_{1}}-{{k}_{1}}}} & {} & {}  \\
   {} & \ddots  & {}  \\
   {} & {} & {{z}^{{{k}_{n}}-{{k}_{n}}}}  \\
\end{matrix} \right)\tilde{\Theta }\left( z \right)=\tilde{\Theta }\left( z \right).
\end{equation}
Hence, we define $E\left( {{k}_{1}},\ldots {{k}_{n}} \right)$ to be the holomorphic orbifold bundle by patching $E\left| _{U\backslash \left\{ x \right\}} \right.$ and $U\times {{\mathbb{C}}^{n}}$ via $\tilde{\Theta }=f{{\left( {{k}_{1}},\ldots {{k}_{n}} \right)}^{-1}}\Theta $. From Equation (5), we also know the isotropy representation is trivial, thus $E\left( {{k}_{1}},\ldots {{k}_{n}} \right)$ is a well-defined holomorphic bundle over the underlying space $U$. To define the filtration corresponding to the orbifold bundle $\left( E,\Theta  \right)$, we have to make another assumption on the numbers ${{k}_{i}}$: We say the local trivialization $\Theta $ is \emph{good}, if
	\[{{k}_{1}}\le {{k}_{2}}\le \ldots \le {{k}_{n}}.\]
Let $r$ be the number of distinct ${{k}_{i}}$ and let ${{\kappa }_{1}},\ldots ,{{\kappa }_{r}}$ be the respective multiplicities of each of those distinct numbers. We define the \emph{parabolic structure} on $F=E\left( {{k}_{1}},\ldots {{k}_{n}} \right)$ at a point $p$ by the following filtration
	\[{{F}_{p}}={{F}_{1}}\supset {{F}_{2}}\supset \ldots \supset {{F}_{r+1}}=\left\{ 0 \right\},\]
where ${{F}_{i}}=\overbrace{0\oplus \cdots \oplus 0}^{{{j}_{i-1}}}\oplus \overbrace{\mathbb{C}\oplus \cdots \oplus \mathbb{C}}^{r-{{j}_{i-1}}}$, with weight $\frac{{{k}_{{{j}_{s}}}}}{m}$ and ${{j}_{s}}={{\kappa }_{1}}+\ldots +{{\kappa }_{s}}$. Clearly, $F$ is a parabolic vector bundle over the underlying space.

\begin{thm}[Theorem 5.7 in \cite{FuSt}]\label{607}
The construction from $E$ to $F=E\left( {{k}_{1}},\ldots {{k}_{n}} \right)$ gives a bijective correspondence between isomorphism classes of holomorphic orbifold bundles with good trivialization $\left( E,\Theta  \right)$ and isomorphism classes of parabolic bundles $\left( F,\tilde{\Theta } \right)$.
\end{thm}

\begin{rem}
Recall that this correspondence is established assuming rational weights in the parabolic structure.
\end{rem}

\subsubsection{Higgs field.}

We now describe the correspondence for the Higgs fields for $\left( E,\Theta  \right)$ and $\left( F,\tilde{\Theta } \right)$. M. Furuta and B. Steer have constructed this correspondence in the rank 2 case. We construct the Higgs field for any rank in a similar way. The difference is that our construction of the Higgs field preserves the filtration of a parabolic bundle (Definition 2.5).

Recall that Equation (4) gives the local description of the orbifold Higgs field on $M=[U / \mathbb{Z}_m]$. Under the correspondence $E \to F$ described in \S 6.3.1, the corresponding Higgs field $\tilde{\Phi }$ over the underlying space $\tilde{U}$ should be the conjugation of $\Phi $ by the matrix $f\left( {{k}_{1}},\ldots {{k}_{r}} \right)$. Hence, we have
\begin{align*}
{{\tilde{\phi }}_{ij}} & ={{z}^{{{k}_{j}}-{{k}_{i}}}}{{\hat{\phi}(z^m) }_{ij}}\frac{dz}{z}\\
& = \begin{cases}
\frac{{{{\hat{\phi }}}_{ij}}\left( w \right)}{m w}dw & \mathrm{ if } k_i > k_j\\
0 & \mathrm{ if } k_i \leq k_j,
\end{cases}
\end{align*}
where we change the coordinate by $w={{z}^{m }}$ in the second equality above. From this calculation, it is implied that $\tilde{\Phi }=\left( {{{\tilde{\phi }}}_{ij}} \right)$ is a section with at most simple pole at $p$ in ${{H}^{0}}\left( \mathrm{End}_0\left( F \right)\otimes K\left( p \right) \right)$, in other words, $\tilde{\Phi }:F\to F\otimes K\left( p \right)$.

Since the trivialization $\Theta $ is good, that is, ${{k}_{1}}\le {{k}_{2}}\le \ldots \le {{k}_{r}}$, the orbifold Higgs field $\Phi $ is a lower-triangular matrix and the same is true for $\tilde{\Phi }$, thus $\tilde{\Phi }$ preserves the filtration. In conclusion, $\tilde{\Phi }$ is a parabolic Higgs field. It is now not hard to recover the orbifold Higgs field $\Phi $ from $\tilde{\Phi }$, giving a one-to-one correspondence. In summary, we have the following theorem:
\begin{thm}\label{608}
The above construction gives a bijective correspondence between isomorphism classes of holomorphic orbifold Higgs bundles with good trivialization $\left( E,\Theta ,\Phi  \right)$ and isomorphism classes of parabolic Higgs bundles $\left( F,\tilde{\Theta },\tilde{\Phi } \right)$.
\end{thm}

Given this theorem, the next step is to show that this correspondence holds in the semistable (resp. stable) case. The following theorem gives us a way to calculate the degree of an orbifold line bundle.

\begin{thm}[Kawasaki-Riemann-Roch \cite{Kawasaki}]\label{609}
If $E$ is a holomorphic orbifold line bundle over $[Y / \Gamma]$ with isotropy $\sigma_i^{\beta_i}$ at $x_i$, $1 \leq i \leq s$, then
\begin{align*}\tag{6}
\dim H^0(M,E) -\dim H^1(M,E)=1-g+\deg(E)-\sum_{i=1}^s \frac{\beta_i}{\alpha_i},
\end{align*}
where $\deg(E)$ is the degree of $E$ as an orbifold bundle over $[Y/\Gamma]$, which is a rational number, and $\alpha_i$ is the order of the group generated by $\sigma_i$.
\end{thm}
We remind the reader that $\deg(E)-\sum_{i=1}^s \frac{\beta_i}{\alpha_i}$ is an integer. Under the correspondence of Theorem 6.8, we have
\begin{equation*}
\deg(F)=\deg(E)-\sum_{i=1}^s \frac{\beta_i}{\alpha_i},
\end{equation*}
whereas Formula (6) implies that $\deg(E)= \pardeg (F)$. In conclusion, the equality of the degree provides the following proposition:
\begin{prop}[Proposition 5.9 in \cite{FuSt}]\label{610}
We have a bijective correspondence between isomorphism classes of holomorphic semistable (resp. stable) orbifold Higgs bundles with good trivialization $\left( E,\Theta ,\Phi  \right)$ and isomorphism classes of semistable (resp. stable) parabolic Higgs bundles $\left( F,\tilde{\Theta },\tilde{\Phi } \right)$.
\end{prop}
The above correspondence is also established assuming rational weights in the parabolic structure.

\begin{rem}\label{611}
For the special maximal parabolic $G$-Higgs bundles we are considering, we have seen that the defining parabolic bundle data for those can be reinterpreted as a direct sum of parabolic vector bundles (as is $E=V\oplus {{V}^{\vee }}$ in the $\mathrm{Sp(}2n\mathrm{,}\mathbb{R}\mathrm{)}$ case), thus the correspondence of Proposition \ref{610} can be used into our setting.
\end{rem}

\section{Topological Invariants of Maximal Parabolic $\mathrm{Sp}(2n,\mathbb{R})$-Higgs bundles}

Under the correspondence described in the last section, we use the $V$-cohomology to describe the topological invariants of the maximal $\mathrm{Sp}(2n,\mathbb{R})$-parabolic Higgs bundles. In \cite{Scott}, an explanation is provided on how to construct the fundamental group of the orbifold, and  on p. 426-427 of the same article the homology group is defined and the character is calculated. From this, we can clearly define the orbifold cohomology ${{H}^{1}}(M)$ in our case. The $V$-manifold we discuss in this section is exactly an orbifold. The terminology $V$-manifold comes from \cite{FuSt} and \cite{NaSt}. We first review some basic properties of a $V$-manifold.

\subsection{$V$-manifold}
The $V$-manifold is an orbifold. We review the definition of an orbifold and a holomorphic bundle over an orbifold from the last section.

Let $\widetilde{M}$ be a $k$-dimensional manifold with $s$-many marked points $x_1,...,x_s$. For each marked point, there is a linear representation $\sigma_i: \Gamma_i \rightarrow \mathrm{Aut}(\mathbb{R}^k)$ of a cyclic group ${{\Gamma }_{i}}=\left\langle {{\sigma }_{i}} \right\rangle $, $1 \leq i \leq s$, where $\Gamma_i$ acts freely on $\mathbb{R}^k \backslash \{0\}$ together with an atlas of coordinate charts
\begin{align*}
& \phi_i: U_i \rightarrow D^k / \sigma_i, & 1 \leq i \leq s;\\
& \phi_x: U_x \rightarrow D^k, & x \in \widetilde{M} \backslash \{x_1,...,x_s\}.
\end{align*}
An orbifold $M$ is obtained by gluing all local coordinate charts above, while $\widetilde{M}$ is the underlying manifold of $M$. We call $M$ the $V$-manifold in this section. The \emph{$V$-bundle} $E$ of rank $l$ over $M$ (or vector bundle over the $V$-manifold $M$) is defined locally on the charts as above with a collection of isotropy representations $\tau_i: F_i \rightarrow \mathrm{Aut}(C^l)$ and local trivializations $\theta_i : E|_{U_i} \rightarrow D^k \times C^l / \sigma_i \times \tau_i, 1 \leq i \leq s$. We are interested in the case when the manifold $\widetilde{M}$ is $X=Y / \Gamma$, where $Y$ is a compact Riemann surface and $\Gamma$ is a finite group acting effectively on $Y$. The following theorem gives us the condition when the $V$-manifold $M$ can be written in the form $[Y / \Gamma]$.

\begin{thm}[Theorem 1.2 in \cite{FuSt}]\label{701}
Let ${{\alpha }_{i}}$ denote the order of the cyclic group generated by the linear representation  ${{\sigma }_{i}}$ described above. Any compact oriented $V$-surface $M$, with $s \geq 3$ or $s=2$ and $\alpha_1 \neq \alpha_2$ if $g=0$, has the form $Y / \Gamma$, where $Y$ is a compact Riemann surface with genus $g$ and $\Gamma$ is a finite group acting effectively.
\end{thm}

\begin{thm}[Theorem 1.3 in \cite{FuSt}]\label{702}
If the compact oriented $V$-manifold $M$ has the form $Y / \Gamma$, then there is a bijective correspondence between isomorphism classes of complex $V$-bundles over $M$ and equivariant isomorphism classes of complex vector bundles over $Y$ with a $\Gamma$-action.
\end{thm}

We now define $V$-cohomology as follows. Recall that $M$ is a union of $Y \backslash \{x_1,...,x_s\}$ and $\coprod U_i$, where $U_i = D_i/ \mathbb{Z}_{\alpha_i}$, $1 \leq i \leq s$. Then $M_V$ is defined as the union of $Y \backslash \{x_1,...,x_s\}$ and $\coprod E \mathbb{Z}_{\alpha_i} \times_{\mathbb{Z}_{\alpha_i}}D_i$.

\begin{defn}\label{703}
The $V$-cohomology $H^{*}_V(M)$ is defined as the cohomology
\begin{align*}
H^{*}_V(M)=H^{*}(M_V).
\end{align*}
\end{defn}
The following theorem is the basic tool in order to calculate the cohomology group $H^{*}_V(M)$:
\begin{thm}[Theorem 2.2 in \cite{FuSt}]\label{704}
We have the following isomorphism about the first $V$-cohomology group
\begin{align*}
H^1_{V}(M,\mathbb{Z}) \cong H^1(M, \mathbb{Z}).
\end{align*}
\end{thm}

\subsection{Line V-bundles}
Let $M$ be an $V$-manifold. Line $V$-bundles are line bundles over the $V$-manifold $M$. Note that the $V$-manifold is also an orbifold. Thus the line $V$-bundles can be considered as line bundles over the orbifold.

\begin{defn}\label{705}
Under the tensor product, the topological isomorphism classes of line $V$-bundles form a group, which shall be denoted by $\mathrm{Pic}_V(M)$.
\end{defn}

The topological classification of the bundles on a $V$-Riemann surface is already done by M. Furuta and B. Steer \cite{FuSt}. Recall that there is a canonical line bundle $L_i$ at each point $x_i$ such that $L_i^{k_i}=\mathsf{\mathcal{O}}(x_i)$, and this bundle has isotropy $e^{2\pi \sqrt{-1} \frac{1}{k_i}}$ at each point $x_i$. Thus, if we have a $V$-bundle $L$ with isotropy $((\beta_1,k_1),\ldots,(\beta_s,k_s))$, we define the \textit{desingularization} $|L|$ to be \[|L|=L\otimes L_1^{-\beta_1}\otimes\ldots\otimes L_s^{-\beta_s}\]
and it turns out that these completely classify the line bundles topologically.

\begin{thm}[Proposition 1.4 in \cite{FuSt}]\label{706}
There is a bijective correspondence between isomorphism classes of complex V-bundles and isotropy classes $\sigma_1^{\beta_1},\ldots,\sigma_s^{\beta_s}$ over points $x_1,\dots,x_s$ respectively, as well as the first Chern class of a line bundle on the underlying manifold $\widetilde{M}$.  Thus
\[\mathrm{Pic}_V(M)=\mathrm{Pic}(\widetilde{M})\oplus \bigoplus_{i=1}^s \mathbb{Z}_{k_i}.\]
\end{thm}
The idea of the proof is the following: We recall that on an ordinary Riemann surface, two line bundles $L_1$ and $L_{2}$ are topologically equivalent, if and only if $c_1(L_1)=c_1(L_2)$. Now we can use the desingularization to define line bundles $|L_1|$ and $|L_2|$ over $\widetilde{M}$. The line $V$-bundles are equivalent if and only if $c_1(|L_1|)=c_1(|L_2|)$ and their isotropy classes coincide under some trivialization.

The class $(c_1(|L|),(\beta_1,k_1),\ldots,(\beta_s,k_s))$ is called the \textit{Seifert invariant} of this bundle. Note that this invariant depends on the local trivialization of each neighborhood.

We have more invariants if we would like to classify the holomorphic bundle instead of the topological bundle. From the classical Narasimhan-Seshadri correspondence, there is a correspondence between unitary representations of the fundamental group and rank $n$ polystable bundles $E$ with trivial first Chern class and trivial $c_2(E)\cdot c_1(L)^{n-1}$ with the polystable Higgs bundle, where $L$ is an ample line bundle on a K\"{a}hler manifold. I. Biswas and A. Hogadi in \cite{BiHo} generalized this correspondence for a compact orbifold of any dimension and any rank:

\begin{thm}[Theorem 1.2 in \cite{BiHo}]\label{707}
Let $M$ be a complex projective orbifold of dimension $n$ and $E$ a vector bundle over $X$ with $L$ an ample line bundle. Then $E$ is polystable with respect to $L$ if and only if it corresponds to a unitary representation of an orbifold line bundle.
\end{thm}

Thus, in particular, in the case of a line bundle over an orbifold Riemann surface, the stability condition is trivial. We know that $\mathrm{Pic}^{0}_{V}(M):=\mathrm{Hom}(\pi_V^1(M),\mathbb{C}^*)$. We see that in the case of $s$-many marked points, since $\mathbb{C}^*$ is commutative, it is
\[\mathrm{Pic}^0_V(M):=\mathrm{Hom}(\langle a_i,b_i, \sigma_i|\sigma_i^{k_i}=1,\prod_{k=1}^s \sigma_i=1\rangle,U(1))\cong(S^1)^{2g}\times\frac{\bigoplus_{i=1}^s\mathbb{Z}_{k_i}}{(1,1,\ldots,1)}.\]
We deduce that in our case, that is, considering the trivial parabolic structure with weight $\frac{1}{2}$ over each point in the divisor $D$, in other words, with $\mathbb{Z}_2$-isotropy at the $s$-many marked points on the genus $g$ Riemann surface $M$, we have the identification
\[\mathrm{Pic}^0_V(M)\cong(S^1)^{2g}\oplus \mathbb{Z}_2^{s-1}.\]
For bundles of higher degree, we can get a degree $0$ bundle by tensoring with a degree $-d$ bundle, thus this also reduces to the degree $0$ case as the stability condition is trivial. We finally imply the following:

\begin{prop}\label{708}
Let $X$ be a Riemann surface with genus $g$. Denote by $M$ the $V$-manifold with $s$-many marked points $x_1,...,x_s$, around which the isotropy group is $\mathbb{Z}_2$, such that $X$ is the underlying surface of $M$. Let $\pi: M \rightarrow X$ be the natural map. Given any line bundle $L$ over $X$, the line bundle $\pi^* L$ has $2^{2g+s-1}$ many square roots over $M$. In particular, let $K$ be the canonical line bundle over $X$. The line $V$-bundle $\pi^* K$ over $M$ has $2^{2g+s-1}$ many square roots over $M$.
\end{prop}
Let $L$ be a line bundle over $X$. Sometimes we abuse the notation $L$, to mean the corresponding line $V$-bundle $\pi^* L$ over $M$ in the rest of the paper.

\subsection{Calculations in orbifold cohomology}
We consider the following special $V$-manifold
\begin{align*}
& M=U_1 \bigcup U_2, & U_1 = X \backslash \{x_1,...,x_s\}, && U_2 = \coprod_{i=1}^s D / \mathbb{Z}_2,\\ & M_V=V_1 \bigcup V_2, & V_1 = X \backslash \{x_1,...,x_s\}, && V_2 = \coprod_{i=1}^s D\times_{\mathbb{Z}_2} E\mathbb{Z}_2,
\end{align*}
where $D$ is a disk around the punctures $x_{i}$ and $X$ is a compact Riemann surface of genus $g$. We only calculate the rank of the $V$-cohomology group $H^{*}_V(M)$ with coefficients $\mathbb{Z}_2$.

By the Mayer-Vietoris sequence, we have
\begin{align*}
0 \rightarrow H^0(M_V) \rightarrow H^0(V_1) \bigoplus H^{0}(V_2) \rightarrow H^0(V_1\bigcap V_2)\\
\xrightarrow[]{j_1}  H^1(M_V) \rightarrow H^1(V_1) \bigoplus H^{1}(V_2) \rightarrow H^1(V_1\bigcap V_2)\\
\xrightarrow[]{j_2} H^2(M_V) \rightarrow H^2(V_1) \bigoplus H^{2}(V_2) \rightarrow H^2(V_1\bigcap V_2).\\
\end{align*}

\begin{enumerate}
\item[(1)] Clearly, $V_1\bigcap V_2=\prod_{i=1}^s S^1$. We have
\begin{align*}
& \mathrm{rk}(H^0(V_1\bigcap V_2))=s,\\
& \mathrm{rk}(H^1(V_1\bigcap V_2))=s,\\
& \mathrm{rk}(H^2(V_1\bigcap V_2))=0.\\
\end{align*}

\item[(2)] For the cohomology group of $V_1=X \backslash \{p_1,...,p_s\}$ we check that
\begin{align*}
\mathrm{rk}(H^0(V_1))=1,\\
\mathrm{rk}(H^1(V_1))=2g+s-1,\\
\mathrm{rk}(H^2(V_1))=0.\\
\end{align*}

\item[(3)] We use the Leray spectral sequence to calculate the cohomology group of $V_2$. We have the following fibration
\begin{align*}
B\mathbb{Z}_2 \rightarrow D\times_{\mathbb{Z}_2} E\mathbb{Z}_2 \rightarrow D,
\end{align*}
where $B\mathbb{Z}_2$ is the classifying space of $\mathbb{Z}_2$ and $E\mathbb{Z}_2$ is the universal bundle over $B\mathbb{Z}_2$. By Leray spectral sequence, we have
    \begin{align*}
    H^{*}(B\mathbb{Z}_2,H^{*}(D)) \Rightarrow H^{*}(D\times_{\mathbb{Z}_2} E\mathbb{Z}_2).
    \end{align*}
We know
\begin{align*}
    H^{i}(D)=
    \begin{cases}
    \mathbb{Z}, \quad i=0\\
    0, \quad \mathrm{otherwise.}
    \end{cases}
    \end{align*}
    Hence, we have
    \begin{align*}
    H^{i}(D\times_{\mathbb{Z}_2} E\mathbb{Z}_2)=
    \begin{cases}
    \mathbb{Z}_2, \quad i=0,1,2\\
    0, \quad \mathrm{otherwise,}
    \end{cases}
\end{align*}
where $H^{i}(D\times_{\mathbb{Z}_2} E\mathbb{Z}_2)$ is the $\mathbb{Z}_2$-cohomology. Based on the calculation above, we have
\begin{align*}
    \mathrm{rk}(H^0(V_2))=s,\\
    \mathrm{rk}(H^1(V_2))=s,\\
    \mathrm{rk}(H^2(V_2))=s.\\
\end{align*}
\end{enumerate}

From the calculation of the ranks of the cohomology groups, we induce that the map $j_1$ in the Mayer-Vietoris sequence is injective, while the map $j_2$ is an isomorphism. Since $M_V$ is connected, it is implied that $\mathrm{rk}(H^{0}(M_V))=1$. The Mayer-Vietoris sequence now provides that
\begin{align*}
& \mathrm{rk}(H^1(M_V))=2g+s-1,\\
& \mathrm{rk}(H^2(M_V))=s.\\
\end{align*}

\subsection{Topological Invariants of Parabolic $\mathrm{Sp}(2n,\mathbb{R})$-Higgs Bundles in $\mathcal{M}_{par}^{max}(\mathrm{Sp}(2n,\mathbb{R}))$}
In this subsection, we study the topological invariants of parabolic $\mathrm{Sp}(2n,\mathbb{R})$-Higgs bundles in $\mathcal{M}_{par}^{max}(\mathrm{Sp}(2n,\mathbb{R}))$. For a parabolic Higgs bundle $(E,\Phi) \in \mathcal{M}_{par}^{max}(\mathrm{Sp}(2n,\mathbb{R}))$, the denominator of its parabolic weight $\alpha(x)$ is $2$, for every $x \in D$ (with the same notation as in Definition \ref{201}). In other words, the parabolic weight $\alpha(x)$ can be either $0$ or $\frac{1}{2}$. We want to remind the reader one more time that the denominator $2$ of the weight corresponds to the group action $\mathbb{Z}_2$ around the punctures in $D$; if the denominator is $n$, then the group action is $\mathbb{Z}_n$. To study the topological invariants of parabolic Higgs bundles, we consider parabolic Higgs bundles as Higgs $V$-bundles under the correspondence we studied in \S 6. Thus we take the topological invariants of the corresponding Higgs $V$-bundles to define invariants for parabolic Higgs bundles. In this subsection, we slightly abuse terminology between parabolic Higgs bundles and Higgs $V$-bundles; when we discuss topological invariants this should always refer to the Higgs $V$-bundles.

We first study the $\mathrm{Sp}(4,\mathbb{R})$-case. Recall that the definition of a parabolic $\mathrm{Sp(4}\mathrm{,}\mathbb{R}\mathrm{)}$-Higgs bundle over $X$ with divisor $D=\{x_1,...,x_s\}$ involves a pair $(E,\Phi)$, where $E=V\oplus {{V}^{\vee }}$ is a rank 4 parabolic vector bundle over $X$ and $\Phi =\left( \begin{matrix}
   0 & \beta   \\
   \gamma  & 0  \\
\end{matrix} \right): E \rightarrow E \otimes K(D)$ is a parabolic Higgs field. From Proposition 5.4, the maximal parabolic degree of the parabolic vector bundle $V$ is $2g-2+s$. In this maximal case, the proof of Proposition 5.4 implies that $\gamma$ is an isomorphism.

Fix a square root $L_0$ of $K(D)$ (as a line $V$-bundle) and define $W:=V \otimes L^{-1}_0$. Clearly, we have
\begin{align*}
& c=\gamma \otimes 1_{L^{-1}_0} : W=V \otimes L^{-1}_0 \rightarrow V^{\vee} \otimes K(D) \otimes L^{-1}_0=W^{\vee},\\
& \phi=(\beta \otimes 1_{L_0}) \circ (\gamma \otimes 1_{L^{-1}_0}): W=V\otimes L^{-1}_0 \rightarrow V \otimes L^{\frac{3}{2}}_0=W\otimes K(D)^2,
\end{align*}
where the first map $c$ is an isomorphism. Given a parabolic $\mathrm{Sp}(4,\mathbb{R})$-Higgs bundle with maximal degree $(E,\Phi)$, we consider the associated triple $(W,c,\phi)$ when we discuss the topological invariants. It is easy to check that the parabolic structure of $E$ uniquely determines the parabolic structure of $W$. Thus we use the same notation $\alpha$ for the parabolic structure of $W$.

From the correspondence between the orbifold bundle ($V$-bundle) and the parabolic bundle we studied in \S 6, the parabolic Higgs bundle $E=V\oplus {{V}^{\vee }}$ over $X$ is equivalent to a Higgs $V$-bundle over $M$, where $M$ is the $V$-manifold with $s$-many marked points $x_1,...,x_s$, around which the isotropy group is $\mathbb{Z}_2$, and $X$ is the underlying surface of $M$.

Under this correspondence, $c$ induces a quadratic form on the $V$-bundle $W$. Hence, the structure group of $W$ is $\mathrm{O}(2,\mathbb{C})$. Also, note that the $\mathrm{Sp}(4,\mathbb{R})$-Higgs bundle $E$ has the real structure. More precisely, $E$ can be written as $E=E_{\mathbb{R}} \otimes \mathbb{C}$, where $E_{\mathbb{R}}$ is a real vector bundle. Similarly, we can write $W$ as $W_{\mathbb{R}}\otimes \mathbb{C}$. Therefore the structure group of $W_{\mathbb{R}}$ is $\mathrm{O}(2)$. In parallel to the definition of a Stiefel-Whitney class, the corresponding class $w_1$ of $W_{\mathbb{R}}$ in $H_V^1(M, \mathbb{Z}_2)$ is a well-defined topological invariant for $W$. By Theorem 7.7 and the calculations in \S 7.3, we deduce that the number of different elements in $H_V^1(M, \mathbb{Z}_2)$ is $2^{2g+s-1}$.

An alternative description of this cohomology group is given by the fundamental group
\begin{align*}
H^1_V(M, \mathbb{Z}_2)=\mathrm{Hom}(\pi_V^1(M),\mathbb{Z}_2),
\end{align*}
where $\pi^{1}_V(M)$ is the $V$-fundamental group with presentation
\begin{align*}
\pi_V^1(M)=\{a_1,b_1,...,a_g,b_g,\sigma_1,...,\sigma_s \quad | \quad \sigma_1...\sigma_s[a_1,b_1]...[a_g,b_g]=1,\sigma_i^2=1, 1 \leq i \leq s\},
\end{align*}
where $a_1,b_1,...,a_g,b_g$ are generators of the underlying surface $X$ and $\sigma_i$ are represented by small loops around the points $x_i$, $1 \leq i \leq s$.

We discuss the topological invariants for maximal parabolic $\mathrm{Sp(4}\mathrm{,}\mathbb{R}\mathrm{)}$-Higgs bundles based on the cohomology group $H^1_V(M,\mathbb{Z}_2)$. Let $w_1 \in H^1_V(M,\mathbb{Z}_2)$ and $w_2 \in H^2_V(M,\mathbb{Z}_2)$. We distinguish the following cases:
\begin{enumerate}
\item[(1)] If $w_1 \neq 0$, every pair $(w_1,w_2)$ is a topological invariant for $W_{\mathbb{R}}$. Thus the pair $(w_1,w_2)$ can be considered as the topological invariant of $W$. The number of topological invariants in this case is $2^s(2^{2g+s-1}-1)$.
\item[(2)] If $w_1=0$, then the structure group can be reduced to $\mathrm{SO}(2, \mathbb{C}) \subset \mathrm{O}(2,\mathbb{C})$. From the identification $\mathrm{SO}(2,\mathbb{C}) \cong \mathbb{C}^*$, $W$ can be decomposed as the direct sum $W=L \bigoplus L^{\vee}$. Now, stability for the map $\phi:W \rightarrow W \otimes K(D)^2$, provides the existence of a non-trivial holomorphic map $L \rightarrow L^{\vee}\otimes K(D)^2$, therefore it holds necessarily that $\mathrm{pardeg}(L) \leq 2g-2+s$.
    \begin{enumerate}
    \item[a.] If $\mathrm{pardeg}(L)\neq 2g-2+s$, then every value of the parabolic degree gives a topological invariant and there are at least $2g-2+s$ different values. But the parabolic degree is not enough to give all possible topological invariants of $L$. Recall from the definition of parabolic degree that the parabolic degree $``\mathrm{pardeg}"$ can be written as the sum of the classical degree $``\deg"$ and the weight $``w"$, which is the sum of the weights over each point $x \in D$. Note that $w$ is uniquely determined by the parabolic structure of $L$. If we fix the parabolic structure of $L$, there are $2g-2+s$ many choices for the part of the classical degree $``\deg"$. At the same time, we have $2^s$ many choices for the parabolic structures of $L$. Thus the number of topological invariants in this case is $2^s(2g-2+s)$.
    \item[b.] If $\mathrm{pardeg}(L)=2g-2+s$, then $L^2 \cong K(D)^2$. This describes parabolic $\mathrm{Sp}(4,\mathbb{R})$-Higgs bundles $\left( E=V\oplus {{V}^{\vee }},\Phi  \right)$ with $V=N\oplus {{N}^{\vee }}K\left( D \right)$, for a line bundle  $N=K{{\left( D \right)}^{\frac{3}{2}}}$. Thus, square roots of $K\left( D \right)$ parameterize components containing such Higgs bundles, and this contributes to at least $2^{2g+s-1}$ topological invariants by Proposition \ref{708}.
    \end{enumerate}
\end{enumerate}

The discussion above implies our main theorem:

\begin{thm}\label{709}
The moduli space $\mathcal{M}_{par}^{max}(\mathrm{Sp}(4,\mathbb{R}))$ of maximal polystable parabolic $\mathrm{Sp}\left( 4,\mathbb{R} \right)$-Higgs bundles over a compact Riemann surface $X$ of genus $g$ with a divisor $D$ of $s$-many distinct points on $X$, such that $2g-2+s>0$, has at least $(2^s+1)2^{2g+s-1}+2^s(2g-3+s)$ connected components.
\end{thm}

For  $n\ge 3$, the structure group of the $V$-bundle $W$ above is $\mathrm{O}( n,\mathbb{C})$ and the classification of $\mathrm{O}(n,\mathbb{C})$-bundles does not provide the extra invariant $\mathrm{pardeg}(L)$ in this case. Moreover, for every $n\ge 1$ in general, there are $2^{2g+s-1}$ connected components of the moduli space ${{\mathsf{\mathcal{M}}}_{par}^{max}}\left( \mathrm{Sp}\left( 2n,\mathbb{R} \right) \right)$ parameterized by the square roots of the canonical line bundle $K(D)$ (the parabolic Teichm{\"u}ller components). This provides the following:

\begin{thm}\label{710}
The moduli space ${{\mathsf{\mathcal{M}}}_{par}^{\mathrm{max}}}( \mathrm{Sp}\left( 2,\mathbb{R} \right))$ of maximal polystable parabolic $\mathrm{Sp}\left( 2,\mathbb{R} \right)$-Higgs bundles over a compact Riemann surface $X$ of genus $g$ with a divisor of $s$-many distinct points on $X$, such that $2g-2+s>0$, has at least $2^{2g+s-1}$ connected components and the moduli space  ${{\mathsf{\mathcal{M}}}_{par}^{\mathrm{max}}}\left( \mathrm{Sp}\left( 2n,\mathbb{R} \right) \right)$ for $n \ge 3$ has at least $(2^s+1)2^{2g+s-1}$ connected components.
\end{thm}

\subsection{Topological Invariants of Parabolic $\mathrm{Sp}(2n,\mathbb{R})$-Higgs Bundles in $\mathcal{M}_{par}^{max,\alpha}(\mathrm{Sp}(2n,\mathbb{R}))$}
In this subsection, we study the topological invariants of parabolic $\mathrm{Sp}(2n,\mathbb{R})$-Higgs bundles in $\mathcal{M}_{par}^{max,\alpha}(\mathrm{Sp}(2n,\mathbb{R}))$, where $\alpha$ is a given parabolic structure. Note that the parabolic structure was not fixed for parabolic Higgs bundles in $\mathcal{M}_{par}^{max}(\mathrm{Sp}(2n,\mathbb{R}))$ earlier. In this subsection, all parabolic $\mathrm{Sp}(2n,\mathbb{R})$-Higgs bundles in $\mathcal{M}_{par}^{max,\alpha}(\mathrm{Sp}(2n,\mathbb{R}))$ are assumed to have the same fixed parabolic structure $\alpha$. More precisely, they have the same filtration over each $x \in D$ with the same weight $\alpha(x)$, for every $x \in D$. The moduli space $\mathcal{M}_{par}^{max,\alpha}(\mathrm{Sp}(2n,\mathbb{R}))$ is a subspace of $\mathcal{M}_{par}^{max}(\mathrm{Sp}(2n,\mathbb{R}))$, studied in \S 7.4.

We shall still deal with the case of parabolic $\mathrm{Sp}(4,\mathbb{R})$-Higgs bundles. With the same notation as we used in \S 7.4, a parabolic $\mathrm{Sp}(4,\mathbb{R})$-Higgs bundle $(E=V \oplus V^{\vee},\Phi =\left( \begin{matrix}
   0 & \beta   \\
   \gamma  & 0  \\
\end{matrix} \right))$ with maximal degree corresponds to a triple $(W,c,\phi)$, where $c$ is an isomorphism and $\phi$ is a parabolic $K(D)^2$-twisted Higgs field. In \S 7.4, we study the topological invariants of $(W,c,\phi)$, which gives the topological invariants of parabolic Higgs bundle $(E,\Phi)$. Now we use the same approach to study the topological invariants of parabolic $\mathrm{Sp}(4,\mathbb{R})$-Higgs bundles with a given parabolic structure $\alpha$.

By considering the real structure of $W=W_{\mathbb{R}}\otimes \mathbb{C}$, the first and second $V$-cohomology $H^1_V(M,\mathbb{Z}_2)$, $H^2_V(M,\mathbb{Z}_2)$ are considered as our topological invariants. Similar to the classical case, there is natural map $\mathrm{Pic}_V(M) \rightarrow H^2_V(M,\mathbb{Z}_2)$. By the calculation in \S 7.3, we know that $\mathrm{rk }H^2_V(M,\mathbb{Z}_2)=s$. Thus the total number of elements in $H^2_V(M,\mathbb{Z}_2)$ is $2^s$. In fact, there is a one-to-one correspondence between $H^2_V(M,\mathbb{Z}_2)$ and the torsion points in $\mathrm{Pic}_V(M)$. Note that, for $M_{V}$, the exact sequence
\begin{align*}
0 \rightarrow \mathbb{Z} \rightarrow O_{M_V} \rightarrow O^*_{M_V} \rightarrow 0,
\end{align*}
provides that
\begin{align*}
H^1(M_V, O_{M_V}) \rightarrow H^1(M_V, O^*_{M_V}) \xrightarrow{\varpi}  H^2_V(M, \mathbb{Z}).
\end{align*}
The first cohomology $H^1(M_V, O^*_{M_V})$ is exactly the $V$-Picard group $\mathrm{Pic}_V(M)$. Therefore, there is a morphism
\begin{align*}
\varpi:\mathrm{Pic}_V(M) \rightarrow  H^2_V(M, \mathbb{Z}).
\end{align*}
By taking the $\mathbb{Z}_2$-coefficient, it is easy to see that the morphism $\varpi$ induces the isomorphism
\begin{align*}
\bigoplus_{i=1}^{s} \mathbb{Z}_2 \cong H^2_V(M, \mathbb{Z}).
\end{align*}
This gives the one-to-one correspondence between $H^2_V(M,\mathbb{Z}_2)$ and the torsion points in $\mathrm{Pic}_V(M)$. In our case, this correspondence is precisely between the second $V$-cohomology $H^2_V(M,\mathbb{Z}_2)$ and parabolic structures of $W$. Thus fixing a parabolic structure $\alpha$ is equivalent to fixing an element in the second $V$-cohomology $H^2_V(M,\mathbb{Z}_2)$.

Now suppose that $w_1$ is trivial. The parabolic bundle $W$ can be written as the sum of parabolic line bundles $L \oplus L^{\vee}$ such that $0 \leq \mathrm{pardeg} (L) \leq 2g-2+s$. As we discussed in {\rm Case (2a)} in \S 7.4, if we fix a parabolic structure $\alpha$ and assume that $\mathrm{pardeg} (L) \neq 2g-2+s$, the parabolic degree is the only topological invariant. In other words, the parabolic structure of $L$ is uniquely determined by that of $W$. We still use the notation $\alpha$ for the parabolic structure of $L$.

In this paragraph, we consider $\alpha$ as the parabolic structure of the line bundle $L$. If $\mathrm{pardeg} (L) =2g-2+s$, the topological invariants are described by the square roots of $K(D)^2$. In other words, $L$ is a square root of $K(D)^2$, however, its parabolic structure is not arbitrary. Only a line bundle $L$ with \emph{even parabolic structure} can be a square root of $K(D)^2$. We will explain this property and define \emph{even} (resp. \emph{odd}) parabolic structure in this paragraph. As we discussed in \S 7.2, there are $2^{2g+s-1}$ many square roots, which correspond to $\mathrm{Hom}(\pi_V^1(M),\mathbb{Z}_2)$. Recall that
\begin{align*}
\pi_V^1(M)=\{a_1,b_1,...,a_g,b_g,\sigma_1,...,\sigma_s \quad | \quad \sigma_1...\sigma_s[a_1,b_1]...[a_g,b_g]=1,\sigma_i^2=1, 1 \leq i \leq s\},
\end{align*}
where the $\sigma_i$ describe the monodromy around the point $x_i$. By the correspondence between line $V$-bundles and parabolic line bundles, the monodromy around $x_i$ corresponds to the weight of the corresponding parabolic line bundle over the point $x_i$. Thus fixing a parabolic structure $\alpha$ is equivalent to fixing the monodromy around $x_i$, $1 \leq i \leq s$. However, not every parabolic structure corresponds to a well-defined element in $\mathrm{Hom}(\pi_V^1(M),\mathbb{Z}_2)$. Indeed, the relation $\sigma_1...\sigma_s[a_1,b_1]...[a_g,b_g]=1$ implies that the number of nontrivial $\sigma_i$ is even. Equivalently, if the cardinality of the set
\begin{align*}
\{x \in D \mathrm{ } | \mathrm{ } \alpha(x)=\frac{1}{2}\}
\end{align*}
is even, then the parabolic structure corresponds to an element in $\mathrm{Hom}(\pi_V^1(M),\mathbb{Z}_2)$, and such a parabolic structure could be a choice for the square root of $K(D)^2$. Thus we say that the parabolic structure $\alpha$ is \emph{even} (resp. \emph{odd}) if the cardinality of the set
\begin{align*}
\{x \in D \mathrm{ } | \mathrm{ } \alpha(x)=\frac{1}{2}\}
\end{align*}
is even (resp. odd).

Here is an easy example. Let $s=1$ and $D=\{x\}$. The parabolic structure $\alpha$ of a parabolic line bundle $L$ is uniquely determined by the weight $\alpha(x)$, which is either $0$ or $\frac{1}{2}$. In this case, the parabolic structure is even if and only if the weight $\alpha(x)=0$, and the parabolic structure is odd if and only if $\alpha(x)=\frac{1}{2}$.

Based on the above discussion, the topological invariants of parabolic $\mathrm{Sp}(4,\mathbb{R})$-Higgs bundles with a fixed parabolic structure $\alpha$ are given as follows.
\begin{enumerate}
\item[(1)] If $w_1 \neq 0$, each element $w_1 \in H_V^1(M,\mathbb{Z}_2)$ is a topological invariant. The number of topological invariants in this case is $2^{2g+s-1}-1$.
\item[(2)] Suppose that $w_1=0$.
    \begin{enumerate}
    \item[a.] If $\mathrm{pardeg}(L) \neq 2g-2+s$, the topological invariants are given by the parabolic degree. The number of topological invariants in this case is $2g-2+s$.
    \item[b.] If $\mathrm{pardeg}(L)=2g-2+s$, the topological invariants are given by the square roots. Note that the parabolic degree of any square root is an integer. In other words, the number of points with nontrivial monodromy is even.
        \begin{enumerate}
        \item[$\mu.$] If the parabolic structure $\alpha$ is even, the number of topological invariants is $2^{2g}$.
        \item[$\nu.$] If the parabolic structure $\alpha$ is odd, the square roots are not well-defined.
        \end{enumerate}
    \end{enumerate}
\end{enumerate}

The discussion above implies the following proposition.

\begin{prop}\label{711}
Let $X$ be a smooth Riemann surface of genus $g$ and let $D$ be a reduced effective divisor of $s$ many points on $X$, such that $2g-2+s>0$. Consider the moduli space $\mathcal{M}_{par}^{max,\alpha}(\mathrm{Sp}(2n,\mathbb{R}))$ of maximal polystable parabolic $\mathrm{Sp}(2n,\mathbb{R})$-Higgs bundles, where $\alpha$ is a given parabolic structure, which is fixed for all Higgs bundles in the moduli space, this means, the parabolic Higgs bundles have the same filtration over each $x \in D$ with the same weight $\alpha(x)$, for every $x \in D$. Then,
\begin{enumerate}
\item[$\mathrm{i}.$] If $\alpha$ is even, the moduli space $\mathcal{M}_{par}^{max,\alpha}(\mathrm{Sp}(4,\mathbb{R}))$ has at least $2^{2g+s-1}+(2g-3+s)+2^{2g}$ connected components.
\item[$\mathrm{ii}$] If $\alpha$ is odd, the moduli space $\mathcal{M}_{par}^{max,\alpha}(\mathrm{Sp}(4,\mathbb{R}))$ has at least $2^{2g+s-1}+(2g-3+s)$ connected components.
\item[$\mathrm{iii}.$] If $\alpha$ is even, the moduli space $\mathcal{M}_{par}^{max,\alpha}(\mathrm{Sp}(2,\mathbb{R}))$ has at least $2^{2g}$ connected components, and the moduli space $\mathcal{M}_{par}^{max,\alpha}(\mathrm{Sp}(2n,\mathbb{R}))$ has at least $2^{2g+s-1}+2^{2g}$ connected components.
\item[$\mathrm{iv}.$] If $\alpha$ is odd, there are no maximal polystable parabolic $\mathrm{Sp}(2,\mathbb{R})$-Higgs bundles with fixed parabolic structure $\alpha$, and the moduli space $\mathcal{M}_{par}^{max,\alpha}(\mathrm{Sp}(2n,\mathbb{R}))$ has at least $2^{2g+s-1}$ many connected components.
\end{enumerate}
\end{prop}

\section{Other Lie groups}

The topological invariants and component count method developed for the case $G=\mathrm{Sp}(2n,\mathbb{R})$ in the previous section hint towards counting the minimum number of maximal components of moduli of polystable parabolic $G$-Higgs bundles also in other cases in which ${G}/{H}\;$ is a Hermitian symmetric space. We directly adapt the treatment followed by the authors in \cite{BrGPGHermitian} in the non-parabolic case. We will restrict to the cases when the bounded symmetric domain corresponding to the Hermitian symmetric space ${G}/{H}\;$ is of tube type. For the classical semisimple Lie groups this means we will be interested in the groups $\mathrm{SU}(n,n)$, $\mathrm{S}{{\mathrm{O}}^{*}}(2n)$ for even integer $n$, $\mathrm{S}{{\mathrm{O}}_{0}}(2,n)$ and $E_{7}^{-25}$ (cf. \cite{BrGPGHermitian} for a more detailed description).

In the sequel, $\left( X,D \right)$ will always denote a compact Riemann surface $X$ of genus $g$ together with a divisor $D:=\left\{ {{x}_{1}},\ldots ,{{x}_{s}} \right\}$ of $s$-many distinct points on $X$, assuming that $2g-2+2s>0$.

\subsection{$G=\mathrm{SU}(n,n)$}
\begin{defn}\label{801}
A \emph{parabolic $\mathrm{SU}(n,n)$-Higgs bundle} over $\left( X,D \right)$ is a parabolic Higgs bundle $\left( E,\Phi  \right)$, such that
\begin{enumerate}
  \item $E=V\oplus W$, where $V$ and $W$ are parabolic vector bundles of rank $n$ with $\pardeg V=-\pardeg W$.
  \item $\Phi =\left( \begin{matrix}
   0 & \beta   \\
   \gamma  & 0  \\
\end{matrix} \right):E\to E\otimes K\left( D \right)$, where $\beta :W\to V\otimes K\left( D \right)$ and $\gamma :V\to W\otimes K\left( D \right)$ are parabolic morphisms.
\end{enumerate}
\end{defn}

A parabolic Toledo invariant for a parabolic $\mathrm{SU}(n,n)$-Higgs bundle is defined by $\tau =\pardeg V=-\pardeg W$ and similarly to the proof of Proposition 5.4 one can establish the Milnor-Wood bound
\[\left| \tau  \right|\le n\left( g-1+\frac{s}{2} \right).\]
Maximality for the Toledo invariant provides that $\gamma :V\to W\otimes K\left( D \right)$ is a parabolic isomorphism. This, together with the condition $\det W={{\left( \det V \right)}^{-1}}$ for the corresponding $V$-bundles $V,W$ imply that
	\[{{\left( \det W \right)}^{2}}\simeq {{\left( {{\left( K\left( D \right) \right)}^{\vee }} \right)}^{n}}.\]
Choosing a square root
${{L}_{0}}$ of $K\left( D \right)$ and defining $\tilde{W}=W\otimes {{L}_{0}}$, we have ${{\left( \det \tilde{W} \right)}^{2}}\simeq \mathsf{\mathcal{O}}$. Therefore, the topological invariant for $\det \tilde{W}$ is defined by the choices of a square root of the trivial line $V$-bundle, which can take ${{2}^{2g+s-1}}$ different values. We deduce the following:

\begin{thm}\label{802}
The minimal number of connected components of the moduli spaces of parabolic Higgs bundles $\mathsf{\mathcal{M}}_{par}^{max}\left( \mathrm{SU}(n,n)\right)$, $\mathsf{\mathcal{M}}_{par}^{max,\alpha}\left( \mathrm{SU}(n,n)\right)$ is given as follows

\begin{center}
\begin{tabular}{|c|c|}
\hline
Moduli Space $\mathcal{M}$ & $\#{{\pi }_{0}}\left( \mathsf{\mathcal{M}}\right)$ \\
    \hline\hline
    $\mathsf{\mathcal{M}}_{par}^{max}\left( \mathrm{SU}(n,n)\right)$   &  \small{${{2}^{2g+s-1}}$}  \\
    \hline
    $\mathsf{\mathcal{M}}_{par}^{max,\alpha}\left( \mathrm{SU}(n,n)\right)$, $\alpha$ is even   & \small{$2^{2g}$}  \\
    \hline
    $\mathsf{\mathcal{M}}_{par}^{max,\alpha}\left( \mathrm{SU}(n,n)\right)$, $\alpha$ is odd   & \small{$-$} \\
    \hline
\end{tabular}
\end{center}
\end{thm}

\begin{rem}\label{803}
The preceding analysis coincides with the analysis for $\mathrm{Sp}(2,\mathbb{R})\simeq \mathrm{SU}(1,1)$. Note, however, that for $n\ne 1$ there are no Teichm{\"u}ller components, since $\mathrm{SU}(n,n)$ is not a split real form.
\end{rem}

\subsection{$\mathrm{S}{{\mathrm{O}}^{*}}(2n)$, for $n$ even}

\begin{defn}\label{804}
A \emph{parabolic $\mathrm{S}{{\mathrm{O}}^{*}}(2n)$-Higgs bundle} over $\left( X,D \right)$ for $n=2m$ is a parabolic Higgs bundle $\left( E,\Phi  \right)$, such that
\begin{enumerate}
  \item $E=V\oplus {{V}^{\vee }}$, where $V$ is a parabolic vector bundle of rank $n$, and
  \item $\Phi =\left( \begin{matrix}
   0 & \beta   \\
   \gamma  & 0  \\
\end{matrix} \right):E\to E\otimes K\left( D \right)$, where $\beta :{{V}^{\vee }}\to V\otimes K\left( D \right)$ and $\gamma :V\to {{V}^{\vee }}\otimes K\left( D \right)$ are skew-symmetric parabolic morphisms.
\end{enumerate}
\end{defn}
A parabolic Toledo invariant for a parabolic $\mathrm{S}{{\mathrm{O}}^{*}}(2n)$-Higgs bundle is defined by $\tau =par\deg V$, for which:
\[\left| \tau  \right|\le n\left( g-1+\frac{s}{2} \right).\]
Again the maximal case for $\tau$ will imply that $\gamma$ is an isomorphism and for a fixed square root ${{L}_{0}}$  of $K\left( D \right)$, and $\tilde{W}={{V}^{\vee }}\otimes {{L}_{0}}$, the homomorphism
\[\omega :=\gamma \otimes {{I}_{L_{0}^{\vee }}}:{{\tilde{W}}^{\vee }}\to \tilde{W}\]
is a skew-symmetric isomorphism defining a symplectic structure on the $V$-bundle $\tilde{W}$, in other words, $\left( \tilde{W},\omega  \right)$ is an $\mathrm{Sp}(2n,\mathbb{C})$-holomorphic $V$-bundle. Thus, the moduli space of maximal parabolic $\mathrm{S}{{\mathrm{O}}^{*}}(2n)$-Higgs bundles is homeomorphic to the moduli space of principal ${{H}^{\mathbb{C}}}$-bundles for $H\simeq \mathrm{Sp}(n)$, and $\mathrm{Sp}(n)$ is simply connected. Note that the moduli space of symplectic vector bundles is connected \cite{Rama}. Thus if we fix a parabolic structure $\alpha$ of $E$, i.e. a parabolic strucutre of $\tilde{W}$, the moduli space of pairs $\left( \tilde{W},\omega  \right)$, where $\tilde{W}$ is of parabolic structure $\alpha$, is connected. The above discussion provides the following theorem.

\begin{thm}\label{805}
The moduli space $\mathsf{\mathcal{M}}_{par}^{max}\left( \mathrm{S}{{\mathrm{O}}^{*}}(2n)\right)$ of maximal polystable parabolic $\mathrm{S}{{\mathrm{O}}^{*}}(2n)$-Higgs bundles has at least $2^s$ many connected components. The moduli space $\mathsf{\mathcal{M}}_{par}^{max,\alpha}\left( \mathrm{S}{{\mathrm{O}}^{*}}(2n)\right)$ has at least one connected component.
\end{thm}

\subsection{$\mathrm{S}{{\mathrm{O}}_{0}}(2,n)$}

\begin{defn}\label{806}
A \emph{parabolic $\mathrm{S}{{\mathrm{O}}_{0}}(2,n)$-Higgs bundle} over $\left( X,D \right)$ is a parabolic Higgs bundle $\left( E,\Phi  \right)$, such that
\begin{enumerate}
  \item $E=V\oplus W$, where $V=L\oplus {{L}^{\vee }}$ for a parabolic line bundle $L$ and $W$ corresponds to a rank $n$ orthogonal $V$-bundle.
  \item $\Phi =\left( \begin{matrix}
   0 & 0 & \beta   \\
   0 & 0 & \gamma   \\
   -{{\gamma }^{t}} & -{{\beta }^{t}} & 0  \\
\end{matrix} \right):E\to E\otimes K\left( D \right)$, where $\beta :W\to L\otimes K\left( D \right)$ and $\gamma :W\to {{L}^{\vee }}\otimes K\left( D \right)$ are parabolic morphisms.
\end{enumerate}
\end{defn}
A parabolic Toledo invariant for a parabolic $\mathrm{S}{{\mathrm{O}}_{0}}(2,n)$-Higgs bundle is defined by $\tau =par\deg L$ and a Milnor-Wood bound is described by
\[\left| \tau  \right|\le 2g-2+s.\]
Maximality for the Toledo invariant provides that $\gamma :V\to {{L}^{\vee }}\otimes K\left( D \right)$ has maximal rank one at all points and hence is surjective. Define $F=\ker \gamma $ and consider the short exact sequence
	\[0\to F\to W\to {{L}^{\vee }}\otimes K\left( D \right)\to 0.\]
Then the sequence splits and $F$ inherits an
$\mathrm{O}\left( n-1,\mathbb{C} \right)$-structure. Consider the line bundle ${{L}_{0}}:={{L}^{\vee }}\otimes K\left( D \right)$. From the exact sequence we deduce that ${{L}_{0}}\otimes \det F\simeq \mathsf{\mathcal{O}}$, hence $L_{0}^{2}\simeq \mathsf{\mathcal{O}}$. This, in turn, implies that ${{L}^{2}}\simeq {{\left( K\left( D \right) \right)}^{2}}$. From this point on, we further distinguish two cases:

\paragraph{\textit{Case 1: $n\ge 4$.}} In this case, the only topological invariants we obtain are the Stiefel-Whitney classes for the $\mathrm{O}\left( n-1,\mathbb{C} \right)$-bundle. This provides a minimum of ${{2}^{s}}\cdot {{2}^{2g+s-1}}$ connected components for $\mathsf{\mathcal{M}}_{par}^{max }\left( \mathrm{S}{{\mathrm{O}}_{0}}(2,n) \right)$.

\paragraph{\textit{Case 2: $n=3$.}} In this case, $F$ is an $\mathrm{O}\left( 2,\mathbb{C} \right)$-bundle and the treatment is similar to the $\mathrm{Sp}(4,\mathbb{R})$-case. There is a distinguished component for every value of $\left( w_1,w_2 \right)$, for $w_1 \ne 0$, where $w_1 \in H^1_V(M,\mathbb{Z}_2)$ and $w_2 \in H^2_V(M,\mathbb{Z}_2)$; this provides at least ${{2}^{s}}\left( {{2}^{2g+s-1}}-1 \right)$ connected components.

For $w_1=0$, there is a decomposition $F=M\oplus {{M}^{-1}}$ for a line $V$-bundle $M$. As in the case of $\mathrm{Sp}(4,\mathbb{R})$, one can show that there is a non-trivial holomorphic map $M\to {{\left( K\left( D \right) \right)}^{2}}$, which provides that $0\le \mathrm{pardeg} (M)\le 4g-4+2s$. For each value of the degree $\mathrm{pardeg} (M)<4g-4+2s$, there is a distinguished connected component for each fixed parabolic structure $\alpha$. Note here that in contrast to the $\mathrm{Sp}(4,\mathbb{R})$-case, when $\mathrm{pardeg} (M)=4g-4+2s$, there is an isomorphism $M\simeq {{\left( K\left( D \right) \right)}^{2}}$. Thus, there are no further invariants coming from this case. We conclude to the following:

\begin{thm}\label{807}
The minimal number of connected components of the moduli spaces of maximal polystable parabolic Higgs bundles $\mathsf{\mathcal{M}}_{par}^{max}\left( \mathrm{S}{{\mathrm{O}}_{0}}(2,n)\right)$, $\mathsf{\mathcal{M}}_{par}^{max,\alpha}\left( \mathrm{S}{{\mathrm{O}}_{0}}(2,n)\right)$ is given as follows
\begin{table}[htb]
\begin{tabular}{|c|c|}
\hline
Moduli Space $\mathcal{M}$ & $\#{{\pi }_{0}}\left( \mathsf{\mathcal{M}}\right)$ \\
    \hline\hline
    $\mathsf{\mathcal{M}}_{par}^{max}\left( \mathrm{S}{{\mathrm{O}}_{0}}(2,3)\right)$   &  \small{${{2}^{s}}\left( {{2}^{2g+s-1}}-1 \right)+2^s(4g-3+2s)$}  \\
    \hline
    $\mathsf{\mathcal{M}}_{par}^{max,\alpha}\left( \mathrm{S}{{\mathrm{O}}_{0}}(2,3)\right)$ & \small{$2^{2g+s-1}+(4g-3+2s)$}  \\
    \hline
    $\mathsf{\mathcal{M}}_{par}^{max}\left( \mathrm{S}{{\mathrm{O}}_{0}}(2,n)\right)$, $n \geq 4$   &  \small{${{2}^{2g+2s-1}}$}  \\
    \hline
    $\mathsf{\mathcal{M}}_{par}^{max,\alpha}\left( \mathrm{S}{{\mathrm{O}}_{0}}(2,n)\right)$, $n \geq 4$ & \small{$2^{2g+s-1}$}  \\
    \hline
  \end{tabular}
\end{table}
\end{thm}

\subsection{$\bf{E}_{7}^{-25}$}
The general Milnor-Wood type inequality established in \cite{BiGaRu} provides a description of maximal $G$-Higgs bundles also in the exceptional tube case when $G=E_{7}^{-25}$. In this case, a maximal compact subgroup is $H=E_{6}^{-78}{{\times }_{{{\mathbb{Z}}_{3}}}}\mathrm{U}\left( 1 \right)$ and $\mathrm{rk}\left( {G}/{H}\; \right)=3$. For the Toledo invariant $\tau =\tau \left( E \right)$ as defined in \cite{BiGaRi} in this full generality in the parabolic case, the Milnor-Wood inequality is given by
	\[\left| \tau  \right|\le 3\left( g-1+\frac{s}{2} \right).\]
In the maximal case for $\tau $, following the non-parabolic treatment of \cite{BiGaRu}, we get that a maximal parabolic $E_{7}^{-25}$-Higgs bundle corresponds to an ${H}'={{F}_{4}}\times {{\mathbb{Z}}_{2}}$-holomorphic $V$-bundle $\tilde{W}$. The group ${{{H}'}^{\mathbb{C}}}$ is not connected and the short exact sequence
	\[1\to {{{H}'}_{0}}^{\mathbb{C}}\to {{{H}'}^{\mathbb{C}}}\to {{\pi }_{0}}\left( {{{{H}'}}^{\mathbb{C}}} \right)\cong {{\mathbb{Z}}_{2}}\to 1\]
provides the homomorphism in the induced long exact sequence in $V$-cohomology
	\[H_{V}^{1}\left( M,{{{{H}'}}^{\mathbb{C}}} \right)\to H_{V}^{1}\left( M,{{\mathbb{Z}}_{2}} \right)\cong {{\left( {{\mathbb{Z}}_{2}} \right)}^{2g+s-1}}.\]
The associated invariants in $H_{V}^{1}\left( M,{{\mathbb{Z}}_{2}} \right)$ provide the following:

\begin{thm}\label{808}
The minimal number of connected components of the moduli space $\mathsf{\mathcal{M}}_{par}^{max }\left( E_{7}^{-25}\right)$ is $2^{2g+s-1}$.
\end{thm}

\begin{rem}
Theorem \ref{808} gives the least number of components for the corresponding moduli space for the group $E_{7}^{-25}$ . We leave open the possibility of having more topological invariants coming from the second cohomology group $H^2_V(M,\mathbb{Z}_2)$.
\end{rem}

Summarizing the results of the theorems in the previous sections on the minimum number of connected components of the moduli spaces of polystable maximal parabolic $G$-Higgs bundles for the classical Hermitian symmetric Lie groups $G$, one has Tables 1, 2 and 3 included at the end of the main body of this article. 

\section{Two Special Cases}
In this section, we discuss how one can obtain the classical component counts in \cite{Stru}, \cite{BrGPGHermitian},\cite{GaGMsymplectic}, \cite{Gothen} and \cite{Hit92} as special cases of the Theorems in \S 7.4 and \S 8.

\subsection{Punctured Riemann Surface}
In \cite{Stru}, T. Strubel defined Fenchel-Nielsen coordinates on the moduli space of maximal representations of the fundamental group of a topological surface ${{\Sigma }_{g,m}}$ of genus $g$ and $m\ge 1$ boundary components into $\mathrm{Sp}\left( 2n,\mathbb{R} \right)$. Using these coordinates and counting parameters for gluing pairs of pants to obtain a surface with $m$-boundary components, he showed that the moduli space ${{\mathsf{\mathcal{R}}}^{\max }}\left( {{\Sigma }_{g,m}},\mathrm{Sp}\left( 2n,\mathbb{R} \right) \right)$ has exactly ${{2}^{2g+m-1}}$ connected components for every $n\ge 1$. Note that for such representations there is no assumption on the monodromy around the boundary components.

From our point of view, let $X$ be a compact Riemann surface of genus $g$ and $\left\{ {{p}_{1}},\ldots ,{{p}_{s}} \right\}$ a collection of $s$-many distinct points on $X$. We may use the method from \S 7 to compute the number of topological invariants, however, in this case, we do not have to construct the $V$-manifold:

Let $M=U_1=X \backslash \{p_1,...,p_s\}$ be a punctured Riemann surface without any action on the $\Gamma$-equivariant bundle, in other words, without a construction of a $V$-bundle. The calculations from \S 7.3 now adapt to give the following
\begin{align*}
\mathrm{rk}(H^0(M))=1,\\
\mathrm{rk}(H^1(M))=2g+s-1,\\
\mathrm{rk}(H^2(M))=0.\\
\end{align*}
Moreover, since $H^2(M)$ is trivial, the number of topological invariants of maximal parabolic $\mathrm{Sp}(4,\mathbb{R})$-Higgs bundles over the punctured Riemann surface $M$ is only determined by the first cohomology group $H^1(M)$, thus is exactly $2^{2g+s-1}$. This is Case (1) as we discussed in \S 7.4. We further notice that when the first cohomology $u \in H^1(M)$ is trivial, one can decompose $W=L \bigoplus L^{\vee}$ as the direct sum of line bundles and this line bundle $L$ is over a punctured Riemann surface, which is an affine space, and there is only one line bundle over an affine space, the trivial one. As a result, there are no extra topological invariants coming from Cases 2(a) and 2(b), therefore the minimum number of connected components over $M$ is $2^{2g+s-1}$.

\begin{rem}\label{901}
The same argument provides the number $2^{2g+s-1}$ also in the cases for $\mathrm{Sp}(2,\mathbb{R})$ and $\mathrm{Sp}(2n,\mathbb{R})$ for $n \ge 3$ from Theorem \ref{710}; this gives an alternative explanation to T. Strubel's main result from \cite{Stru}.
\end{rem}

\subsection{The case when s=1}

The number of connected components of moduli of maximal $G$-Higgs bundles (non-parabolic) for the classical Hermitian symmetric spaces ${G}/{H}\;$ has been determined in \cite{BrGPGHermitian} and the references therein. For the reader's convenience the basic results from that article are included in Table 9.2.1. 

\begin{table}[htb]
\caption*{Table 9.2.1. Number of connected components of the non-parabolic moduli space $\mathsf{\mathcal{M}}^{max }\left( G \right)$.}
\begin{tabularx}{\textwidth}{XXX}
Lie group $G$ & $\#{{\pi }_{0}}\left( \mathsf{\mathcal{M}}^{max }\left( G \right) \right)$ & Teichm{\"u}ller components \\
    \hline\hline
    $\mathrm{Sp}(2,\mathbb{R})=\mathrm{SL}(2,\mathbb{R})$   &  \small{${{2}^{2g}}$}  &  \small{${{2}^{2g}}$}  \\
    $\mathrm{Sp}(4,\mathbb{R})$   & \small{$3\cdot {{2}^{2g}}+4g-4$}  & \small{${{2}^{2g}}$}  \\
    $\mathrm{Sp}(2n,\mathbb{R})$, for $n\ge 3$   & \small{$3\cdot{{2}^{2g}}$}  & \small{${{2}^{2g}}$}  \\
    $\mathrm{SU}(n,n)$   & \small{${{2}^{2g}}$}  & -  (\small{${{2}^{2g}}$ if $n=1$}) \\
    $\mathrm{S}{{\mathrm{O}}^{*}}(2n)$, for $n$: even   & \small{1} & -  \\
    $\mathrm{S}{{\mathrm{O}}_{0}}(2,3)$   & \small{$ {{2}^{2g+1}}+8g-4$}  & \small{1}  \\
    $\mathrm{S}{{\mathrm{O}}_{0}}(2,n)$, for $n\ge 4$   & \small{${{2}^{2g+1}}$}  & -  \\
    $E_{7}^{-25}$   & \small{${{2}^{2g}}$}  & -  \\
    \hline\hline
  \end{tabularx}
\end{table}

For a line $V$-bundle $\tilde{L}$, consider a local chart ${U}/{{{\mathbb{Z}}_{2}}}\;$ around a single point $p\in D$. The cohomology group for the $V$-manifold $M$ is described by
\begin{align*}
H^1_V(M, \mathbb{Z}_2) = \mathrm{Hom}(\pi_V^1(M),\mathbb{Z}_2),
\end{align*}
where
\begin{align*}
\pi_V^1(M)=\{a_1,b_1,...,a_g,b_g,\sigma \quad | \quad \sigma[a_1,b_1]..[a_g,b_g]=1,\sigma^2=1\}.
\end{align*}

For a well-defined morphism $\rho \in \mathrm{Hom}\left( \pi _{V}^{1}\left( M \right),{{\mathbb{Z}}_{2}} \right)$, the first relation for the fundamental group provides that $\det \left( \rho \left( \sigma \left[ {{a}_{1}},{{b}_{1}} \right]\ldots \left[ {{a}_{g}},{{b}_{g}} \right] \right) \right)=1$. But
$\det \left( \rho (\left[ {{a}_{i}},{{b}_{i}} \right]) \right)=1$ for $1\le i\le g$, since $\operatorname{Im}\rho $ lies in $\mathbb{Z}_2$. Thus  $\det \left( \rho \left( \sigma  \right) \right)=1$, that is, in terms of the construction included in \S 7.2 the $\mathbb{Z}_2$-isotropy around the point $p$ is \emph{trivial}. Hence, $\tilde{L}\to M$ is a holomorphic line bundle over the Riemann surface (non-parabolic); denote the latter by $L\to X$, and there are $2^{2g}$ many non-isomorphic square roots of the canonical line bundle, when $s=1$. This implies the non-parabolic component count for $G=\mathrm{Sp}(2n,\mathbb{R})$ when $n\ne 2$, $\mathrm{SU}(n,n)$, $E_{7}^{-25}$, $\mathrm{S}{{\mathrm{O}}^{*}}(2n)$ when $n$ is even, and $\mathrm{S}{{\mathrm{O}}_{0}}(2,n)$ when $n\ge 4$.

\textit{The component count for $\mathrm{Sp}(4,\mathbb{R})$.} The topological invariants that distinguish the connected components of $\mathsf{\mathcal{M}}_{par}^{\max }\left( \mathrm{Sp}(4,\mathbb{R}) \right)$ are $w_1 \in H^1_V(M,\mathbb{Z}_2)$, $w_2 \in H^2(M,\mathbb{Z}_2)$, the parabolic structure of a parabolic line bundle $L$ (equivalently a line bundle over $M$) and its parabolic degree $\mathrm{pardeg}(L)$ such that
	\[0 \le \mathrm{pardeg}(L) \le 2g-2+s\]
as we discussed in \S 7.4. Their values distinguish connected components as follows:
	\[\underbrace{{{2}^{s}}\left( {{2}^{2g+s-1}}-1 \right)}_{{w_1}\ne 0,{w_2}}+ \underbrace{2^s(2g-2+s)}_{\mathrm{pardeg}(L)=0,1,\ldots ,2g-3+s}+\underbrace{{{2}^{2g+s-1}}}_{\mathrm{pardeg}(L)=2g-2+s}\]
Now, as we have already checked, when $s=1$, $H_{V}^{1}\left( M,{{\mathbb{Z}}_{2}} \right)\simeq \mathbb{Z}_{2}^{2g}$ and this space is parameterizing non-isomorphic holomorphic line bundles $L\to M$, where $M$ is a $V$-manifold as considered in \S 7. Thus, we get accordingly the invariants:
\[\underbrace{2\left( {{2}^{2g}}-1 \right)}_{{w_1}\ne 0,{w_2}}+\underbrace{2(2g-2+1)}_{\mathrm{pardeg}(L)}+\underbrace{ {{2}^{2g}} }_{\mathrm{pardeg}(L)=2g-1}.\]

Now we consider the non-parabolic case of $K(D)$-twisted $\mathrm{Sp}(4,\mathbb{R})$-Higgs bundles, where $D$ includes a single point. The component count is given as follows \cite{BarSch2}
\begin{align*}
2(2^{2g}-1)+(2g-2+1)+2^{2g}.
\end{align*}
Compared to the parabolic case, the difference comes from the middle part $(2g-2+1)$. The reason is that we have to consider the parabolic structure of the line bundle $L$ as we discussed in \S 7.4. If we forget the parabolic structure (or consider the parabolic structure with weight $0$), we revoke the classical $K(D)$-twisted case: $3\cdot 2^{2g}+2g-3$.

\textit{The component count for $\mathrm{S}{{\mathrm{O}}_{0}}(2,3)$.} Quite similarly, the topological invariants that distinguish the connected components of $\mathsf{\mathcal{M}}_{par}^{max }\left( \mathrm{S}{{\mathrm{O}}_{0}}(2,3) \right)$ are $u,v$, the parabolic structure of a line bundle $M$ and its parabolic degree $\mathrm{pardeg}(M)$. As we discussed in \S 8.3, if $w_1=0$, there is a decomposition of the rank $3$ bundle $W=M \oplus \mathcal{O} \oplus M^{-1}$, where $M$ is a parabolic line bundle with $0 \le\mathrm{pardeg}(M) \le 4g-4+2s$. Connected components are distinguished as follows:
	\[\underbrace{{{2}^{s}}\left( {{2}^{2g+s-1}}-1 \right)}_{{w_1}\ne 0,{w_2}}+\underbrace{2^s(4g-3+2s)}_{\mathrm{pardeg}(M)}\]
When $s=1$, for the parabolic line bundle $M$, we have
\begin{align*}
\underbrace{2 \left( {{2}^{2g}}-1 \right)}_{{w_1}\ne 0,{w_2}}+\underbrace{2(4g-3+2)}_{\mathrm{pardeg}(M)},
\end{align*}
where $\mathrm{pardeg}(M)$ has $4g-3+2$ many choices and $2$ is the number of choices of the parabolic structures.

Now we consider the non-parabolic case of $K(D)$-twisted $\mathrm{S}{{\mathrm{O}}_{0}}(2,3)$-Higgs bundles, where $D$ includes a single point. The component count is similar to the $\mathrm{S}{{\mathrm{O}}_{0}}(2,3)$-case, and we have
\begin{align*}
2(2^{2g}-1)+(4g-3+2).
\end{align*}
Compared to the parabolic case under the assumption $s=1$, the difference comes from the part $(4g-3+2)$. The reason is the same as we discussed for the case of $\mathrm{Sp}(4,\mathbb{R})$. More precisely, we have to consider the parabolic structure of the line bundle $M$. If we forget the parabolic structure, we will go back to the non-parabolic $K(D)$-twisted case: $3\cdot 2^{2g}+4g-3$.

\begin{rem}\label{902}
The description of how the component count specializes to the non-parabolic case when $s=1$ for $G=\mathrm{Sp}(4,\mathbb{R})$ and $G=\mathrm{S}{{\mathrm{O}}_{0}}(2,3)$, points out an important difference between parabolic and non-parabolic bundles. As we have seen already, all degree zero line bundles on an orbifold surface can be naturally lifted to a compact Riemann surface. The extra $2^s$-many choices of the invariants of the line $V$-bundle $L$ for $G=\mathrm{Sp}(4,\mathbb{R})$ are coming from tensoring with the square roots of $\mathsf{\mathcal{O}}\left( {{p}_{i}} \right)$, where $p_{i}$ are the points in the divisor.
\end{rem}

\begin{rem}\label{903}
When $s=1$, we note that there is an element $\sigma \in \pi_V^1(M)$. By the discussion above, it seems that the calculation of connected components does not depend on the monodromy action, which means that the connected components should be the same for all $p \geq 2$ such that $\sigma^p=1$. We want to remind the reader that $p$ also corresponds to the denominator of the weight in the parabolic structure.

If we change the monodromy action with $\sigma^p=1$, we use the following local charts $U_2 = \prod_{i=1}^s D / \mathbb{Z}_p$ and $V_2= \prod_{i=1}^s D \times_{\mathbb{Z}_p}E \mathbb{Z}_p$ to construct $M_V$. Here we follow the notation from \S 7.3. We use the Leray spectral sequence to calculate the cohomology of $V_2$,
\begin{align*}
H^{*}(B \mathbb{Z}_p, H^{*}(D)) \Rightarrow H^{*}(V_2),
\end{align*}
where $B \mathbb{Z}_p$ is the classifying space of $\mathbb{Z}_{p}$. The $\mathbb{Z}$-coefficient cohomology $H^i(B\mathbb{Z}_p,\mathbb{Z})$ is well-known (see for instance \cite{Hat}):
\begin{equation*}
H^i(B\mathbb{Z}_p,\mathbb{Z})=
\begin{cases}
\mathbb{Z},\, \mathrm{ for }\, i=0,\\
\mathbb{Z} / p\mathbb{Z}, \,\mathrm{ for } \, 2 | i,\\
0,\, \mathrm{otherwise}.
\end{cases}
\end{equation*}
Since $H^{*}(V_2)$ is $\mathbb{Z}_2$-cohomology, we consider the following two cases:
\begin{enumerate}
\item[(1)] When $p$ is odd, $H^1 (V_2)= H^2(V_2)=0$. In this case, $\mathrm{rk}(H^1(M_V))=2g+s-1$ and $\mathrm{rk}(H^2(M_V))=s$, which is the same as what we calculated in \S 7.
\item[(2)] When $p$ is even, $H^1 (V_2)=0$ while $H^2(V_2) \neq 0$. Note that in this case, $\mathrm{rk}(H^1(M_V))$ and $\mathrm{rk}(H^2(M_V))$ do not coincide with our calculation in \S 7.
\end{enumerate}
In conclusion, when the monodromy group is the cyclic group $\mathbb{Z}_p$ with $p$ an odd integer, then the number of connected components coincides with the $\mathbb{Z}_2$ case. If $p$ is an even number, it does not.
\end{rem}

\newpage

\begin{table}[htb]
\caption{Minimum number of connected components of $\mathsf{\mathcal{M}}_{par}^{max }\left( G \right)$.}
\begin{tabularx}{\textwidth}{XXX}
Lie group $G$ & $\#{{\pi }_{0}}\left( \mathsf{\mathcal{M}}_{par}^{max }\left( G \right) \right)$ & Teichm{\"u}ller components \\
    \hline\hline
    $\mathrm{Sp}(2,\mathbb{R})=\mathrm{SL}(2,\mathbb{R})$   &  \small{${{2}^{2g+s-1}}$}  &  \small{${{2}^{2g+s-1}}$}  \\
    $\mathrm{Sp}(4,\mathbb{R})$   & \small{$\left( {{2}^{s}}+1 \right){{2}^{2g+s-1}}+2^s(2g-3+s)$}  & \small{${{2}^{2g+s-1}}$}  \\
    $\mathrm{Sp}(2n,\mathbb{R})$, for $n\ge 3$   & \small{$\left( {{2}^{s}}+1 \right){{2}^{2g+s-1}}$}  & \small{${{2}^{2g+s-1}}$}  \\
    $\mathrm{SU}(n,n)$   & \small{${{2}^{2g+s-1}}$}  & -  (\small{${{2}^{2g+s-1}}$ if $n=1$}) \\
    $\mathrm{S}{{\mathrm{O}}^{*}}(2n)$, for $n$: even   & \small{$2^s$} & -  \\
    $\mathrm{S}{{\mathrm{O}}_{0}}(2,3)$   & \small{${{2}^{s}}\left( {{2}^{2g+s-1}}-1 \right)+2^s(4g-3+2s)$}  & \small{1}  \\
    $\mathrm{S}{{\mathrm{O}}_{0}}(2,n)$, for $n\ge 4$   & \small{${{2}^{2g+2s-1}}$}  & -  \\
    $E_{7}^{-25}$   & \small{${{2}^{2g+s-1}}$}  & -  \\
    \hline\hline
  \end{tabularx}
\end{table}

\vspace{10mm}

\begin{table}[htb]
\caption{Minimum number of connected components of $\mathsf{\mathcal{M}}_{par}^{max,\alpha }\left( G \right)$ with even $\alpha$.}
\begin{tabularx}{\textwidth}{XXX}
Lie group $G$ & $\#{{\pi }_{0}}\left( \mathsf{\mathcal{M}}_{par}^{max,\alpha }\left( G \right) \right)$ & Teichm{\"u}ller components \\
    \hline\hline
    $\mathrm{Sp}(2,\mathbb{R})=\mathrm{SL}(2,\mathbb{R})$   &  \small{${{2}^{2g}}$}  &  \small{${{2}^{2g}}$}  \\
    $\mathrm{Sp}(4,\mathbb{R})$   & \small{${2}^{2g+s-1}+(2g-3+s)+2^{2g}$}  & \small{${{2}^{2g}}$}  \\
    $\mathrm{Sp}(2n,\mathbb{R})$, for $n\ge 3$   & \small{${2}^{2g+s-1}+2^{2g}$}  & \small{${{2}^{2g}}$}  \\
    $\mathrm{SU}(n,n)$   & \small{${{2}^{2g}}$}  & -  (\small{${{2}^{2g}}$ if $n=1$}) \\
    $\mathrm{S}{{\mathrm{O}}^{*}}(2n)$, for $n$: even   & \small{1} & -  \\
    $\mathrm{S}{{\mathrm{O}}_{0}}(2,3)$   & \small{${{2}^{2g+s-1}} +(4g-3+2s)$}  & \small{1}  \\
    $\mathrm{S}{{\mathrm{O}}_{0}}(2,n)$, for $n\ge 4$   & \small{${{2}^{2g+s-1}}$}  & -  \\
    \hline\hline
  \end{tabularx}
\end{table}

\vspace{10mm}

\begin{table}[htb]
\caption{Minimum number of connected components of $\mathsf{\mathcal{M}}_{par}^{max,\alpha }\left( G \right)$ with odd $\alpha$.}
\begin{tabularx}{\textwidth}{XXX}
Lie group $G$ & $\#{{\pi }_{0}}\left( \mathsf{\mathcal{M}}_{par}^{max,\alpha }\left( G \right) \right)$ & Teichm{\"u}ller components \\
    \hline\hline
    $\mathrm{Sp}(2,\mathbb{R})=\mathrm{SL}(2,\mathbb{R})$   &  \small{-}  &  -  \\
    $\mathrm{Sp}(4,\mathbb{R})$   & \small{${2}^{2g+s-1}+(2g-3+s)$}  & -  \\
    $\mathrm{Sp}(2n,\mathbb{R})$, for $n\ge 3$   & \small{${2}^{2g+s-1}$}  & -  \\
    $\mathrm{SU}(n,n)$   & -  & -  \\
    $\mathrm{S}{{\mathrm{O}}^{*}}(2n)$, for $n$: even   & \small{1} & -  \\
    $\mathrm{S}{{\mathrm{O}}_{0}}(2,3)$   & \small{${{2}^{2g+s-1}} +(4g-3+2s)$}  & \small{1}  \\
    $\mathrm{S}{{\mathrm{O}}_{0}}(2,n)$, for $n\ge 4$   & \small{${{2}^{2g+s-1}}$}  & -  \\
    \hline\hline
  \end{tabularx}
\end{table}

\newpage

\appendix

\section{Stability Condition for Parabolic $G$-Higgs Bundles}

In this appendix, we review the definition by O. Biquard, O. Garc\'{i}a-Prada and I. Mundet i Riera \cite{BiGaRi} for a general parabolic $G$-Higgs bundle and the definition of parabolic degree for parabolic principal bundles. We start from the basic Lie algebra background and claim that for $G=\mathrm{U}(n)$ and $G^{\mathbb{C}}=\mathrm{GL}(n,\mathbb{C})$, the definition of parabolic degree coincides with the one defined earlier in \S 2.1. Moreover, for $G=\mathrm{Sp}(2n,\mathbb{R})$, the general definitions also reduce to the ones considered in \S 5.

Let $G$ be a real reductive group, thus one has a Cartan decomposition of the Lie algebra $\mathfrak{g}=\mathfrak{h}\oplus\mathfrak{m}$, where $\mathfrak{h}$ is the Lie algebra of the maximal compact subgroup $H$ of $G$. This decomposition has the property that $[\mathfrak{h},\mathfrak{m}]\subset \mathfrak{m}$ and $[\mathfrak{m},\mathfrak{m}]\subset \mathfrak{h}$.  Now, the right action of $H$ defines the symmetric space $H\backslash G$. The stabilizer to $[1]\in H\backslash G$ is $H$, so we can identify $T_{[1]}H\backslash G$ with $\mathfrak{h}\backslash \mathfrak{g}\cong \mathfrak{m}$ and is stabilized by the adjoint action of $H$. Thus any metric on $\mathfrak{m}$ defines an $H$-invariant Riemannian metric on the symmetric space $H\backslash G$. This $H\backslash G$ is a symmetric of negative curvature, whose boundary could be defined by the geodesic rays denoted by $\partial_\infty(H\backslash G)$. This can be described in terms of a parabolic group as follows:

\begin{defn}\label{parabolicdefn}
A subgroup $P$ of $G$ is called \textit{parabolic}, if there exists $s\in\mathfrak{m}$ such that
\[P= P_{s}:=\{g\in G|d([e^{ts}ge^{-ts}],[1])\mbox{ is bounded when }t\to \infty\}.\]
and the parabolic subalgebra of $\mathfrak{g}$ is then defined as
\[\mathfrak{p}_s:=\{x\in \mathfrak{g}| Ad(e^{ts})x\mbox{ is bounded when }t\to \infty\}.\]
We call $s\in \mathfrak{m}$ to be the \textit{antidominant element} for $P$, if $P=P_s$.
\end{defn}

For a choice of $s\in\mathfrak{m}$ and $g\in G$, we may write $g=ph$ for some $p\in P_s$ and $h\in H$. We then set
\[s\cdot g:=Ad(h^{-1})(s).\]
Note that the element coming from this action still stays in $\mathfrak{m}$.

We define the \textit{geodesic ray} in $H\backslash G$ to be a morphism of the form $\gamma:[0,\infty)\to H\backslash G$ by $\gamma(t)= [ e^{t s}\cdot g]$ for some $s\in \mathfrak{h}$ and $g\in G$. Here $e^{ts}\cdot g$ is the ordinary product in $G$ and $[g]$ means the representative in $H\backslash G$.

Now we define the equivalence of two geodesics as follows:
\begin{defn}
Let $d$ be the distance function between points in $H\backslash G$. We say two geodesic rays $\gamma_1,\gamma_2$ are \textit{equivalent} if
$d(\gamma_1(t),\gamma_2(t))$ is bounded and independent of $t$.
\end{defn}
Now the \textit{boundary at infinity} of the symmetric space is defined by
\[\partial_\infty (H\backslash G)=\{\mbox{geodesic rays}\}/\sim.\]

\begin{rem}
The parabolic group $P_s$ is actually the stabilizer of the element $\gamma(t)=[e^{ts}\cdot g]$.
\end{rem}

For any $s\in \mathfrak{m}\backslash\{0\}$, the ordinary geodesic $\eta_s(t):t\mapsto [ e^{ts}\cdot 1]$ provides an element in $\partial_\infty(H\backslash G)$. The claim is that for all possible $s$, then $\eta_s$ enumerates all elements in $\partial_\infty(H\backslash G)$. Indeed, one has the following lemma:
\begin{lem}
There is an equivalence between $[ e^{ts}\cdot g]$ and $[e^{tu}]$, where $u$ is determined as follows: If $g$ has the decomposition $g=ph$ with $p\in P_s$, $h\in H$, then $u=s\cdot g:=\mathrm{Ad}_{h^{-1}}(s)$.
\end{lem}
\begin{proof}
One has $d([e^{ts}\cdot g],[e^{tu}])=d([e^{ts}ge^{-tu}],[1])=d([e^{ts}p e^{-ts} e^{ts}he^{-tu}],[1])$, thus the distance is bounded if and only if $e^{ts}he^{-tu}$ is bounded. On the other hand, we know that $e^{ts}h=he^{t \mathrm{Ad}_{h^{-1}}(s)}$ from the Baker-Hausdorff formula. It follows that in order to have boundedness, we may just let $u=\mathrm{Ad}_{h^{-1}}(s)$. This calculation also shows that the following limit exists:
\[\lim_{t\to\infty}\frac{1}{t}\log(e^{ts}\cdot g)=\mathrm{Ad}_{h^{-1}}(s).\]
Here the $\log$ is defined to be $\log(g)=v$ if $g$ has the Cartan decomposition $g= ke^{v}$ for $k\in H$ and $v\in\mathfrak{m}$.
\end{proof}

We deduce that for every element $\gamma$ in $\partial_\infty(H\backslash G)$ and for every element $x\in H\backslash G$, we can find an element $s$ in $\mathfrak{m}$ such that $\gamma(t)=[x\cdot e^{ts}]$. We call $v(x,\gamma):=s\in \mathfrak{m}\cong T_x(H\backslash G)$.

The Tits distance between $\gamma,\gamma'\in \partial_\infty(H\backslash G)$ is defined by
\[d_{Tits}(\gamma,\gamma')=\sup_{x\in H\backslash G} \textrm{Angle}(v(x,\gamma),v(x,\gamma')),\]
where the angle is in $[0,\pi]$; it measures the maximal possible angle from the same initial point to the required boundary $\gamma$ and $\gamma'$.

Let $P$ be the parabolic group $P_s$ associated to $s\in\mathfrak{m}$ and $Q=Q_\sigma$ the parabolic for $\sigma\in \mathfrak{m}$. Then we define the \emph{relative degree} as
\[\deg((P,s),(Q,\sigma)):=|s|\cdot|\sigma|\cdot\cos d_{Tits}(\eta(s),\eta(\sigma)).\]
The definition can be seen from the topology of $\partial_\infty(H\backslash G)$. We are measuring the inner product of the tangent vector associated to $\eta(s):=\gamma_s(t)=[e^{ts}]$ and $\eta_\sigma:=\gamma_\sigma(t)=[e^{t\sigma}]$ and are trying to determine the maximum possible degree. An alternative characterization is given by the function 

\[\mu_s:\mathfrak{m}\to\mathbb{R},\quad \mu_s(\sigma)=\lim_{t\to+\infty}\langle s\cdot e^{-t\sigma},\sigma\rangle.\]

\begin{lem}
For $|s|=|\sigma|=1$, we can identify $s$ and $\sigma$ with the corresponding element in $\partial_\infty(H\backslash G)$
\[\mu_s(\sigma)=\cos d_{Tits}(s,\sigma).\]
\end{lem}
This lemma follows directly from our previous calculation identifying $[e^{ts}]$ with $[x\cdot e^{tu}]$. The proof can be found in \cite{BiGaRi} Proposition B.1.

Using these definitions, we can describe several examples in further detail; the following example is worked out based on \cite{Rier1} and \cite{Rier2}.
\begin{exmp}[ $G=\mathrm{GL}(n,\mathbb{C}),H=\mathrm{U}(n)$]
In this situation, $\mathfrak{h}$ is the set of anti-Hermitian matrices and $\mathfrak{m}$ is the set of Hermitian matrices. We also claim that the map $\mathfrak{m}\backslash \{0\}\to \partial_{\infty}(H\backslash G)$ sending $s\in \mathfrak{m}$ to the geodesic $\gamma(t)=[e^{ts}]$ defines a bijection. Indeed, we can describe the structure of parabolic group explicitly. For an element $s\in\mathfrak{m}$, $s$ is a Hermitian matrix and thus it can be diagonalized. Moreover, the matrix $s$ has real eigenvalues, say $\lambda_1<\lambda_2<\ldots<\lambda_r$. We denote $V_j=\ker(\lambda_j \mathrm{I}- s)$ to be the eigenspaces associated to each eigenvalue. We let $W_k=V_1\oplus\ldots V_k$ and now the group
\[P_s=\{g\in \mathrm{GL}(n,\mathbb{C})| g(W_k)\subset W_k,\forall k\in [1,n]\}\]
is parabolic, considering $e^{ts} g e^{-ts}$. We can diagonalize $e^{s}$ and the eigenspaces are exactly the $V_j$'s and $e^{ts}$ acts on $V_j$ by multiplying by $e^{\lambda_j t}$. It follows that in order to have a bounded $e^{ts} g e^{-ts} v$ for $v\in V_j$, then it must be $g(v)\subset W_j$.

Now the relative degree can also be calculated by the choice of two antidominant elements. In fact, for two choices of $s,\sigma \in\mathfrak{m}$, we just need to calculate $\mu_s(\sigma)$.

For any $g\in \mathrm{GL}(n,\mathbb{C})$, we first compute $s\cdot g$. Recall that $s\cdot g=\mathrm{Ad}(h^{-1})(s)$ for $g=ph$, where $p\in P_s$ and $h\in H$. Since $s$ can be diagonalized, $s(v)=\lambda_j v$ for $v\in V_j$. We also see that $g^{-1}(W_k)=h^{-1}(W_k)$ by preservation of flags. Then the claim is that \[s\cdot g=\sum_{j=1}^k \lambda_j \pi_{g^{-1}(W_j)\cap g^{-1}(W_{j-1})^\perp}=\sum_{j=1}^k (\lambda_j -\lambda_{j+1})\pi_{g^{-1}W_j},\]
where $\pi_{W_j}$ is the projection of vector onto $W_j$.

Then we can use this claim to compute the relative degree, knowing that for two $s,\sigma$ there are $\lambda_1,\ldots,\lambda_k$ for $s$ and $\mu_1,\ldots,\mu_l$ for $\sigma$, as well as eigenspaces $V_1,\ldots,V_k $ for $s$ and $A_1,\ldots,A_l$ for $\sigma$. Also, we let $W_j=V_1\oplus\ldots\oplus V_j$ and $B_j=A_1\oplus\ldots\oplus A_j$ and then a calculation shows
\[\deg((P_s,s),(P_\sigma,\sigma))=\sum_{i=1}^k \sum_{j=1}^l(\lambda_i-\lambda_{i+1})(\mu_j-\mu_{j+1})\dim(W_i\cap B_j).\]
Here we assumed that $\lambda_{k+1}=0$.
\end{exmp}
We can now define the stability condition for parabolic principal bundles from the previous definition for the degree.

In \cite{BiGaRi} the authors introduce parabolic $G$-Higgs bundles over a punctured Riemann surface $X$ for a non-compact real reductive Lie group $G$ and establish a Hitchin-Kobayashi type correspondence for such pairs. This definition involves a choice for each puncture of an element in the Weyl alcove $\mathcal{A}$ of a maximal compact subgroup $H\subset G$. 

Let $\left( X,D \right)$ be as earlier, a pair of a compact connected Riemann surface $X$ and $D=\left\{ {{x}_{1}},\ldots ,{{x}_{s}} \right\}$ a divisor of $s$-many distinct points on $X$. Let also ${{H}^{\mathbb{C}}}$ be a reductive, complex Lie group. Fix a maximal compact subgroup $H\subset {{H}^{\mathbb{C}}}$, and a maximal torus $T\subset H$ with Lie algebra $\mathfrak{t}$. For an $H^{\mathbb{C}}$-principal holomorphic bundle $E$ over $X$, and for any set $W$ on which $H^{\mathbb{C}}$ acts on the left, we denote by $E(W)$ the twisted product $E\times_{H^{\mathbb{C}}}W$, that is, the product quotient out the equivalence relation $E\times W/\sim$ with the latter defined by 
$(e,hw)\sim (eh,w)$. If $W$ is the representation of $H^{\mathbb{C}}$, we can always associate a vector bundle $E(W)$.

The following argument for the definition of a stability condition is analogous to the non-parabolic case (see \cite{GaGoRi}). For a representation into a Hermitian vector space $\rho:H\to U(B)$, let us denote the holomorphic extension still  by $\rho:H^\mathbb{C}\to \mathrm{GL}(B)$. Given a parabolic subgroup $P_s$ we can define the subspace 
\begin{align}B^-_s=\{v\in B| \rho(e^{t s})v\mbox{ is bounded when }t\to\infty\}\label{bs-}\end{align}
and the invariant subspace under the Levi subalgebra
\begin{align}B^0_s=\{v\in B| \rho(e^{t s})v=v\mbox{ for any t}.\}\label{bs0}\end{align}
We can also define the difference between these two, namely the bounded but not fixed part
\begin{align}B^{<0}_s=\{v\in B^-_s |v\not \in B^0_s\} \label{bsls}\end{align}
\begin{rem}
There is a one-to-one correspondence between an antidominant element $s$ and an \textit{antidominant character} $\chi$ (see \cite{Rier2}). In the sequel, we shall be making use of the subspaces $B_\chi^-$ and $B_\chi^0$ in accordance to \cite{BiGaRi}.
\end{rem}

A \textit{holomorphic reduction} of structure group to a parabolic subgroup $P$ is a section of the principal bundle $E(H^\mathbb{C}/P)$. Since canonically $E(H^\mathbb{C}/P)\cong E/P$, a holomorphic reduction $\sigma:X\to E/P$ together with the quotient map $E\to E/P$ gives a $P$-principle bundle $E_\sigma:=\sigma^*(E)$. This is exactly where the term ``reduction" is coming from. It then follows immediately that $E(B)\cong E_\sigma \times_P B$ and we have the vector bundle $$E_{\sigma,\chi}^-(B):=E_\sigma \times_{P_\chi} B_\chi^-,$$ where $P_\chi$ is the parabolic subgroup associated to the antidominant character $\chi$. 

Following Lemma 2.12 from \cite{GaGoRi} we imply the definition for the degree of a principal bundle associated to a holomorphic reduction $\sigma$ and an antidominant character $\chi$. 

\begin{defn}\label{a09}
For a holomorphic reduction $\sigma\in \Gamma(E(H^\mathbb{C}/P))$ for a parabolic subgroup $P$ and an antidominant character $\chi$ for $P$, we can associate the antidominant element $s_\chi$. Moreover, for a representation $\rho_W:H\to U(W)$ for some Hermitian vector space $W$, then $\rho_W(s_\chi)$ diagonalizes with real eigenvalues $\lambda_1<\lambda_2<\ldots<\lambda_r$. Let $V_j = \ker (\lambda_j \mathrm{Id}_W-\rho_W(s_\chi))$ and $W_j=V_1\oplus\ldots \oplus V_j$, for $j=1, \ldots , r$ and $W=W_r$. Suppose that the associated vector bundles $\mathcal{W}_j=E(W_j)$ and $\mathcal{W}=E(W)$ are all holomorphic, then the degree of $E$ is defined as
\[\deg E(\sigma,\chi)=\lambda_k\deg \mathcal{W}+\sum_{i=1}^{r-1}(\lambda_i-\lambda_{i+1})\deg \mathcal{W}_j.\]
\end{defn}

We now include into our study the parabolic structure at each point. We fix an alcove $\mathsf{\mathcal{A}}\subset \mathfrak{t}$ of $H$ containing $0\in \mathfrak{t}$ and for ${{\alpha }_{i}}\in \sqrt{-1}\mathsf{\bar{\mathcal{A}}}$ where $\bar{\mathcal{A}}$ is the closure of $\mathcal{A}$, and we let ${{P}_{{{\alpha }_{i}}}}\subset {{H}^{\mathbb{C}}}$ be the parabolic subgroup defined by the ${{\alpha }_{i}}$ coming from definition \ref{parabolicdefn}. 

\begin{rem}\label{Hecke}
In fact, we can choose any element in $\mathfrak{t}$ and define the corresponding parabolic structure. But according to Lemma 3.3 in \cite{BiGaRi} one can always choose a shift in the cocharacter lattice $\Lambda_{cochar}$ and get a 1-1 correspondence between the local holomorphic sections. It is therefore more convenient to restrict to the alcoves, instead of the whole maximal toric algebra.
\end{rem}

\begin{defn}\label{208}
A \emph{parabolic structure} of weight ${{\alpha }_{i}}$ on $E$ over a point ${{x}_{i}}$ is defined as the choice of a subgroup ${{Q}_{i}}\subset E{{\left( {{H}^{\mathbb{C}}} \right)}_{{{x}_{i}}}}$ with an antidominant character $\alpha_i$ for $Q_i$.
\end{defn}

\begin{rem}
In order to be compatible with Definition \ref{201}, a small modification for the choice of the $\alpha_i^j$ is necessary here. Indeed, we need a decreasing sequence of weights 
\[\alpha_i^{r}>\alpha_i^{r-1}>\ldots>\alpha_i^{1},\]
so that the increasing filtration is then 
\[E_r\subset E_{r-1}\subset\ldots E_1=E_{x_i}.\]
If this is written as a decreasing filtration, it will recover the original filtration at the point. This now provides
\begin{align*}\deg((P_s,s),(Q_i,\alpha_i))&=\sum_{i=1}^s\sum_{j=1}^r(\lambda_i-\lambda_{i+1})(\alpha_i^{r+1-j}-\alpha_i^{r-j})\dim(W_i\cap E_{r+1-j})
\\ &=\sum_{i=1}^s\sum_{j=1}^r(\lambda_i-\lambda_{i+1})(\alpha_i^{j}-\alpha_i^{j-1})\dim(W_i\cap E_{j}).\end{align*}
Note that we have assumed here that $\alpha_i^0=0$.
\end{rem}

The definition for the parabolic degree of a parabolic principal bundle now combines the last two definitions:
\begin{defn}
Consider a parabolic principal $H^\mathbb{C}$-bundle $E$ together with a holomorphic reduction $\sigma: X\to E/P_\chi$ and an antidominant character $\chi$. We define the \textit{parabolic degree} of $E$ with respect to $\sigma$ and $\chi$ to be the real number
\begin{align}
    \pardeg E(\sigma,\chi):=\deg E(\sigma,\chi)+\sum_{i=1}^s \deg((E_\sigma,\chi),(Q_i,\alpha_i)). 
\end{align}
\end{defn}
 \begin{rem} In \cite{BiGaRi} the authors take the minus sign in the term including the contribution from the weights in the expression of the parabolic degree. We take a plus sign here instead, coming from the increasing sequence of weights that we considered above; this shall ensure that when examining the case when $G=\mathrm{GL}(n,\mathbb{C})$ later on, the definition for the parabolic degree will coincide with Definition \ref{203}.
 \end{rem}
 
For a real reductive Lie group $G$ with a maximal compact subgroup $H$, let $\mathfrak{g}=\mathfrak{h}\oplus \mathfrak{m}$ be the Cartan decomposition of the Lie algebra into its $\pm 1$-eigenspaces, where $\mathfrak{h}=\mathrm{Lie}\left( H \right)$ and let $E\left( {{\mathfrak{m}}^{\mathbb{C}}} \right)$ be the bundle associated to $E$ via the isotropy representation. Choose a trivialization $e\in E$ near the point ${{x}_{i}}$, such that near ${{x}_{i}}$ the parabolic weight lies in ${{\alpha }_{i}}\in \sqrt{-1}\mathsf{\bar{\mathcal{A}}}$. In the trivialization $e$, we can decompose the bundle $E\left( {{\mathfrak{m}}^{\mathbb{C}}} \right)$ under the eigenvalues of $\mathrm{ad}\left( {{\alpha }_{i}} \right)$ acting on ${{\mathfrak{m}}^{\mathbb{C}}}$ as
	\[E\left( {{\mathfrak{m}}^{\mathbb{C}}} \right)=\underset{\mu }{\mathop{\oplus }}\,\mathfrak{m}_{\mu }^{\mathbb{C}}.\]

In particular, take ${{\alpha }_{i}}\in \sqrt{-1}{{\mathsf{{\mathcal{A}}'}}_{\mathfrak{g}}}$, where ${{\mathsf{{\mathcal{A}}'}}_{\mathfrak{g}}}$ is the space of $\alpha \in \mathsf{\bar{\mathcal{A}}}$ such that the eigenvalues of $\mathrm{ad}(\alpha) $ have modulus smaller than 1 on the entire $\mathfrak{g}$, and consider for $\alpha \in \sqrt{-1}\mathfrak{h}$ the subspaces of ${{\mathfrak{m}}^{\mathbb{C}}}$ similar to definition of $B_\alpha^-$ and $B_\alpha^0$. We can define $\mathfrak{m}_i:=(\mathfrak{m}^{\mathbb{C}})_{\alpha_i}^-$ as well as $\mathfrak{m}_i^0=(\mathfrak{m}^{\mathbb{C}})_{\alpha_i}^0$ according to equation \ref{bs-} and \ref{bs0}.
The decomposition implies that $\mathfrak{m}_i=\mathfrak{m}_i^0\oplus \mathfrak{n}_i$ for some $\mathfrak{n}_i$.
 We define the sheaf $PE\left( {{\mathfrak{m}}^{\mathbb{C}}} \right)$ of \emph{parabolic sections of $E\left( {{\mathfrak{m}}^{\mathbb{C}}} \right)$} as the sheaf of local holomorphic sections $\psi $ of $E\left( {{\mathfrak{m}}^{\mathbb{C}}} \right)$ such that $\psi \left( {{x}_{i}} \right)\in {{\mathfrak{m}}_{i}}$. Similarly, the sheaf $NE\left( {{\mathfrak{m}}^{\mathbb{C}}} \right)$ of \emph{strongly parabolic sections of $E\left( {{\mathfrak{m}}^{\mathbb{C}}} \right)$} is defined as the sheaf of local holomorphic sections $\psi $ of $E\left( {{\mathfrak{m}}^{\mathbb{C}}} \right)$ such that $\psi \left( {{x}_{i}} \right)\in {{\mathfrak{n}}_{i}}$. The following short exact sequences of sheaves are then realized

\begin{equation}
0\to PE\left( {{\mathfrak{m}}^{\mathbb{C}}} \right)\to E\left( {{\mathfrak{m}}^{\mathbb{C}}} \right)\to \underset{i}{\mathop{\bigoplus }}\,E{{\left( {{\mathfrak{m}}^{\mathbb{C}}} \right)}_{{{x}_{i}}}}/{{\mathfrak{m}}_{i}}\to 0,
\label{PEses}
\end{equation}
\begin{equation}
0\to NE\left( {{\mathfrak{m}}^{\mathbb{C}}} \right)\to E\left( {{\mathfrak{m}}^{\mathbb{C}}} \right)\to \underset{i}{\mathop{\bigoplus }}\,E{{\left( {{\mathfrak{m}}^{\mathbb{C}}} \right)}_{{{x}_{i}}}}/{{\mathfrak{n}}_{i}}\to 0.
\label{NEses}
\end{equation}

\begin{defn}
Let $\left( X,D \right)$ be a pair of a compact connected Riemann surface $X$ and $D=\left\{ {{x}_{1}},\ldots ,{{x}_{s}} \right\}$ a divisor of $s$-many distinct points on $X$.
A \textit{parabolic G-Higgs bundle} over $(X, D)$ is a pair $(E,\Phi)$ such that $E$ is an $H^\mathbb{C}$-principal bundle over $X$ and $\Phi$ is a holomorphic section of $PE(\mathfrak{m}^\mathbb{C})\otimes K(D)$. A \textit{strongly parabolic G-Higgs bundle} $(E,\Phi)$ over $(X, D)$ involves $\Phi\in \Gamma(NE(\mathfrak{m}^\mathbb{C})\otimes K(D))$.
\end{defn}

We have the following three notions of stability. Note that the parameter $\alpha$ in the following definition should not be confused with the parabolic structure $\alpha$ in previous settings (Definition \ref{201}).
\begin{defn}[Stability conditions]Let $(E,\Phi)$ be a parabolic $G$-Higgs bundle, and let $\alpha\in i\mathfrak{h}\cap\mathfrak{z}$, where $\mathfrak{z}$ is the center of $\mathfrak{h}^\mathbb{C}$. Also, let $\langle\cdot,\cdot\rangle$ be a nondegenerate bilinear pairing on $\mathfrak{h}^{\mathbb{C}}$.

\begin{enumerate}
    \item We say the parabolic $G$-Higgs bundle $(E,\Phi)$ is \textbf{$\alpha$-semistable}, if for any subgroup $P\subset H^\mathbb{C}$, any anti-dominant character $\chi$ of $P$ and any holomorphic reduction $\sigma\in \Gamma(H^\mathbb{C}/P)$ such that $\Phi\in H^0(E(\mathfrak{m}^{\mathbb{C}})_{\sigma,\chi}^-\otimes K(D))$, one has
    \[\pardeg E(\sigma,\chi)-\langle \alpha,\chi\rangle\ge 0.\]
    \item We say the parabolic $G$-Higgs bundle is \textbf{$\alpha$-stable}, if the above inequality is strict.
    \item We say the parabolic $G$-Higgs bundle is \textbf{$\alpha$-polystable}, if it is semistable and if for $\sigma,\chi$ as earlier such that $\Phi\in H^0(E(\mathfrak{m}^{\mathbb{C}})_{\sigma,\chi}^-\otimes K(D))$, it is
    \[\pardeg E(\sigma,\chi)-\langle \alpha,\chi\rangle=0,\]
    then there is a holomorphic reduction $\sigma_{L_s}\in \Gamma(E_\sigma(P/L_s))$ where $L_s$ is the Levi subgroup such that $\Phi\in H^0(E(\mathfrak{m}^\mathbb{C})_{\sigma_{L_s},\chi}^0\otimes K(D))\subset H^0(E(\mathfrak{m}^\mathbb{C})_{\sigma,\chi}^-\otimes K(D))$.
    \item If the parameter $\alpha=0$ above, we just call the corresponding Higgs bundle \textbf{semistable}, \textbf{stable} and \textbf{polystable} respectively.
    \end{enumerate}
\end{defn}
This abstract definition can be unraveled considering particular examples, and indeed it turns out that it coincides with Definition \ref{203} in the case $G=\mathrm{GL}(n,\mathbb{C})$. We describe this in detail in the following:

\begin{exmp}[ $G=\mathrm{GL}(n,\mathbb{C}),H=\mathrm{U}(n)$]\label{Gln}
 For a vector space $\mathbb{C}^n$ and a representation $H^{\mathbb{C}}$ acting on $V$, there is an identification of a parabolic $G=\mathrm{GL}(n,\mathbb{C})$-bundle with a parabolic vector bundle $W$. Thus $E(\mathfrak{m}^{\mathbb{C}})$ is in fact $\mathrm{End}(W)$ and $\Phi\in H^0(\mathrm{End}(W)\otimes K(D))$. Moreover, one has to check that $\mathrm{Res}_{x_i}\Phi \in \mathfrak{m}_i:=\mathfrak{m}_{\alpha}^-$.
 
 We already know that the parabolic structure $(Q_i,\alpha_i)$ involves a choice of real eigenvalues from $0$ to $1$ from Example \ref{Gln} and Remark \ref{Hecke}. Thus this corresponds to a choice of real numbers $-\frac{1}{2}\le \alpha_i^1\le\alpha_i^2\le\ldots\le \alpha_i^n\le \frac{1}{2}$ and thus the subspaces $A_j=\ker(\alpha_i^j \mathrm{Id}-\alpha_i)$ and also $B_{j}=A_n\oplus A_{n-1} \oplus\ldots\oplus A_j$. Moreover, 
 \[\Gamma(PE(\mathfrak{m}^{\mathbb{C}})\otimes K(D))=\{\Phi\in \mathrm{Hom}(E,E\otimes K(D))|\Phi(B_j)\subset B_j\otimes K(D)\},\]
 which corresponds to the classical definition of a parabolic Higgs bundle.
 
We next check the $\alpha$-stability condition. For the group $G=\mathrm{GL}(n,\mathbb{C})$, its center is the set of diagonal matrices. Therefore, a choice of $\alpha$ in fact involves a choice of a real number. Indeed, since $\chi$ can be diagonalized with real eigenvalues $\lambda_1<\ldots<\lambda_r$ with eigenspaces of dimension $\dim V_1,\ldots \dim V_r$, then
$\langle \alpha,\chi\rangle=\alpha \sum_{i=1}^r \lambda_i \dim V_i$.
 
For every classical group $G$ and any $G$-principal bundle $P$, one can associate a vector bundle $W$ by considering the fundamental representation $\mathcal{W}=P\times_{G}W$ where $W$ is the fundamental representation of $G$ and the Higgs field will be a section of the vector bundle $P\times_{Ad}\mathfrak{m}^{\mathbb{C}}$. If $G=\mathrm{GL}(n,\mathbb{C})$, the associated vector bundle will be the vector bundle $\mathcal{W}$ and $\Phi$ will be section of $\mathrm{End}(\mathcal{W})\otimes K(D)$, thus we have the following claim:
 \begin{lem}
 With the same notation from Definition \ref{a09}, a parabolic $GL(n,\mathbb{C})$-Higgs bundle $(E,\Phi)$ is $\mu$-semistable/stable/polystable if and only if the associated parabolic Higgs bundle $(\mathcal{W},\Phi_{\mathcal{W}})$ is semistable/stable/polystable. Here $\mu:=\mu(\mathcal{W})=\frac{\pardeg \mathcal{W}}{n}$.
 \end{lem}
 \begin{proof}
 ($\Leftarrow$) Let $P\subset H^\mathbb{C}$ be a parabolic subgroup, $\chi$ an  antidominant character of $P$ and  $\sigma\in \Gamma(H^\mathbb{C}/P)$ a holomorphic reduction such that $\Phi\in H^0(E(\mathfrak{m}^{\mathbb{C}})_{\sigma,\chi}^-\otimes K(D))$.
 
 We may consider the case where there is only one point in the divisor $D$ and then generalize the argument. Consider the antidominant character $\chi$ of $P$ and the representation $\rho_W:H\to U(W)$ with the real eigenvalues $\lambda_1<\ldots<\lambda_r$. Then by Definition \ref{a09}, if we set $V_i=\ker(\lambda_j\mathrm{Id}_W-\rho_W(s_\chi))$, $W_i=V_1\oplus\ldots\oplus V_r$ and $W=W_r$, one can calculate the parabolic degree as follows
 \begin{align*}
     & \pardeg E(\sigma,\chi)-\langle\alpha,\chi\rangle=
     \\& =\sum_{k=1}^{r}(\lambda_k-\lambda_{k+1})(\deg \mathcal{W}_k-\mu\mathrm{rk}(\mathcal{W}_k)) + \sum_{k,j}(\lambda_k-\lambda_{k+1}) (\alpha_i^j-\alpha_i^{j-1})\dim(W_k\cap B_{j})
     \\&=\sum_{k=1}^{r}(\lambda_k-\lambda_{k+1}) \left(\deg \mathcal{W}_k-\mu\mathrm{rk}(\mathcal{W}_k)+\sum_{j=1}^{n}(\alpha_i^{j}-\alpha_i^{j-1})\dim(W_k\cap B_{j})\right)
     \\&=\lambda_r (\pardeg \mathcal{W}-\mu\mathrm{rk}(\mathcal{W}))+\sum_{k=1}^{r-1} (\lambda_k-\lambda_{k+1}) (\pardeg \mathcal{W}_k-\mu\mathrm{rk}(\mathcal{W}_k)) \ge 0.
 \end{align*}
 The last inequality follows since the Higgs bundle preserves the flag, $\Phi(\mathcal{W}_j)\subset \Phi(\mathcal{W}_j)\otimes K(D)$, thus the slope inequality for stable/semistable parabolic Higgs bundle yields $\frac{\pardeg \mathcal{W}_k}{\mathrm{rank}(\mathcal{W}_k)}\le \mu(\mathcal{W})= \frac{\pardeg \mathcal{W}}{n}$, so $\pardeg \mathcal{W}_k\le \mu \mathrm{rk}(\mathcal{W}_k)$ if we rearrange the terms in the equation. Moreover, $\lambda_k-\lambda_{k+1}$ is less than $0$ by definition. 

Note that $\Phi\in H^0(E(\mathfrak{m}^\mathbb{C})^-_{\sigma,\chi}\otimes K(D))$ actually says that $\Phi(\mathcal{W}_k)\subset \mathcal{W}_k\otimes K(D)$. Together with the assumption about parabolic structure from $\Phi(B_j)\subset B_j\otimes K(D)$ we derive a strong restriction on the possible $W$'s that are allowed.
 
Now we have to show that the polystability conditions are corresponding to each other. But this just comes from the fact that our Higgs bundle $\mathcal{W}=V_1\oplus V_2$ and $\Phi\in H^0(\mathrm{Hom}(V_1,V_1\otimes K(D))\oplus\mathrm{Hom}(V_2,V_2\otimes K(D)) )$, thus for any parabolic subgroup that preserves the flag of $\mathcal{W}$ we can reduce it to the subgroup $\mathrm{GL}(\dim V_1,\mathbb{C})\times\mathrm{GL}(\dim V_2,\mathbb{C}) $, the Levi subgroup of parabolic reduction will be straightforward as well.
 
($\Rightarrow$) Let $(E(W),\Phi_W)$ be a $\mu$-semistable parabolic $G$-Higgs bundle. For any subbundle $V$ of $W$, one can associate the filtration $0\subset V\subset W$ and the associated parabolic subgroup that preserves this flag. For any antidominant character $\chi$ of the associated parabolic and any holomorphic reduction $\sigma$ one has that 
 \begin{align}
     \pardeg E(\sigma,\chi)-\langle\mu,\chi\rangle=(\lambda_1-\lambda_2)(\pardeg V-\mu(\mathcal{W})\mathrm{rk}(V))\ge 0,
 \end{align}
 which provides that $\frac{\pardeg V}{\mathrm{rk}(V)}\le \mu(\mathcal{W})$. The polystability condition is checked similarly.
 \end{proof}
\end{exmp}

We now come to our key example $G=\mathrm{Sp}(2n,\mathbb{R})$, for which we shall examine in detail what it means to be semistable, stable or polystable. The treatment is similar to Section 4.3 of \cite{GaGoRi}.
\begin{exmp}[$G=\mathrm{Sp}(2n,\mathbb{R}),H=U(n)$]
For the group $G = \mathrm{Sp}(2n,\mathbb{R})$, it is $H^\mathbb{C}=\mathrm{GL}(n,\mathbb{C})$, and for the fundamental representation $\mathbb{V}$ of $\mathrm{GL}(n,\mathbb{C})$ one has the isotopy representation
\[\mathfrak{m}^\mathbb{C}=\mathrm{Sym}^2\mathbb{V}\oplus \mathrm{Sym}^2\mathbb{V}^*.\]
Following Section 4.3 in \cite{GaGoRi} and replacing the degree by the parabolic degree, one has the following characterization of semistability, stability and polystability:

The parabolic $\mathrm{Sp}(2n,\mathbb{R})$-Higgs bundle consists of a parabolic vector bundle $V=E(\mathbb{V})$ and a section
\[\Phi=(\beta,\gamma)\in H^0(K(D)\otimes \mathrm{Sym}^2V\oplus K(D)\otimes \mathrm{Sym}^2V^*).\]
 
The stability condition turns out to be similar while we have a different Higgs field. For the parabolic vector bundle $V$ we define the \textit{filtration of $V$ of length $k-1$} to be a strictly increasing filtration of holomorphic subbundles
\[\mathcal{V}=(0=V_0\subsetneq V_1\subsetneq V_2\subsetneq\ldots \subsetneq V_{k-1}\subsetneq V_k=V).\]
Let $\lambda=(\lambda_1<\lambda_2<\ldots<\lambda_k)$ be a strictly increasing sequence of $k$-real numbers, and we define
\[\mathcal{N}_\beta(\mathcal{V},\lambda)=\sum_{\substack{\lambda_i+\lambda_j\le 0 \\ 1\le i,j\le k}}K(D)\otimes (V_i\otimes_S V_j)\mbox{ and }\mathcal{N}_\gamma(\mathcal{V},\lambda)=\sum_{\substack{\lambda_i+\lambda_j\ge 0 \\ 1\le i,j\le k}}K(D)\otimes (V_{i-1}^\perp\otimes_S V_{j-1}^\perp).\]
Here, for $W_1,W_2$ subbundles of $W$, then $W_1\otimes_S W_2$ is the image of the map $W_1\otimes W_2\subset W\otimes W\to \mathrm{Sym}^2 W$. Also, $W_1^\perp$ is the kernel of restriction map $W^*\to W_1^*$. We also set
\[\mathcal{N}(\mathcal{V},\lambda):=\mathcal{N}_\beta(\mathcal{V},\lambda)\oplus \mathcal{N}_\gamma(\mathcal{V},\lambda)\]
and define
\[pd(\mathcal{V},\lambda,\alpha):=\sum_{j=1}^k(\lambda_j-\lambda_{j+1})(\pardeg V_j-\alpha \mathrm{rk}(V_j)).\]

Now we can define the subspaces associated to the parabolic structure and determine $PE(\mathfrak{m}^\mathbb{C})$. For any parabolic structure at $x_i$ we have a filtration $0=U_0\subsetneq U_1 \subsetneq\ldots \subsetneq U_{r-1}\subsetneq U_r=V_{x_i}$ with associated weights $\alpha_1<\alpha_2<\ldots<\alpha_r$. Define
\[U_\beta(V_{x_i},\alpha):=\sum_{\substack{\alpha_i+\alpha_j\le 0 \\ 1\le i,j\le k}}K(D)\otimes (U_i\otimes_S U_j)\mbox{ and }U_\gamma(V_{x_i},\alpha)=\sum_{\substack{\alpha_i+\alpha_j\ge 0 \\ 1\le i,j\le k}}K(D)\otimes (U_{i-1}^\perp\otimes_S U_{j-1}^\perp),\]
and analogously $U(V_{x_i},\alpha)=U_\beta(V_{x_i},\alpha)\oplus U_\gamma(V_{x_i},\alpha)$. Then the condition $\Phi\in H^0(PE(\mathfrak{m}^\mathbb{C})\otimes K(D))$ is nothing but $\Phi_{x_i}\in U(V_{x_i},\alpha)$.

Thus we have the following characterization of the stability condition by the same argument in \cite{GaGoRi}:
\begin{lem} A parabolic $\mathrm{Sp}(2n,\mathbb{R})$-Higgs bundle $(E,\Phi)$ over $(X, D)$ with parabolic structure $\alpha_i$ at each point $x_i\in D$ is a vector bundle $V$ together with $\Phi=(\beta,\gamma)\in H^0(K(D)\otimes \mathrm{Sym}^2V\oplus K(D)\otimes \mathrm{Sym}^2V^*)$ such that $\Phi|_{x_i}\in U(V_{x_i},\alpha_i)$. Moreover the Higgs bundle is $\alpha$-semistable if for a filtration of $V$ and any strictly increasing sequence $\lambda$ we have $\Phi\in H^0(\mathcal{N}(\mathcal{V},\lambda))$, then $pd(\mathcal{V},\lambda,\alpha)\ge 0$. When the inequality is strict, we say the Higgs bundle is stable.

The pair is $\alpha$-polystable if for a nontrivial $\mathcal{V}$ and $\lambda$ such that $\Phi\in H^0(\mathcal{V},\lambda)$ and such that $pd(\mathcal{V},\lambda,\alpha)=0$ there is an isomorphism of vector bundles
\[\sigma:V\to V_1\oplus V_2/V_1\oplus\ldots \oplus V_k/V_{k-1},\]
such that $V_j=\sigma^{-1}(V_1\oplus V_2/V_1\oplus\ldots \oplus V_j/V_{j-1})$ and such that
\[\beta \in H^0\left(\sum_{\lambda_i+\lambda_j}K(D)\otimes \sigma^{-1}(V_i/V_{i-1})\otimes_S \sigma^{-1}(V_j/V_{j-1})\right)\]
and 
\[\gamma \in H^0\left(\sum_{\lambda_i+\lambda_j}K(D)\otimes \sigma^{*}((V_i/V_{i-1})^*)\otimes_S \sigma^{*}((V_j/V_{j-1})^*)\right).\]
\end{lem}
Since $\mathrm{Sym}(V^*)\subset V^*\otimes V^*$ by sending $v\otimes_S w\to \frac{1}{2}(v\otimes w+ w\otimes v)$, we can write the $\mathrm{Sp}(2n,\mathbb{R})$-bundle as an $\mathrm{SL}(2n,\mathbb{C})$-bundle. Thus, we may write $E=V\oplus {{V}^{\vee }}$ and $\Phi =\left( \begin{matrix}
   0 & \beta   \\
   \gamma  & 0  \\
\end{matrix} \right):E\to E\otimes K\left( D \right)$.  Equipped with the stability condition and since $\beta,\gamma$ are fixed under the change of basis $a\otimes b\mapsto b\otimes a$ in $V\otimes V$ and $V^*\otimes V^*$, we exactly revoke our Definition \ref{502} of a parabolic $\mathrm{Sp}(2n,\mathbb{R})$-Higgs bundle.
\end{exmp}

The proof for the other examples of Section 8 is entirely analogous.

\vspace{2mm}
\textbf{Acknowledgements}.
The authors are happy to express their warmest acknowledgements to Indranil Biswas, Steven Bradlow and Ningchuan Zhang for useful discussions and shared insights. We are also very grateful to Philip Boalch, Roberto Rubio, Laura Schaposnik and anonymous referees for making useful comments, once a first version of this article became available, as well as to the Mathematics Department of the University of Illinois at Urbana-Champaign, where this work was initiated, for providing a productive working environment.

\bigskip

\noindent\small{\textsc{Institut de Recherche Math\'{e}matique Avanc\'{e}e, Universit\'{e} de Strasbourg}\\
7 rue Ren\'{e}-Descartes, 67084 Strasbourg Cedex, France}\\
\emph{E-mail address}:  \texttt{kydonakis@math.unistra.fr}

\bigskip
\noindent\small{\textsc{Department of Mathematics, Sun Yat-Sen University}\\
135 Xingang W Rd, BinJiang Lu, Haizhu Qu, Guangzhou Shi, Guangdong Sheng, China}\\
\emph{E-mail address}:  \texttt{sunh66@mail.sysu.edu.cn}

\bigskip
\noindent\small{\textsc{Department of Mathematics, University of Illinois at Urbana-Champaign}\\
1409 W. Green St, Urbana, IL 61801, USA}\\
\emph{E-mail address}: \texttt{lzhao35@illinois.edu}


\vspace{2mm}

\begin{thebibliography}{99}

\bibitem{BarSch1}
D. Baraglia and L.P. Schaposnik, Cayley and Langlands type correspondences for orthogonal Higgs bundles, Trans. Amer. Math. Soc. \textbf{371} (2019), 7451-7492.
\bibitem{BarSch2}
D. Baraglia and L.P. Schaposnik, Monodromy of rank 2 twisted Hitchin systems and real character varieties, Trans. Amer. Math. Soc. \textbf{370} (2018), no. 8, 5491-5534.
\bibitem{Bhosle}
U. N. Bhosle, Moduli of parabolic $G$-bundles. Bull. Amer. Math. Soc. (N. S.) \textbf{20} (1989), no. 1, 45-48.
\bibitem{BiGaRi}
O. Biquard, O. Garc\'{i}a-Prada and I. Mundet i Riera, Parabolic Higgs bundles and representations of the fundamental group of a punctured surface into a real group, arXiv:1510.04207.
\bibitem{BiGaRu}
O. Biquard, O. Garc\'{i}a-Prada and R. Rubio, Higgs bundles, the Toledo invariant and the Cayley correspondence, J. Topol. \textbf{10} (2017), no. 3, 795-826.
\bibitem{Biswas-Chern}
I. Biswas, Chern classes for parabolic bundles, J. Math. Kyoto Univ. \textbf{37} (1997), no. 4, 597-613.
\bibitem{Biswas1}
I. Biswas, Parabolic ample bundles, Math. Ann. \textbf{307} (1997), 511-529.
\bibitem{Biswas2}
I. Biswas, Parabolic bundles as orbifold bundles,  Duke Math. J. \textbf{88} (1997), no. 2, 305-325.
\bibitem{Biswas3}
I. Biswas, P. Ar\'{e}s-Gastesi and S. Govindarajan, Parabolic Higgs bundles and Teichm\"{u}ller spaces for punctured surfaces, Trans. Amer. Math. Soc. \textbf{349} (1997), no. 4, 1551-1560.
\bibitem{BiHo}
I. Biswas and A. Hogadi, Unitary representations of the fundamental group of orbifolds, Proc. Indian Acad. Sci. (Math. Sci.) \textbf{126} (2016), no. 4, 557-575.
\bibitem{BMW}
I. Biswas, S. Majumder and M. L. Wong, Parabolic Higgs bundles and $\Gamma$-Higgs bundles, J. Aust. Math. Soc. \textbf{95} (2013), no. 3, 315-328.
\bibitem{Boalch}
P. P. Boalch, Riemann-Hilbert for tame complex parahoric connections, Transform. Groups \textbf{16} (2011), no. 1, 27-50.
\bibitem{Boden}
H. U. Boden, Representations of orbifold groups and parabolic bundles, Comment. Math. Helv. \textbf{66} (1991), no. 1, 389-447.
\bibitem{BoYo}
H. U. Boden and K. Yokogawa, Moduli spaces of parabolic Higgs bundles and $K\left( D \right)$ pairs over smooth curves:I, Int. J. Math. \textbf{7} (1996), 573-598.
\bibitem{BrGPGHermitian}
S. B. Bradlow, O. Garc{\'i}a-Prada and P. B. Gothen, Maximal surface group representations in isometry groups of classical Hermitian symmetric spaces, Geom. Dedicata \textbf{122} (2006), 185-213.
\bibitem{FuSt}
M. Furuta and B. Steer, Seifert fibred homology 3-spheres and the Yang-Mills equations on Riemann surfaces with marked points, Adv. in Math. \textbf{96} (1992), no.1, 38-102.
\bibitem{GaGMsymplectic}
O. Garc{\'i}a-Prada, P. B. Gothen and I. Mundet i Riera, Higgs bundles and surface group representations in the real symplectic group, J. Topol. \textbf{6} (2013), no. 1, 64-118.
\bibitem{GaGoRi}
O. Garc\'{i}a-Prada, P. B. Gothen and I. Mundet i Riera, The Hitchin-Kobayashi correspondence, Higgs
pairs and surface group representations, arXiv: 0909.4487.
\bibitem{GaGoMu}
O. Garc\'{i}a-Prada, P. B. Gothen and V. Mu\~{n}oz, Betti numbers of the moduli space of rank 3 parabolic Higgs bundles, Mem. Amer. Math. Soc. \textbf{187} (2007), no. 879, viii+80 pp.
\bibitem{GaLoMu}
O. Garc\'{i}a-Prada, M. Logares and V. Mu\~{n}oz, Moduli spaces of parabolic $\mathrm{U}\left( p,q \right)$-Higgs bundles, Q. J. Math. \textbf{60} (2009), no. 2, 183-233.
\bibitem{Gothen}
P. B. Gothen, Components of spaces of representations and stable triples, Topology \textbf{40} (2001), no. 4, 823-850.
\bibitem{Hat}
A. Hatcher, Algebraic Topology, Cambridge University Press (2002).
\bibitem{Hit92}
N. J. Hitchin, Lie groups and Teichm{\"u}ller space, Topology \textbf{31} (1992), no. 3, 449-473.
\bibitem{HitSch}
N. J. Hitchin and L. P. Schaposnik, Nonabelianization of Higgs bundles, J. Differ. Geom. \textbf{97} (2014), no. 1, 79-89.
\bibitem{Iyer-Simpson}
J. N. N. Iyer and C. T. Simpson, A relation between the parabolic Chern characters of the de Rham bundles, Math. Ann. \textbf{338} (2007) 347-383.
\bibitem{Kawa}
Y. Kawamata, Characterization of the abelian varieties, Compos. Math. \textbf{43} (1981), 253-276.
\bibitem{Kawasaki}
T. Kawasaki, The Riemann-Roch theorem for complex $V$-manifolds, Osaka J. Math. \textbf{16} (1979), 151-159.
\bibitem{Konn}
H. Konno, Construction of the moduli space of stable parabolic Higgs bundles on a Riemann surface, J. Math. Soc. Japan \textbf{45} (1993), no. 2, 253-276.
\bibitem{Kydon}
G. Kydonakis, Model Higgs bundles in exceptional components of the $\mathrm{Sp(4}\mathrm{,}\mathbb{R}\mathrm{)}$-character variety, arXiv: 1805.10497.
\bibitem{KSZ2}
G. Kydonakis, H. Sun and L. Zhao, The Beauville-Narasimhan-Ramanan correspondence for twisted Higgs $V$-bundles and components of parabolic $\mathrm{Sp}(2n,\mathbb{R})$-Higgs moduli, arXiv: 1901.09148.  
\bibitem{KSZ3}
G. Kydonakis, H. Sun and L. Zhao, Monodromy of rank 2 parabolic Hitchin systems, arXiv: 1906. 03740.
\bibitem{Loga}
M. Logares, Parabolic $\mathrm{U}\left( p,q \right)$-Higgs bundles, Ph. D. thesis, Universidad Aut{\'o}noma de Madrid (2006).
\bibitem{Loga2}
M. Logares, Betti numbers of parabolic $\mathrm{U}(2,1)$-Higgs bundles moduli spaces, Geom. Dedicata \textbf{123} (2007), no. 1, 187-200.
\bibitem{Meht}
V. B. Mehta and C. S. Seshadri, Moduli of vector bundles on curves with parabolic structures, Math. Ann. \textbf{248} (1980), no. 3, 205-239.
\bibitem{Rier2}
I. Mundet i Riera, Maximal weights in K{\"a}hler geometry: Flag manifolds and Tits distance, Vector bundles and Complex Geometry: Conference on vector bundles in honor of S. Ramanan on the occassion of his 70th birthday, American Mathematical Society, 2010.
\bibitem{Rier1}
I. Mundet i Riera, Parabolic Higgs bundles for real reductive Lie groups: A very basic introduction, Geometry and Physics vol. II: A Festschrift in honour of Nigel Hitchin, p. 653-680, Cambridge University Press, 2018.
\bibitem{Rier}
I. Mundet i Riera, Parabolic vector bundles and equivariant vector bundles, Internat. J. Math. \textbf{13} (2002), no. 9, 907-957.
\bibitem{NaSt}
B. Nasatyr and B. Steer, Orbifold Riemann surfaces and the Yang-Mills-Higgs equations, Ann. Scuola Norm. Sup. Pisa Cl. Sci. \textbf{4} (1995), no. 4, 595-643.
\bibitem{Rama}
A. Ramanathan, Stable principal bundles on a compact Riemann surface, Math. Ann. \textbf{213} (1975), 129-152.
\bibitem{Schap}
L. P. Schaposnik, Monodromy of the $\mathrm{SL}_{2}$ Hitchin fibration, Internat. J. Math. \textbf{24} (2013), no. 2, 1350013, 21 pp.
\bibitem{Scott}
P. Scott, The geometries of 3-manifolds, Bull. London Math. Soc. \textbf{15} (1983), no. 5, 401-487.
\bibitem{Sesh}
C. S. Seshadri, Moduli of vector bundles on curves with parabolic structures, Bul. of the AMS \textbf{83} (1977), no. 1, 124-126.
\bibitem{Simp}
C. T. Simpson, Harmonic bundles on noncompact curves, J. Amer. Math. Soc. \textbf{3} (1990), no. 3, 713-770.
\bibitem{Stru}
T. Strubel, Fenchel-Nielsen coordinates for maximal representations, Geom. Dedicata \textbf{176} (2015), 45-86.
\bibitem{Thad}
M. Thaddeus, Variation of moduli of parabolic Higgs bundles. J. Reine Angew. Math. \textbf{547} (2002), 1-14.
\bibitem{Yoko1}
K. Yokogawa, Compactification of moduli of parabolic sheaves and moduli of parabolic Higgs sheaves, J. Math. Kyoto Univ. \textbf{33} (1993), 451-504.
\bibitem{Yoko2}
K. Yokogawa, Infinitesimal deformation of parabolic Higgs sheaves, Int. J. Math. \textbf{6} (1995), 125-148.
\end{thebibliography}
\end{document}